\documentclass[11p]{article}
\usepackage[a4paper, total={6in, 9in}]{geometry}
\usepackage{amsmath}
\usepackage{tikz}
\makeatletter
\newcommand{\leqnomode}{\tagsleft@true\let\veqno\@@leqno}
\newcommand{\reqnomode}{\tagsleft@false\let\veqno\@@eqno}
\makeatother
\usepackage{hyperref}
\usepackage{amssymb}
\usepackage{amsthm}
\usepackage{float}
\usepackage{bbm}
\usepackage{cases}
\usepackage{color}
\usepackage{multirow}
\usepackage{booktabs}
\usepackage{subcaption}
\usepackage{array}
\usepackage{enumerate}
\usepackage{algpseudocode}
\usepackage{algorithm}
\usepackage{subfloat}
\newtheorem{theorem}{Theorem}
\newtheorem{lemma}{Lemma}

\newtheorem{proposition}{Proposition}
\theoremstyle{definition}
\theoremstyle{definition}
\newtheorem{remark}{Remark}
\usepackage[toc,page]{appendix}
\usepackage{ulem}

\newtheorem{definition}{Definition}
\newtheorem{assumption}{Assumption}
\allowdisplaybreaks

\newcommand{\tr}{\text{tr}}
\newcommand{\Diag}{\text{Diag}}
\newcommand{\diag}{\text{diag}}

\newcommand{\Col}{\text{Col}}
\newcommand{\Nul}{\text{Nul}}
\newcommand{\Conv}{\text{Conv}}

\newcommand{\rank}{\text{rank}}
\newcommand{\Span}{\text{Span}}
\newcommand{\inner}{\text{inner}}

\newcommand{\arrow}{\text{arrow}}

\newcommand{\aff}{\text{aff}}
\newcommand{\supp}{\text{supp}}
\DeclareMathOperator*{\argmax}{arg\,max}

\begin{document}

\title{ SDP-based bounds for the Quadratic Cycle Cover Problem via cutting plane augmented Lagrangian methods and reinforcement learning}
\author{Frank de Meijer \thanks{CentER, Department of Econometrics and OR, Tilburg University, The Netherlands, {\tt f.j.j.demeijer@uvt.nl}}
	\and {Renata Sotirov}  \thanks{CentER, Department of Econometrics and OR, Tilburg University, The Netherlands, {\tt r.sotirov@uvt.nl}}}

\date{}

\maketitle

\begin{abstract}
We study the Quadratic Cycle Cover Problem (\textsc{QCCP}), which aims to find a node-disjoint cycle cover in a directed graph with minimum interaction cost between successive arcs.
We derive several semidefinite programming (SDP) relaxations and use facial reduction to make these strictly feasible. We investigate a nontrivial relationship between the transformation matrix used in the reduction and the structure of the graph, which is exploited in an efficient algorithm that constructs this matrix for any instance of the problem. To solve our relaxations, we propose an algorithm that incorporates an augmented Lagrangian method into a cutting plane framework by utilizing Dykstra's projection algorithm.
Our algorithm is suitable for solving  SDP relaxations with a large number of cutting planes.
Computational results show that our SDP bounds and our efficient cutting plane algorithm outperform other \textsc{QCCP} bounding approaches from the literature.
Finally, we provide several SDP-based upper bounding techniques, among which a sequential Q-learning method that exploits a solution of our SDP relaxation within a reinforcement learning environment.
\end{abstract}


\section{Introduction}
A disjoint cycle cover in a graph is a set of node-disjoint cycles such that every node is covered by exactly one cycle. The cycle cover problem (\textsc{CCP}) is the problem of finding a disjoint cycle cover such that the total arc weight is minimized.
In this paper we focus on its quadratic version, which is known as the quadratic cycle cover problem (\textsc{QCCP}).
The \textsc{QCCP}  is the problem of finding a disjoint cycle cover in a graph such that the total sum of interaction costs between consecutive arcs is minimized.
Although the problem can be defined for both directed and undirected graphs, we focus here on the asymmetric version which is defined on directed graphs.

The \textsc{QCCP} is introduced by J\"ager and Molitor \cite{Jager}. Fischer et al.\ \cite{FischerEtAl} show that the problem is $\mathcal{NP}$-hard. This result is later on strengthened by De Meijer and Sotirov \cite{DeMeijerSotirov}, who prove that the \textsc{QCCP} is strongly $\mathcal{NP}$-hard and not approximable within any constant factor.

There exist several  special cases of the \textsc{QCCP} with respect to the objective function, for example, the angular metric cycle cover problem (\textsc{Angle-CCP}) \cite{Aggarwal} and the minimum reload cost cycle cover (\textsc{MinRC3}) problem \cite{Galbiati}.
In the \textsc{Angle-CCP}  the quadratic costs represent the change of the direction induced by two consecutive arcs.
The \textsc{MinRC3} problem is the problem of finding a minimum disjoint cycle cover in an arc-colored graph under the reload cost model.
These problems have applications in various fields, such as robotics \cite{Aggarwal}, cargo and energy distribution networks \cite{WirthSteffan}.
We refer the interested reader to \cite{DeMeijerSotirov} for a more detailed overview of these variants of the  \textsc{QCCP} and their applications.
The \textsc{QCCP} may also be  seen as a generalization of the minimum-turn cycle cover problem, which  belongs to the class of covering-tour problems  introduced in Arkin et al.\ \cite{ArkinEtAl}.
Covering tour problems play an important role in manufacturing, automatic inspection,  spray painting operations, etc.
For a detailed overview of cycle cover problems with turn costs and their applications, we refer the reader to \cite{FeketeKrupke}.

The importance of the  \textsc{QCCP} is also due to its close connection to the quadratic traveling salesman problem (\textsc{QTSP}).
The goal of the \textsc{QTSP} is to find a Hamiltonian cycle in a graph minimizing the total quadratic costs between consecutive arcs.
 After removing the subtour elimination constraints, the \textsc{QTSP} boils down to the \textsc{QCCP}. Not surprisingly, the \textsc{QTSP} was introduced simultaneously with the \textsc{QCCP} in \cite{Jager}.
 The \textsc{QTSP} is proven to be $\mathcal{NP}$-hard and not approximable within any constant factor \cite{Jager} and is generally accepted to be one of the hardest combinatorial optimization problems nowadays. The \textsc{QTSP} has applications in robotics, bioinformatics and telecommunication, see e.g., Fischer et al.\ \cite{AandFFischer}.
 The \textsc{QCCP} plays an important role in obtaining strong lower and upper bounds for the \textsc{QTSP} \cite{AandFFischer, Jager, Stanek}.

Various papers  study theoretical aspects as well as  solution approaches  for the \textsc{QCCP} and its variants.
Fischer \cite{Fischer} studies the polyhedral structure of the quadratic cycle cover polytope and provides  several inequalities that define the facets of this polytope.
B\"uy\"uk\c{c}olak et al.\  \cite{Buyukcolak}  consider  the \textsc{MinRC3} on complete graphs with an equitable or nearly equitable 2-edge coloring
and derive a polynomial time algorithm that constructs a monochromatic cycle cover.
 J\"ager and Molitor \cite{Jager} use approximated solutions of the QCCP  as lower bounds in a branch-and-bound algorithm for the QTSP.
Galbiati et al.\ \cite{Galbiati} exploit a column generation approach to compute lower bounds for the  \textsc{MinRC3} problem.
 Stan\v ek et al.\ \cite{Stanek} use the QCCP combined with a rounding procedure to construct heuristics for the QTSP.
 Several local search algorithms for the  \textsc{MinRC3} problem are given in \cite{Galbiati}.
 Approximation algorithms for the \textsc{QCCP} and its variants are studied in \cite{Aggarwal,ArkinEtAl,FeketeKrupke}.

 The linearization problem of the \textsc{QCCP} is considered by De Meijer and Sotirov \cite{DeMeijerSotirov}. Several sufficient conditions for a \textsc{QCCP} instance to be linearizable are provided, which are used to construct strong linearization based bounds for any instance of the problem.

\subsection*{Main results and outline }

The aim of this paper is to construct efficient lower and upper bounding approaches for the \textsc{QCCP} based on semidefinite programming.  
To achieve this goal we introduce several methods  
that can be extended to a range of other optimization problems.
In our work we combine a  wide variety of different  techniques including facial reduction, projection methods, randomized algorithms and reinforcement learning.

First, we derive three SDP relaxations for the  \textsc{QCCP} with increasing complexity. 
Our strongest SDP relaxation contains nonnegativity constraints and an additional subset of the facet-defining inequalities of the Boolean Quadric Polytope (BQP),
which make it a powerful yet very difficult to solve relaxation.
As a first step in the development of our algorithmic approaches for computing \textsc{QCCP} lower bounds,
 we study the geometry of the feasible sets of our relaxations.
 We prove that the relaxations are not Slater feasible, and show how to perform facial reduction to project the feasible sets onto lower dimensional spaces.
The transformation matrix needed for this projection is  graph-specific.
Therefore we propose a polynomial time algorithm
based on the bipartite representation of the underlying graph that provides a sparse transformation matrix.

To solve our SDP relaxation with nonnegativity constraints, we study the following two variants of the alternating direction augmented Lagrangian method;
the (original) Alternating Direction Method of Multipliers (\textsc{ADMM}) and the Peaceman-Rachford splitting method (\textsc{PRSM}) that is also known as the symmetric \textsc{ADMM}.
Although  the \textsc{ADMM} is tested on SDP relaxations of various optimization problems,
the \textsc{PRSM} with larger stepsize was not implemented up to date for SDP relaxations.
Our results show that the \textsc{PRSM} outperforms the classical \textsc{ADMM} for the relaxation with nonnegativity constraints.
Therefore we take the \textsc{PRSM} as the backbone of our new approach.

It is  well-known that current SDP solvers have difficulties solving relaxations including the facet-defining inequalities of the BQP. To solve our strongest relaxation including these cuts, we present an advanced cutting plane method that extends on the \textsc{PRSM}: a cutting plane augmented Lagrangian method (\textsc{CP-ALM}). 
The \textsc{CP-ALM} exploits the well-known Dykstra projection algorithm to deal with the BQP cuts.
We (partially) parallelize Dykstra's cyclic algorithm by clustering the set of BQP inequalities into subsets of nonoverlapping cuts. We present several other ingredients that improve the efficiency of the algorithm.
The \textsc{CP-ALM} also exploits warm starts each time new violated cuts are added.
Although it might seem that our algorithm is problem specific, all ingredients described in this paper can be easily extended for solving other optimization  problems.

Finally, we derive several upper bounding approaches that exploit the output matrices from the \textsc{CP-ALM}.
 Let us list the most prominent ones.
In our randomized undersampling algorithm we sample a partial solution and deterministically extend it to a full cycle cover.
In randomized oversampling we iteratively draw a pair of successive arcs according to a distribution related to the SDP solution, until we obtain a cycle cover.
Our most sophisticated rounding approach is based on a distributed reinforcement learning technique, i.e., Q-learning.
In particular, we let artificial agents learn  how to find cycle covers by exploiting the SDP solution matrix such that the expected total reward is maximized.
The latter approach provides the best upper bounds among all presented ones. We expect that these rounding approaches can be successfully extended to relaxations of other optimization problems.
Let us emphasize that it is challenging to find good feasible solutions for the \textsc{QCCP}, especially for large instances,
since the considered graphs are not necessarily complete.

We provide extensive numerical tests on data sets used for the \textsc{QCCP} as well as data sets for the \textsc{QTSP}.
Our bounds significantly outperform other bounds from the literature.\\

\noindent The paper is structured as follows.
In Section \ref{SectionQCCP}, we formally introduce  the \textsc{QCCP} and study its associated directed 2-factor polytope.
In Section \ref{SectionSDPrelaxations}, we construct several SDP relaxations for the \textsc{QCCP} of increasing complexity.
The Slater feasibility of the SDP relaxations is the topic of Section \ref{SubsectionFacialReduction}. Since transformation matrices used for projecting onto the minimal face  are graph-specific, we provide a polynomial time algorithm for computing their sparse expressions
in Section \ref{sec:polyMatrix}. In Section \ref{SectionSolving}, we propose a new algorithm  for solving the SDP relaxations that is based on a combination of the \textsc{PRSM},
Dykstra's projection algorithm and a cutting plane method. Several upper bounding approaches are discussed in Section \ref{SectionUpper}. Section \ref{SectionNumerics} provides an extensive numerical study of all introduced methods.

\subsection*{Notation}
A directed graph $G = (N,A)$ is given by a node set $N$ and an arc set $A \subseteq N \times N$, where $n = |N|$ and $m = |A|$. For all $i \in N$, let $\delta^+(i)$ and $\delta^-(i)$ denote the set of arcs leaving and entering node $i$, respectively. For all $S, T \subseteq N$, let $\delta^+(S,T)$ (resp.\ $\delta^-(S,T)$) denote the set of arcs going from $S$ to $T$ (resp.\ $T$ to $S$). The starting node of an arc $e \in A$ is denoted by $e^+$. Similarly, we denote by $e^-$ the ending node of $e$.

Let $\bold{0}_n$ and $\bold{1}_n$ be the $n \times 1$ vector of zeros and the $n \times 1$ vector of ones, respectively. Moreover, we denote the $i$-th elementary vector by $\bold{e}_i$. Let $J_n$ and $I_n$ be the $n \times n$ matrix of all ones and the $n \times n$ identity matrix, respectively. In case the order of these vectors or matrices is clear, we omit the subscript to simplify notation.

In this paper we frequently work with extended matrices of the form $\begin{pmatrix}
1 & x^\top \\ x & X
\end{pmatrix}$ for some $x \in \mathbb{R}^m$ and $X \in \mathbb{R}^{m \times m}$. We index the top row and top column of such extended matrix as row zero and column zero, respectively. Accordingly,
 we denote by $\bold{e}_0$ a vector that has a one on the first position and all other elements zero in the context of extended matrices.

For any matrix $M \in \mathbb{R}^{m \times n}$, let $\Col(M)$ be the linear space spanned by the columns of $M$. The null space of $M$ is denoted by $\Nul(M)$. For all $M, N \in \mathbb{R}^{m \times n}$, the Hadamard product $M \circ N$ equals the entrywise product of $M$ and $N$, i.e., $(M \circ N)_{ij} = M_{ij}N_{ij}$. Moreover, for all square matrices $M$, the operator $\diag : \mathbb{R}^{n \times n} \rightarrow \mathbb{R}^n$ maps a matrix to a vector consisting of its diagonal elements. Its adjoint operator is given by $\Diag : \mathbb{R}^n \rightarrow \mathbb{R}^{n \times n}$. The trace of a square matrix $M$ is given by $\tr(M)$.

Let $\mathcal{S}^m$ denote the set of all $m \times m$ real symmetric matrices. We denote by $M \succeq 0$ that the matrix $M$ is positive semidefinite and let $\mathcal{S}^m_+$ be the set of all positive semidefinite matrices of order $m$, i.e., $\mathcal{S}^m_+ := \{ M \in \mathcal{S}^m \, : \, \, M \succeq 0 \}$. Let $\langle \cdot , \cdot \rangle$ denote the trace inner product. That is, for any $M,N \in \mathbb{R}^{m \times m}$, we define $\langle M, N \rangle := \tr(M^\top N) = \sum_{i = 1}^m\sum_{j = 1}^m M_{ij}N_{ij}$. Its associated norm is the Frobenius norm, denoted by $|| M ||_F := \sqrt{\tr(M^\top M)}$.

\section{The Quadratic Cycle Cover Problem} \label{SectionQCCP}
In this section we formally introduce the asymmetric version of the quadratic cycle cover problem. Moreover, we introduce the directed 2-factor polytope and consider some of its properties.\\ \\
The quadratic cycle cover problem (\textsc{QCCP}) is the problem of finding a set of node-disjoint cycles covering all the nodes such that the sum of interaction costs between successive arcs is minimized. Since we assume that all cycle
covers in this paper are disjoint, we use the term cycle cover to denote this concept in the sequel. An instance of the \textsc{QCCP} is specified by the pair $(G,Q)$, where $G = (N,A)$ is a simple directed graph with $n = |N|$ nodes and $m = |A|$ arcs and $Q = (q_{ef}) \in \mathbb{R}^{m \times m}$ is a quadratic cost matrix. We assume that the entries of $Q$ are such that $q_{ef} = 0$ if arc $f$ is not a successor of arc $e$, i.e., if $f \notin \delta^+(e^-)$. \\ \\
Let $x \in \{0,1\}^m$ represent the characteristic vector of a cycle cover. That is, $x_e = 1$ if arc $e$ belongs to the cycle cover and $x_e = 0$ otherwise. Then the \textsc{QCCP} can be formulated as:
\begin{align} \label{ProblemDef}
 OPT(Q) :=  \min  \left\{  x^\top Qx: ~    x \in P \right\},
\end{align}
where $P$ denotes the set of all cycle covers in $G$, i.e.,
\begin{align}\label{setP}
P := \left\{ x \in \{0,1\}^m \, : \, \sum_{e \in \delta^+(i)} x_e = \sum_{e \in \delta^-(i)}x_e = 1 \quad \forall i \in N \right\}.
\end{align}
In the literature, a cycle cover in a directed graph is also called a directed 2-factor.
For the existence of a directed 2-factor in a graph, see e.g., Chiba and Yamashita \cite{Chiba}. The \textsc{QCCP} is shown to be $\mathcal{NP}$-hard in the strong sense and not approximable within any constant factor \cite{FischerEtAl, DeMeijerSotirov}.

The linear problem corresponding to the \textsc{QCCP} is called the cycle cover problem (\textsc{CCP}). Given a linear arc-weight function, the \textsc{CCP} asks for a minimum weight cycle cover in $G$. The \textsc{CCP} reduces to the well-known linear assignment problem, see e.g., \cite{Burkard}, and is therefore polynomial time solvable.\\

\noindent Let $\Conv(P)$ be the convex hull of all characteristic vectors corresponding to  directed 2-factors in $G$. We call this set the directed 2-factor polytope. Let $U \in \mathbb{R}^{n \times m}$ and $V \in \mathbb{R}^{n \times m}$ be defined as
\begin{align*}
U_{i,e}  := \begin{cases} 1 & \text{if arc $e$ starts at node $i$} \\
0 & \text{otherwise,}
\end{cases} \hspace{1cm} V_{i,e}  := \begin{cases} 1 & \text{if arc $e$ ends at node $i$} \\
0 & \text{otherwise.}
\end{cases}
\end{align*}
Additionally, let $u_i^\top$ and $v_i^\top$ denote the $i$-th row of $U$ and $V$, respectively.
Thus, $P=\{x \in \{0,1\}^m \, : \, \, [U^\top, V^\top]^\top x = \bold{1}_{2n}\}$. It follows from the total unimodularity of $[U^\top, V^\top]^\top$ that the directed 2-factor polytope can be written explicitly as:
\begin{align} \label{ExplicitConvP}
\Conv(P) = \left\{ x \in \mathbb{R}^m \, : \, \, x \geq \bold{0}_m \, , \, \, \begin{bmatrix}
U \\ V
\end{bmatrix} x = \bold{1}_{2n} \right\}.
\end{align}
Observe that the arcs that are never used in a cycle cover are irrelevant for the \textsc{QCCP}. We define the set $\mathcal{J}$ consisting of all arcs with this property, i.e.,
\begin{align*}
\mathcal{J} := \{f \in A \, : \, \, x_f = 0 \text{ for all  } x \in P \}.
\end{align*}
The elements in $\mathcal{J}$ can be obtained in polynomial time by solving  for each $f \in A$  the following \textsc{CCP}:
\begin{align*}
z_f := \max \{\bold{e}_f^\top x \, : \, \, x \in P \}.
\end{align*}
The set $\mathcal{J}$ consists of all arcs $f \in A$ for which $z_f = 0$.
Without loss of generality, we can remove the arcs that are in $\mathcal{J}$ from the given instance to simplify the problem.
This leads to the following assumption that applies to the rest of this paper.

\begin{assumption} \label{AssumptionJempty} There exists at least one cycle cover in $G$, i.e., $P \neq \emptyset$. Moreover, the set $\mathcal{J}$ is empty.
\end{assumption}
\noindent We end this section by considering the dimension of the directed 2-factor polytope. We define
\begin{align} \label{alpha}
\alpha := \rank \left( \begin{bmatrix}
U \\ V
\end{bmatrix} \right).
\end{align}
In Section \ref{sec:polyMatrix}, we derive the value of $\alpha$ in terms of the graph. For now, we note that $n \leq \alpha \leq 2n - 1$, provided that Assumption \ref{AssumptionJempty} holds.

It follows from the rank-nullity theorem that $\dim(\Nul([U^\top , V^\top]^\top)) = m - \alpha$. Let us prove the following lemma.
\begin{lemma} \label{LemmaDimensionConvP}
Under Assumption \ref{AssumptionJempty}, the dimension of the directed 2-factor polytope equals $m - \alpha$.
\end{lemma}
\begin{proof}
It follows from (\ref{ExplicitConvP}) that
\begin{align*}
\Conv(P) = \{x \in \mathbb{R}^m \, : \, \, [U^\top, V^\top]^\top x = \bold{1}_{2n} \} \cap \mathbb{R}^m_+.
\end{align*}
Obviously, the set $\mathbb{R}^m_+$ is full-dimensional, whereas the dimension of $\{x \in \mathbb{R}^m \, : \, \, [U^\top, V^\top]^\top x = \bold{1}_{2n} \}$ equals $\dim(\Nul([U^\top, V^\top ]^\top)) = m - \alpha$. Hence, we have $\dim(\Conv(P)) \leq m - \alpha$, where strict inequality holds only if $\Conv(P)$ is fully contained in one of the facets of $\mathbb{R}^m_+$. Now assume for the sake of contradiction that there exists some arc $e$ such that $\Conv(P) \subseteq \{x \in \mathbb{R}^m_+ \, : \, \, x_e = 0\}$. Consequently, we must have $P \subseteq \{x \in \mathbb{R}^m_+ \, : \, \, x_e = 0\}$, which implies that $e \in \mathcal{J}$. This contradicts Assumption \ref{AssumptionJempty}. We conclude that $\dim(\Conv(P)) = m - \alpha$.
\end{proof}

\section{SDP relaxations for the \textsc{QCCP}} \label{SectionSDPrelaxations}
In this section we focus on constructing several semidefinite programming relaxations for the \textsc{QCCP}. These relaxations increase in strength and complexity.\\ \\
Using the trace inner product, the objective function of (\ref{ProblemDef}) can be rewritten as $x^\top Qx = \langle Q, xx^\top \rangle = \langle Q, X \rangle$, where we replace $xx^\top$ by a matrix variable $X \in \mathcal{S}^m$. We now relax the equality $X - xx^\top = 0$ by replacing it by the SDP constraint $X - xx^\top \succeq 0$. It follows from the Schur complement that we can equivalently write $\begin{pmatrix}
1 & x^\top \\ x & X
\end{pmatrix} \succeq 0$. Moreover, since $x \in P$ is a binary vector, we have $\diag(X) = x$.
This leads to the following basic feasible set for an SDP relaxation of the \textsc{QCCP}:
\begin{equation}\label{feasibleBasic}
{\mathcal F}_{\rm basic} :=
\left \{
\begin{pmatrix}
1 & x^\top \\
x & X
\end{pmatrix}\in \mathcal{S}^{m+1} : ~
u_i^\top x = v_i^\top x = 1 ~(\forall i \in N),  ~x= \text{diag}(X),~
\begin{pmatrix}
1 & x^\top \\
x & X
\end{pmatrix} \succeq 0
 \right \}.
\end{equation}
We show below how to strengthen  the feasible set \eqref{feasibleBasic} by adding valid constraints.

Since each cycle cover consists of $n$ arcs, we know that $\bold{1}_m^\top x = n$ for all $x \in P$.
This can be written equivalently as $\tr(X) = n$, which we refer  to as the trace constraint. Moreover, since $xx^\top$  is replaced  by  $X$, the constraint $\langle J, X \rangle = n^2$, which we call the all-ones constraint, is also valid.

One can also add to ${\mathcal F}_{\rm basic}$ the so-called squared linear constraints. These constraints result from taking the product of the linear constraints
 $u_i^\top x = 1$ and $u_j^\top x = 1$ for all $i, j \in N$, which yield $1 = (u_i^\top x)(x^\top u_j) = \langle u_iu_j^\top, xx^\top \rangle$. Hence, the constraint $\langle u_iu_j^\top , X \rangle = 1$ is valid for ${\mathcal F}_{\rm basic}$. The same can be done by taking the products of the linear constraints $v_i^\top x = 1$ for all $i \in N$, etc.
In total, we distinguish three types of squared linear constraints that are summarized in Table~\ref{Table4Types}.

\begin{table}[H]
\begin{tabular}{@{}cc@{}}
\toprule
Type of squared linear constraint & Constraints on $X$                                                                          \\ \midrule
Type I                            & $\langle u_iu_i^\top, X \rangle = 1$ and $\langle v_iv_i^\top, X \rangle = 1$ for all $i \in N$; \\[1.ex]
Type II                           &                                                                                           $\langle u_iu_j^\top, X \rangle = 1$ and $\langle v_iv_j^\top, X\rangle = 1$ for all $i,j \in N, i \neq j$; \\[1.ex]
Type III & $\langle u_iv_j^\top, X \rangle = 1$ for all $i, j \in N$. \\ \bottomrule
\end{tabular}
\caption{Three types of valid squared  linear constraints for ${\mathcal F}_{\rm basic}$. \label{Table4Types}}
\end{table}

\noindent
We show below how the above mentioned valid constraints relate.
An interesting  result is that the squared linear constraints of Type II and III turn out to be redundant
when the Type I  constraints and some other constraints are added to \eqref{feasibleBasic}.

\begin{proposition} \label{TheoremSquaredLinearConstraints}
Let $x \in \mathbb{R}^m$ and $X \in \mathcal{S}^m$ be such that $\begin{pmatrix}
1 & x^\top \\ x & X
\end{pmatrix} \succeq 0$, $\diag(X) = x$, $\tr(X) = n$ and $\langle J, X \rangle = n^2$. If $\langle u_iu_i^\top, X \rangle = \langle v_iv_i^\top, X \rangle = 1$ for all $i \in N$, then
\begin{enumerate}[(i)]
\item the squared linear constraints of Type II and III are redundant;
\item the linear constraints $u_i^\top x = u_j^\top x = 1$ for all $i, j \in N$ are redundant.
\end{enumerate}
\end{proposition}
\begin{proof}
$(i) \quad$ Let $i,j \in N$ with $i \neq j$, then we have,
\begin{align*}
\langle (u_i - u_j)(u_i - u_j)^\top, X \rangle & = \langle u_iu_i^\top + u_ju_j^\top - 2u_iu_j^\top, X \rangle \\
& = \langle u_iu_i^\top, X \rangle + \langle u_ju_j^\top, X \rangle - 2 \langle u_iu_j^\top, X \rangle \\
& = 2 - 2 \langle u_iu_j^\top, X \rangle  \geq 0,
\end{align*}
since $\langle u_iu_i^\top, X \rangle = \langle u_ju_j^\top, X \rangle = 1$ and $X \succeq 0$. From this it follows that $\langle u_iu_j^\top, X \rangle \leq 1$.

Conversely, as all arcs have exactly one starting node, we have $\bold{1}_m = \sum_{i \in N}u_i$. Using this, we can rewrite the matrix $J$ as $J = \left( \sum_{i \in N}u_i \right)\left( \sum_{i \in N}u_i \right)^\top$ and the constraint $\langle J, X \rangle = n^2$ as follows:
\begin{align*}
n^2 = \langle J, X \rangle & = \left\langle \left( \sum_{i \in N}u_i \right)\left( \sum_{i \in N}u_i \right)^\top, \, \,  X \right\rangle = \sum_{i \in N}\sum_{j \in N} \langle u_iu_j^\top, X \rangle.
\end{align*}
The right-hand side expression is a sum of $n^2$ elements for which $\langle u_iu_j^\top, X \rangle \leq 1$ for all $i,j \in N$. Since the sum has to be equal to $n^2$, it follows that $\langle u_iu_j^\top, X \rangle = 1$ for all $i,j \in N$, $i \neq j$. The other equalities can be proven in a similar fashion. \\ \\
$(ii) \quad$ Let $Y := X - xx^\top$. By the Schur complement, we know that $Y \succeq 0$. Now,
\begin{align*}
 1 = \langle u_iu_i^\top, X \rangle = \langle u_iu_i^\top, Y \rangle + \langle u_iu_i^\top, xx^\top \rangle = \langle u_iu_i^\top, Y \rangle + (u_i^\top x)^2.
\end{align*}
Since $\langle u_iu_i^\top, Y \rangle \geq 0$, it follows that $(u_i^\top x)^2 \leq 1$ and, consequently, $u_i^\top x \leq 1$ for all $i \in N$.

Now we rewrite the trace constraint $\bold{1}_m^\top x = n$ as
\begin{align*}
\bold{1}_m^\top x = \left( \sum_{i \in N}u_i\right)^\top x = \sum_{i \in N}u_i^\top x = n.
\end{align*}
Since each term $u_i^\top x$ is bounded by 1, equality is established only when $u_i^\top x = 1$ for all $i \in N$. In a similar way one can show that $v_i^\top x = 1$ for all $i \in N$.
\end{proof}

\noindent Observe that there exist $2n$ constraints of Type I.
We show how to merge these constraints to obtain a more compact formulation. To this end, we define the matrices $\tilde{U}, \tilde{V} \in \mathbb{R}^{m \times m}$ as follows:
\begin{align*}
\tilde{U} := \sum_{i \in N} \begin{pmatrix}
-1 \\ u_i
\end{pmatrix} \begin{pmatrix}
-1 \\ u_i
\end{pmatrix}^\top \quad \text{and} \quad \tilde{V} := \sum_{i \in N} \begin{pmatrix}
-1 \\ v_i
\end{pmatrix} \begin{pmatrix}
-1 \\ v_i
\end{pmatrix}^\top.
\end{align*}
We establish the following result.

\begin{proposition} \label{LemmaCompactForm} Let $x \in \mathbb{R}^m$ and $X \in \mathcal{S}^m$ be such that $\begin{pmatrix}
1 & x^\top \\ x & X
\end{pmatrix} \succeq 0$ and $\diag(X) = x$. Then the following statements are equivalent:
\begin{enumerate}[(i)]
\item $\tr(X) = n$ and $\langle u_iu_i^\top, X \rangle = \langle v_iv_i^\top , X \rangle = 1$ for all $i \in N$;
\item $\left \langle \tilde{U}, \begin{pmatrix}
1 & x^\top \\ x & X
\end{pmatrix} \right \rangle = \left \langle \tilde{V}, \begin{pmatrix}
1 & x^\top \\ x & X
\end{pmatrix} \right \rangle = 0.$
\end{enumerate}
\end{proposition}
\begin{proof}
It is not difficult to see that $(i) \Rightarrow (ii)$. We now show the reverse statement. We have
\begin{align*}
\left \langle \tilde{U}, \begin{pmatrix}
1 & x^\top \\ x & X
\end{pmatrix} \right \rangle = \left \langle \sum_{i \in N} \begin{pmatrix}
-1 \\ u_i
\end{pmatrix} \begin{pmatrix}
-1 \\ u_i
\end{pmatrix}^\top , \begin{pmatrix}
1 & x^\top \\ x & X
\end{pmatrix} \right \rangle = \sum_{i \in N} (u_i^\top X u_i - 2 u_i^\top x + 1) = 0.
\end{align*}
Since $\begin{pmatrix}
-1 \\ u_i
\end{pmatrix}\begin{pmatrix}
-1 \\ u_i
\end{pmatrix}^\top \succeq 0$, it follows that $u_i^\top Xu_i - 2u_i^\top x + 1 \geq 0$ for all $i \in N$. Combining this with the equality above, we conclude that $u_i^\top Xu_i - 2u_i^\top x + 1 = 0$ for all $i \in N$.

Now define $Y := X - xx^\top \succeq 0$. Then $u_i^\top Xu_i - 2u_i^\top x + 1 = 0$ can be rewritten as
\begin{align*}
u_i^\top (Y + xx^\top )u_i - 2u_i^\top x + 1 = 0 \quad \Leftrightarrow \quad u_i^\top Yu_i + (u_i^\top x - 1)^2 = 0.
\end{align*}
Since $u_i^\top Yu_i \geq 0$ and $(u_i^\top x - 1)^2 \geq 0$, it follows that $u_i^\top x = 1$, which in turn implies that $u_i^\top Xu_i = 2u_i^\top x - 1 = 1$. Similarly, one can prove that $\langle v_iv_i^\top, X \rangle = 1$ for all $i \in N$.

Finally, since $\bold{1}_m = \sum_{i \in N}u_i$, we have
\begin{align*}
\tr(X) = \bold{1}_m^\top x = \sum_{i \in N}u_i^\top x = n.
\end{align*}
We conclude that $(ii) \Rightarrow (i)$.
\end{proof}

\noindent
Proposition \ref{LemmaCompactForm} shows that instead of the trace constraint and the squared linear constraints,
we can equivalently include the merged squared linear constraints.
Let us now define the following set:
\begin{equation} \label{feasible}
\begin{array}{rl}
{\mathcal F}_{1} :=
\bigg \{

\begin{pmatrix}
1 & x^\top \\
x & X
\end{pmatrix}\in \mathcal{S}^{m+1} :

&~
\langle J, X \rangle = n^2,~
\left \langle \tilde{U} , \begin{pmatrix}
1 & x^\top \\ x & X
\end{pmatrix} \right \rangle = \left \langle \tilde{V} , \begin{pmatrix}
1 & x^\top \\ x & X
\end{pmatrix} \right \rangle = 0,~ \\[2.5ex]
&
~\text{diag}(X) = x,~
\begin{pmatrix}
1 & x^\top \\
x & X
\end{pmatrix} \succeq 0 \bigg \}.
\end{array}
\end{equation}
From the above discussion it follows that ${\mathcal F}_{1}  \subseteq {\mathcal F}_{\rm basic}$.
Let us now introduce our first SDP relaxation:
\begin{equation} \label{SDP1}
(SDP_1) \quad \min  \left \{ \langle Q, X \rangle ~:~
\begin{pmatrix}
1 & x^\top \\
x & X
\end{pmatrix} \in {\mathcal F}_{1} \right \}.
\end{equation}

\noindent In the sequel we show how to improve SDP relaxation \eqref{SDP1}.
Let us exploit the structure of a cycle cover to identify a zero pattern in $X$. For each $i \in N$, we know that there is exactly one arc $e$ in $\delta^+(i)$ with $x_e = 1$ and $x_f = 0$ for all other arcs $f$ leaving $i$. Hence, for each pair of distinct arcs $e, f \in \delta^+(i)$ we have $x_ex_f = 0$. This leads to the valid constraint $X_{ef} = 0$ for all $e, f \in \delta^+(i), e \neq f$. The same holds for the incoming arcs. We call these type of equalities the zero-structure constraints. We define:
\begin{align}
\mathcal{Z} := \{ (e,f) \in A \times A \, : \, \, \text{$e$ and $f$ start or end at the same node, $e \neq f$} \}. \label{defCaligraphicZ}
\end{align}
Then the zero-structure constraints read that $X_{ef} = 0$ for all $(e,f) \in \mathcal{Z}$.

Note that one may also add the  nonnegativity constraints on matrix variables in  $(SDP_1)$.
For that purpose, we define the cone of nonnegative symmetric  matrices, i.e.,
\begin{align*}
\mathcal{N}^m_+ :=\{
X \in \mathcal{S}^{m}:~X\geq 0 \}.
\end{align*}
We  show next that after adding nonnegativity constraints to the feasible set of $(SDP_1)$, the zero-structure constraints turn out to be redundant.

\begin{proposition}\label{prop:alloneNon}
Let  $x \in \mathbb{R}^m$ and $X \in \mathcal{S}^m$ be feasible for $(SDP_1)$.
If $X \in \mathcal{N}^m_+$, then $X_{ef} = 0$ for all $(e,f) \in \mathcal{Z}$. \label{LemmaZeroStructureRedundant}
\end{proposition}
\begin{proof}
We prove the statement for the outgoing arcs. The proof for the incoming arcs is similar. Using Proposition \ref{LemmaCompactForm}, we know that $\langle u_iu_i^\top, X \rangle = 1$ for all $i \in N$. We rewrite this equality as:
\begin{align*}
1 = \langle u_iu_i^\top, X \rangle = \sum_{e \in A}\sum_{f \in A}(u_i)_e(u_i)_fX_{ef} = \sum_{e \in \delta^+(i)}X_{ee} + \sum_{\substack{e, f \in \delta^+(i), \\ e \neq f}}X_{ef}.
\end{align*}
Since $\text{diag}(X) = x$, we have $\sum_{e \in \delta^+(i)}X_{ee} = \sum_{e \in \delta^+(i)}x_e = u_i^\top x = 1$, where the last equality follows from Proposition \ref{TheoremSquaredLinearConstraints}. Thus, we have $\sum_{e, f \in \delta^+(i), e \neq f}X_{ef} = 0$,
from where it follows that $X_{ef} = 0$ for all $e,f \in \delta^+(i)$ with $e \neq f$.
\end{proof}

\noindent Let us now define our next, tighter SDP relaxation:
\begin{equation} \label{SDP2}
(SDP_2) \quad \min   \left \{ \langle Q, X \rangle ~:~
\begin{pmatrix}
1 & x^\top \\
x & X
\end{pmatrix} \in {\mathcal F}_{1} \cap \mathcal{N}^{m+1}_+ \right  \}.
\end{equation}

\noindent To further strengthen $(SDP_2)$, we consider an additional set of valid inequalities.
Namely, we consider  cuts that are related to the well-known Boolean Quadric Polytope
 introduced by Padberg \cite{Padberg}. The BQP of order $m$ is defined as
\begin{align*}
BQ^m := \Conv \left\{ (x,X) \in \mathbb{R}^m \times \mathbb{R}^{m(m-1)/2} \, : \, \, x \in \{0,1\}^m, \,\,  X_{ij} = x_ix_j \,\,\,   \forall 1 \leq i < j \leq m  \right \}.
\end{align*}
Since the matrix $X$ in our previous relaxations is such that $X_{ef}$ represents $x_ex_f$, the inequalities that are valid for $BQ^m$ are also valid for our SDP relaxations.
In \cite{Padberg} it is proven  that the following triangle inequalities  (written in our \textsc{QCCP} notation) define facets of $BQ^m$:
\begin{align*}
X_{ef} + X_{eg} \leq x_e + X_{fg} \qquad \text{for all }\quad e, f, g \in A, ~e \neq f, f \neq g, e \neq g.
\end{align*}
\noindent Although there are more facet-defining inequalities for the BQP, we consider only the above mentioned ones  in this paper.
Namely, our preliminary tests show that the triangle inequalities lead to the largest improvement of the SDP bounds.
Note that there are $\mathcal{O}(m^3)$ triangle inequalities
and that it is  challenging to solve even medium-size SDPs that include all triangle inequalities.

Let $\mathcal{T} \subseteq A \times A \times A$ denote the set of arc triples corresponding to the triangle inequalities,
and  $\mathcal{C}(\mathcal{T})$ be the polyhedron induced by these cuts, i.e.,
\begin{align*}
\mathcal{C}(\mathcal{T}) := \left\{ \begin{pmatrix}
1 & x^\top \\
x & X
\end{pmatrix} \in \mathcal{S}^{m+1} \, : \, \, X_{ef} + X_{eg} \leq X_{ee} + X_{fg} \,\,\,  \forall \, (e,f,g) \in \mathcal{T}  \right\},
\end{align*}
where we incorporated the fact that $x_e = X_{ee}$ for all $e \in A$ in our relaxations.
Then our strongest SDP relaxation is:
\begin{equation} \label{SDP3}
(SDP_3) \quad \min  \left \{ \langle Q, X \rangle ~:~
\begin{pmatrix}
1 & x^\top \\
x & X
\end{pmatrix} \in {\mathcal F}_{1} \cap \mathcal{N}^{m+1}_+ \cap  \mathcal{C}(\mathcal{T})  \right \}.
\end{equation}
By abuse of notation, we will also use $\mathcal{T}$ to denote a subset of the set of arc triples corresponding to the triangle inequalities within  a cutting plane environment.

\section{Graph-dependent facial reduction} \label{SectionSlater}
In this section we investigate the Slater feasibility of the relaxations constructed in Section \ref{SectionSDPrelaxations}.
We prove that the relaxations are not Slater feasible and show how to obtain facially reduced  relaxations.
We conclude this section by providing an algorithm that computes a sparse transformation matrix required  for the facial reduction.
 Each transformation matrix is graph specific, and the algorithm exploits the bipartite representation of the underlying graph.
Our algorithm can be downloaded\footnote{The code can be downloaded from \url{https://github.com/frankdemeijer/SDPforQCCP}.} and used whenever one needs to compute a basis for the flow space of the bipartite representation of a directed graph.

\subsection{Strict feasibility by facial reduction } \label{SubsectionFacialReduction}
Recall that Slater's constraint qualification holds  for an SDP relaxation
if there exists a  feasible solution  that is also positive definite.
The following lemma shows that Slater's constraint qualification does not hold for the SDP relaxation \eqref{SDP1}, and consequently, neither for \eqref{SDP2} and \eqref{SDP3}.
\begin{lemma} \label{LemmaNotSlater} Let $Y := \begin{pmatrix}
1 & x^\top \\ x & X
\end{pmatrix} \in \mathcal{S}^{m+1}_+$ be feasible for $(SDP_1)$. Then
\begin{align*}
\mathrm{Span} \left\{ \begin{pmatrix}
-1 \\ u_i
\end{pmatrix}, \, \begin{pmatrix} - 1 \\ v_i
\end{pmatrix} \, : \, \, i \in N \right\} \subseteq \mathrm{Nul}(Y).
\end{align*}
\end{lemma}

\begin{proof}
It follows directly from the fact that $\langle \tilde{U}, Y \rangle = 0$ and positive semidefinite matrices having a nonnegative inner product that we have
\begin{align*}
\begin{pmatrix}
-1 & u_i^\top
\end{pmatrix} \begin{pmatrix}
1 & x^\top \\ x & X
\end{pmatrix} \begin{pmatrix}
-1 \\ u_i
\end{pmatrix} = 0,
\end{align*}
for all $i \in N$.
Since $Y \succeq 0$, this implies that $\begin{pmatrix}
1 & x^\top \\ x & X
\end{pmatrix} \begin{pmatrix}
-1 \\ u_i
\end{pmatrix} = \bold{0}_{m+1}$ for all $i \in N$.
Thus,  $\begin{pmatrix}
-1 \\ u_i
\end{pmatrix} \in \Nul(Y)$ for all $i \in N$. Similarly, one can prove that $\begin{pmatrix}
-1 \\ v_i
\end{pmatrix} \in \Nul(Y)$ for all $i \in N$.
\end{proof}

\noindent
Lemma \ref{LemmaNotSlater} shows  that our SDP relaxations are not Slater feasible.
Thus,  the feasible sets of the SDP relaxations are fully contained in one of the faces of $\mathcal{S}^{m+1}_+$.
For now, we only focus on the relaxation $(SDP_1)$. In order to find an equivalent relaxation for  $(SDP_1)$  that is Slater feasible,
we project the problem onto the minimal face containing the feasible set, i.e., apply facial reduction, see e.g., \cite{BorweinWolkowicz, DrusvyatskiyWolkowicz, Tuncel}.

To find the minimal face containing the feasible set of the SDP relaxation,
one needs to find its exposing vectors, i.e., the vectors orthogonal to the feasible set of the SDP relaxation.
It follows from Lemma \ref{LemmaNotSlater} that the following matrices satisfy that property:
\begin{align} \label{exposing}
\begin{pmatrix}
-1 \\
u_i
\end{pmatrix}\begin{pmatrix}
-1 \\
u_i
\end{pmatrix}^\top \quad \text{and} \quad \begin{pmatrix}
-1 \\
v_i
\end{pmatrix}\begin{pmatrix}
-1 \\
v_i
\end{pmatrix}^\top \quad \text{for all } i \in N.
\end{align}
Now, let $\mathcal{R}$ be defined as follows:
\begin{align} \label{DefinitionR}
\mathcal{R} := \left( \Span \left\{ \begin{pmatrix}
-1 \\ u_i
\end{pmatrix}, \, \begin{pmatrix} - 1 \\ v_i
\end{pmatrix} \, : \, \, i \in N \right\} \right)^\perp = \Nul \left( \begin{bmatrix}
-\bold{1}_n & U \\ - \bold{1}_n & V
\end{bmatrix} \right).
\end{align}
Observe that under Assumption \ref{AssumptionJempty} the rank of $\begin{bmatrix}
-\bold{1}_n & U \\ - \bold{1}_n & V
\end{bmatrix}$ equals the rank of $[U^\top, V^\top ]^\top$ which we defined to be $\alpha$, see (\ref{alpha}). From this it follows that $\dim(\mathcal{R}) = m + 1 - \alpha$. We now define $F_\mathcal{R}$ to be the subset of $\mathcal{S}^{m+1}_+$ that is orthogonal to the exposing vectors (\ref{exposing}), i.e.,
\[
F_{\mathcal{R}} := \{ X \in \mathcal{S}^{m+1}_+ \, : \, \, \Col(X) \subseteq \mathcal{R} \}.
\]
Since faces of $\mathcal{S}^{m+1}_+$ are known to be in correspondence with linear subspaces of $\mathbb{R}^{m+1}$ \cite{DrusvyatskiyWolkowicz}, $F_\mathcal{R}$ is a face of $\mathcal{S}^{m+1}_+$ containing the feasible set of $(SDP_1)$.
Later on we show that $F_\mathcal{R}$  is actually the minimal face with this property, see Theorem \ref{TheoremSlaterPoint}.

In order to derive an explicit expression of $F_\mathcal{R}$,
let $W \in \mathbb{R}^{(m+1) \times (m + 1 - \alpha)}$ be a matrix whose columns form a basis for $\mathcal{R}$. Then the face $F_\mathcal{R}$ can be equivalently written as:
\begin{align}\label{face}
F_\mathcal{R} = W\mathcal{S}^{m + 1 - \alpha}_+ W^\top.
\end{align}
This implies that any $Y\in \mathcal{S}^{m+1}_+$ that is feasible for $(SDP_1)$ can be written as $Y = WZW^\top$ for some $Z \in \mathcal{S}^{m+1-\alpha}_+$.
By substituting this term into $(SDP_1)$, we obtain an equivalent relaxation in a lower dimensional space.
As a direct byproduct, some of the original constraints become redundant.
The resulting relaxation is as  follows:
\begin{equation}\label{SDPs1}
(SDP_{S1}) \quad \min
\left \{
\langle W^\top\hat{Q}W, Z \rangle :~  \diag(WZW^\top ) = WZW^\top \bold{e}_0,~
\bold{e}_0^\top WZW^\top \bold{e}_0 = 1, ~Z \succeq 0
\right \},
\end{equation}
where $\hat{Q} := \begin{pmatrix}
0 & \bold{0}_m^\top \\ \bold{0}_m & Q
\end{pmatrix}$.
Let us define the feasible set of the above relaxation for future reference:
\begin{equation}
{\mathcal F}_{S1} := \left \{ Z \in \mathcal{S}^{m+1-\alpha}_+ :~
\diag(WZW^\top ) = WZW^\top \bold{e}_0,~
\bold{e}_0^\top WZW^\top \bold{e}_0 = 1, ~Z \succeq 0 \right \}.
\end{equation}

\noindent We show below that the SDP relaxations  \eqref{SDPs1} and \eqref{SDP1} are equivalent.
\begin{theorem}
The SDP relaxation $(SDP_{S1})$ is equivalent to the SDP relaxation $(SDP_1)$.
\end{theorem}
\begin{proof}
Let $Z$ be feasible for $(SDP_{S1})$ and define $Y := WZW^\top$, $X:= ( \bold{0}_m, I_m)Y(\bold{0}_m, I_m)^\top$ and $x := \diag(X)$.
Our goal is to show that $x$ and $X$ are feasible for $(SDP_1)$.

Note that the SDP constraint is trivially satisfied. Therefore, it remains to prove that the all-ones constraint and the merged squared linear constraints hold. Observe that
\begin{align*}
\begin{pmatrix}
n \\ \bold{1}_m
\end{pmatrix}^\top \begin{pmatrix}
1 & x^\top \\ x & X
\end{pmatrix} \begin{pmatrix}
-n \\ \bold{1}_m
\end{pmatrix} = \begin{pmatrix}
n \\ \bold{1}_m
\end{pmatrix}^\top WZW^\top \begin{pmatrix}
-n \\ \bold{1}_m
\end{pmatrix} = \sum_{i \in N} \begin{pmatrix}
n \\ \bold{1}_m
\end{pmatrix}^\top WZW^\top \begin{pmatrix}
-1 \\ u_i
\end{pmatrix} = 0,
\end{align*}
where the last equality follows from the construction of $W$.
Since the most left term in the expression above equals $-n^2 + \bold{1}_m^\top X\bold{1}_m$,  it follows that $\langle J, X \rangle = n^2$.

Next, we have
\begin{align*}
 \left\langle \tilde{U} , Y \right\rangle = \left\langle \sum_{i \in N} \begin{pmatrix}
-1 \\ u_i
\end{pmatrix} \begin{pmatrix}
-1 \\ u_i
\end{pmatrix}^\top, \, WZW^\top \right\rangle = \sum_{i \in N} \left\langle W^\top \begin{pmatrix}
-1 \\ u_i
\end{pmatrix} \begin{pmatrix}
-1 \\ u_i
\end{pmatrix}^\top W, \, Z \right\rangle = 0,
\end{align*}
since the columns of $W$ are orthogonal to $(-1, u_i^\top )^\top$ for all $i \in N$. In a similar fashion we can show that $\langle \tilde{V}, Y \rangle = 0$ for all $i \in N$.
We conclude that the matrix $X$ and vector $x$ obtained from $(SDP_{S1})$ are feasible for $(SDP_1)$.

Conversely, let $Y$ be feasible for  $(SDP_1)$.
Then it follows from \eqref{face} that there exists a matrix $Z \succeq 0$ such that $Y = WZW^\top$.
Since the objective functions of $(SDP_1)$ and $(SDP_{S1})$ coincide, we conclude that the two relaxations are equivalent.
\end{proof}

\noindent We now prove that $(SDP_{S1})$  is  indeed Slater feasible, see also \cite{Tuncel}.

\begin{theorem} \label{TheoremSlaterPoint}
The relaxation $(SDP_{S1})$ contains a Slater feasible point.
\end{theorem}
\begin{proof}
Since $\Conv(P)$ has dimension $m - \alpha$, see Lemma \ref{LemmaDimensionConvP}, it follows that there exists an affinely independent set of vectors $\{x_1, ..., x_{m+1 - \alpha}\} \subseteq P$. Because of the affinely independence of these vectors, the set
\begin{align*}
\left\{ \begin{pmatrix}
1 \\ x_1
\end{pmatrix}, \begin{pmatrix}
1 \\ x_2
\end{pmatrix}, ..., \begin{pmatrix}
1 \\ x_{m+1 - \alpha}
\end{pmatrix} \right\}
\end{align*}
is linearly independent in $\mathbb{R}^{m+1}$. Since $(1, x_i^\top )^\top \in \mathcal{R}$ for all $i = 1, ..., m+1 - \alpha$ and the columns of $W$ form a basis for $\mathcal{R}$,
there exist vectors $y_1, ..., y_{m+1-\alpha} \in \mathbb{R}^{m+1-\alpha}$ such that $W y_i = (1, x_i^\top )^\top$ for all $i = 1, ..., m+1-\alpha$.
Moreover, the vectors $y_i$ are linearly independent in $\mathbb{R}^{m+1-\alpha}$ because of the linear independence of the vectors $Wy_i$.

We define
\begin{align*}
Z_{\lambda} := \sum_{i = 1}^{m+1-\alpha} \lambda_i y_iy_i^\top ,
\end{align*}
where $\lambda_i \geq 0$ for all $i = 1, ..., m+1 -\alpha$ and $\bold{1}^\top \lambda = 1$, and rewrite $WZ_\lambda W^\top$ as follows:
\begin{align*}
W Z_\lambda W^\top = \sum_{i = 1}^{m+1 -\alpha} \lambda_i Wy_i(Wy_i)^\top = \sum_{i = 1}^{m+1-\alpha} \lambda_i \begin{pmatrix}
1 \\ x_i
\end{pmatrix}\begin{pmatrix}
1 \\ x_i
\end{pmatrix}^\top.
\end{align*}
It is not difficult to see that $Z_\lambda$ is feasible for $(SDP_{S1})$. By taking $\lambda_i > 0$ for all $i = 1, ..., m+1 - \alpha$, the resulting matrix $Z_\lambda$ is non-singular, which implies that $Z_\lambda \succ 0$. Hence, $(SDP_{S1})$ contains a Slater feasible point.
\end{proof}

\noindent Note that the key in the proof of Theorem \ref{TheoremSlaterPoint} is the known dimension of $\Conv(P)$.

Continuing in the same vein, one can show that the following SDP relaxation is equivalent to the SDP relaxation \eqref{SDP2}:
\begin{equation}\label{SDPS2}
(SDP_{S2}) \quad \min
\left \{
\langle W^\top\hat{Q}W, Z \rangle :~
Z\in {\mathcal F}_{S1}, ~
~WZW^\top \in \mathcal{N}^{m+1}_+
\right \},
\end{equation}
and the following relaxation equivalent to the SDP relaxation \eqref{SDP3}:
\begin{equation}\label{SDPS3}
(SDP_{S3}) \quad \min
\left \{
\langle W^\top\hat{Q}W, Z \rangle :~
Z\in {\mathcal F}_{S1},  ~WZW^\top \in \mathcal{N}^{m+1}_+ \cap ~\mathcal{C}(\mathcal{T})
\right \}.
\end{equation}

\subsection{A polynomial time algorithm for the transformation matrix} \label{sec:polyMatrix}
Although the subspace $\mathcal{R}$ has been defined algebraically in Section \ref{SubsectionFacialReduction}, we now focus on its relation with  the graph $G$. This leads to a polynomial time algorithm for computing a sparse transformation matrix $W$ that depends on the considered graph.
Although one can compute $W$ numerically, we require its sparse expression for  efficient implementation of our cutting plane algorithm, see Section \ref{SectionSolving}.\\

\noindent Recall that the columns of $W$ form a basis for the subspace $\mathcal{R}$, see (\ref{DefinitionR}). A natural way to construct $W$ is as follows: let $\bar{x} \in P$ be the characteristic vector of any cycle cover in $G$. Moreover, let $\overline{W} \in \mathbb{R}^{m \times (m - \alpha)}$ be a matrix whose columns form a basis for $\Nul([U^\top, V^\top ]^\top )$. Then, the matrix
\begingroup
\renewcommand*{\arraystretch}{1.5}
\begin{align}
W := \begin{bmatrix}
1 & \bold{0}_{m - \alpha}^\top \\
\overline{x} & \overline{W}
\end{bmatrix}
\end{align}
\endgroup
forms a basis for the subspace $\mathcal{R}$. Finding a sparse expression for $W$ now boils down to finding a sparse expression for $\overline{W}$. For that purpose, we focus on a graph $B(G)$ that is induced by $G$, the so-called bipartite representation of $G$, which is introduced by Bang-Jensen and Gutin \cite{Bang-JensenGutin}. The graph $B(G) = (V_1 \cup V_2, E)$ is an undirected bipartite graph where $V_1$ and $V_2$ are copies of the set $N$ and the edge set $E$ is defined as:
\begin{align*}
E = \left \{ \{i,j\} \in V_1 \times V_2 \, : \, \, (i,j) \in A \right \}.
\end{align*}
By construction, each arc in $G$ corresponds to exactly one edge in $B(G)$, where the orientation in $G$ determines the configuration of the edges in $B(G)$.
Figure~\ref{FigureBipartiteRepresentation} shows an example of $G$ and its corresponding bipartite representation $B(G)$. Observe that a cycle cover in $G$ corresponds to a perfect matching in $B(G)$ and vice versa.

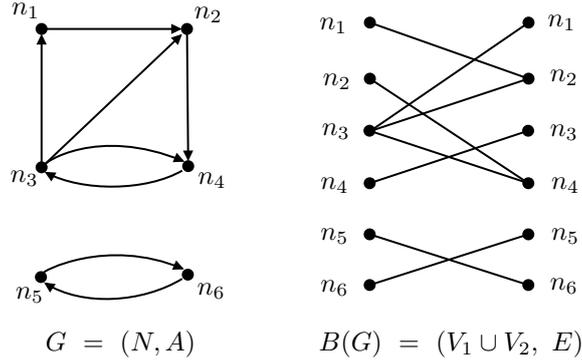
\begin{figure}[h]
\centering
\tikzset{every picture/.style={line width=0.75pt}} 

\begin{tikzpicture}[scale = 0.6, x=0.75pt,y=0.75pt,yscale=-1,xscale=1]

\draw  [fill={rgb, 255:red, 0; green, 0; blue, 0 }  ,fill opacity=1 ] (113,46.15) .. controls (113,43.86) and (114.86,42) .. (117.15,42) .. controls (119.44,42) and (121.3,43.86) .. (121.3,46.15) .. controls (121.3,48.44) and (119.44,50.3) .. (117.15,50.3) .. controls (114.86,50.3) and (113,48.44) .. (113,46.15) -- cycle ;

\draw  [color={rgb, 255:red, 0; green, 0; blue, 0 }  ,draw opacity=1 ]   (121.3,46.15) -- (232,46.15) ;
\draw [shift={(234,46.15)}, rotate = 180] [fill={rgb, 255:red, 0; green, 0; blue, 0 }  ,fill opacity=1 ][line width=0.75]  [draw opacity=0] (8.93,-4.29) -- (0,0) -- (8.93,4.29) -- cycle    ;

\draw    (117.15,162.15) -- (232.7,51.78) ;
\draw [shift={(234.15,50.4)}, rotate = 496.31] [fill={rgb, 255:red, 0; green, 0; blue, 0 }  ][line width=0.75]  [draw opacity=0] (8.93,-4.29) -- (0,0) -- (8.93,4.29) -- cycle    ;

\draw  [color={rgb, 255:red, 0; green, 0; blue, 0 }  ,draw opacity=1 ]   (238.15,46.15) -- (239.13,155) ;
\draw [shift={(239.15,157)}, rotate = 269.48] [fill={rgb, 255:red, 0; green, 0; blue, 0 }  ,fill opacity=1 ][line width=0.75]  [draw opacity=0] (8.93,-4.29) -- (0,0) -- (8.93,4.29) -- cycle    ;

\draw    (120.65,159.4) .. controls (117.66,159.55) and (162.7,126.73) .. (235.06,156.94) ;
\draw [shift={(236.15,157.4)}, rotate = 203.01] [fill={rgb, 255:red, 0; green, 0; blue, 0 }  ][line width=0.75]  [draw opacity=0] (8.93,-4.29) -- (0,0) -- (8.93,4.29) -- cycle    ;

\draw [color={rgb, 255:red, 0; green, 0; blue, 0 }  ,draw opacity=1 ]   (234.15,164.9) .. controls (204.6,182.14) and (151.28,182.88) .. (121.02,166.65) ;
\draw [shift={(119.65,165.9)}, rotate = 389.53999999999996] [fill={rgb, 255:red, 0; green, 0; blue, 0 }  ,fill opacity=1 ][line width=0.75]  [draw opacity=0] (8.93,-4.29) -- (0,0) -- (8.93,4.29) -- cycle    ;

\draw  [color={rgb, 255:red, 0; green, 0; blue, 0 }  ,draw opacity=1 ]   (117.15,158) -- (117.15,52.3) ;
\draw [shift={(117.15,50.3)}, rotate = 450] [fill={rgb, 255:red, 0; green, 0; blue, 0 }  ,fill opacity=1 ][line width=0.75]  [draw opacity=0] (8.93,-4.29) -- (0,0) -- (8.93,4.29) -- cycle    ;

\draw  [fill={rgb, 255:red, 0; green, 0; blue, 0 }  ,fill opacity=1 ] (234,46.15) .. controls (234,43.86) and (235.86,42) .. (238.15,42) .. controls (240.44,42) and (242.3,43.86) .. (242.3,46.15) .. controls (242.3,48.44) and (240.44,50.3) .. (238.15,50.3) .. controls (235.86,50.3) and (234,48.44) .. (234,46.15) -- cycle ;
\draw  [fill={rgb, 255:red, 0; green, 0; blue, 0 }  ,fill opacity=1 ] (235,161.15) .. controls (235,158.86) and (236.86,157) .. (239.15,157) .. controls (241.44,157) and (243.3,158.86) .. (243.3,161.15) .. controls (243.3,163.44) and (241.44,165.3) .. (239.15,165.3) .. controls (236.86,165.3) and (235,163.44) .. (235,161.15) -- cycle ;
\draw  [fill={rgb, 255:red, 0; green, 0; blue, 0 }  ,fill opacity=1 ] (113,162.15) .. controls (113,159.86) and (114.86,158) .. (117.15,158) .. controls (119.44,158) and (121.3,159.86) .. (121.3,162.15) .. controls (121.3,164.44) and (119.44,166.3) .. (117.15,166.3) .. controls (114.86,166.3) and (113,164.44) .. (113,162.15) -- cycle ;
\draw  [fill={rgb, 255:red, 0; green, 0; blue, 0 }  ,fill opacity=1 ] (112.46,254.15) .. controls (112.46,251.86) and (114.33,250) .. (116.64,250) .. controls (118.95,250) and (120.82,251.86) .. (120.82,254.15) .. controls (120.82,256.44) and (118.95,258.3) .. (116.64,258.3) .. controls (114.33,258.3) and (112.46,256.44) .. (112.46,254.15) -- cycle ;
\draw  [fill={rgb, 255:red, 0; green, 0; blue, 0 }  ,fill opacity=1 ] (234.67,252.12) .. controls (234.67,249.82) and (236.54,247.97) .. (238.85,247.97) .. controls (241.16,247.97) and (243.03,249.82) .. (243.03,252.12) .. controls (243.03,254.41) and (241.16,256.27) .. (238.85,256.27) .. controls (236.54,256.27) and (234.67,254.41) .. (234.67,252.12) -- cycle ;
\draw  [color={rgb, 255:red, 0; green, 0; blue, 0 }  ,draw opacity=1 ]   (119.14,250.3) .. controls (116.14,250.45) and (170.6,219.91) .. (234.68,247.18) ;
\draw [shift={(235.65,247.6)}, rotate = 203.47] [fill={rgb, 255:red, 0; green, 0; blue, 0 }  ,fill opacity=1 ][line width=0.75]  [draw opacity=0] (8.93,-4.29) -- (0,0) -- (8.93,4.29) -- cycle    ;

\draw  [color={rgb, 255:red, 0; green, 0; blue, 0 }  ,draw opacity=1 ]   (234.15,255.6) .. controls (201.15,275.79) and (151.18,279.01) .. (120.54,259.51) ;
\draw [shift={(119.15,258.6)}, rotate = 393.90999999999997] [fill={rgb, 255:red, 0; green, 0; blue, 0 }  ,fill opacity=1 ][line width=0.75]  [draw opacity=0] (8.93,-4.29) -- (0,0) -- (8.93,4.29) -- cycle    ;

\draw  [fill={rgb, 255:red, 0; green, 0; blue, 0 }  ,fill opacity=1 ] (386.17,40.82) .. controls (386.17,38.52) and (388.02,36.67) .. (390.32,36.67) .. controls (392.61,36.67) and (394.47,38.52) .. (394.47,40.82) .. controls (394.47,43.11) and (392.61,44.97) .. (390.32,44.97) .. controls (388.02,44.97) and (386.17,43.11) .. (386.17,40.82) -- cycle ;
\draw  [fill={rgb, 255:red, 0; green, 0; blue, 0 }  ,fill opacity=1 ] (386,218.32) .. controls (386,216.02) and (387.86,214.17) .. (390.15,214.17) .. controls (392.44,214.17) and (394.3,216.02) .. (394.3,218.32) .. controls (394.3,220.61) and (392.44,222.47) .. (390.15,222.47) .. controls (387.86,222.47) and (386,220.61) .. (386,218.32) -- cycle ;
\draw  [fill={rgb, 255:red, 0; green, 0; blue, 0 }  ,fill opacity=1 ] (386,175.48) .. controls (386,173.19) and (387.86,171.33) .. (390.15,171.33) .. controls (392.44,171.33) and (394.3,173.19) .. (394.3,175.48) .. controls (394.3,177.78) and (392.44,179.63) .. (390.15,179.63) .. controls (387.86,179.63) and (386,177.78) .. (386,175.48) -- cycle ;
\draw  [fill={rgb, 255:red, 0; green, 0; blue, 0 }  ,fill opacity=1 ] (385.83,131.48) .. controls (385.83,129.19) and (387.69,127.33) .. (389.98,127.33) .. controls (392.28,127.33) and (394.13,129.19) .. (394.13,131.48) .. controls (394.13,133.78) and (392.28,135.63) .. (389.98,135.63) .. controls (387.69,135.63) and (385.83,133.78) .. (385.83,131.48) -- cycle ;
\draw  [fill={rgb, 255:red, 0; green, 0; blue, 0 }  ,fill opacity=1 ] (386.17,87.82) .. controls (386.17,85.52) and (388.02,83.67) .. (390.32,83.67) .. controls (392.61,83.67) and (394.47,85.52) .. (394.47,87.82) .. controls (394.47,90.11) and (392.61,91.97) .. (390.32,91.97) .. controls (388.02,91.97) and (386.17,90.11) .. (386.17,87.82) -- cycle ;
\draw  [fill={rgb, 255:red, 0; green, 0; blue, 0 }  ,fill opacity=1 ] (386,260.82) .. controls (386,258.52) and (387.86,256.67) .. (390.15,256.67) .. controls (392.44,256.67) and (394.3,258.52) .. (394.3,260.82) .. controls (394.3,263.11) and (392.44,264.97) .. (390.15,264.97) .. controls (387.86,264.97) and (386,263.11) .. (386,260.82) -- cycle ;
\draw  [fill={rgb, 255:red, 0; green, 0; blue, 0 }  ,fill opacity=1 ] (518.17,40.82) .. controls (518.17,38.52) and (520.02,36.67) .. (522.32,36.67) .. controls (524.61,36.67) and (526.47,38.52) .. (526.47,40.82) .. controls (526.47,43.11) and (524.61,44.97) .. (522.32,44.97) .. controls (520.02,44.97) and (518.17,43.11) .. (518.17,40.82) -- cycle ;
\draw  [fill={rgb, 255:red, 0; green, 0; blue, 0 }  ,fill opacity=1 ] (518,218.32) .. controls (518,216.02) and (519.86,214.17) .. (522.15,214.17) .. controls (524.44,214.17) and (526.3,216.02) .. (526.3,218.32) .. controls (526.3,220.61) and (524.44,222.47) .. (522.15,222.47) .. controls (519.86,222.47) and (518,220.61) .. (518,218.32) -- cycle ;
\draw  [fill={rgb, 255:red, 0; green, 0; blue, 0 }  ,fill opacity=1 ] (518,175.48) .. controls (518,173.19) and (519.86,171.33) .. (522.15,171.33) .. controls (524.44,171.33) and (526.3,173.19) .. (526.3,175.48) .. controls (526.3,177.78) and (524.44,179.63) .. (522.15,179.63) .. controls (519.86,179.63) and (518,177.78) .. (518,175.48) -- cycle ;
\draw  [fill={rgb, 255:red, 0; green, 0; blue, 0 }  ,fill opacity=1 ] (517.83,131.48) .. controls (517.83,129.19) and (519.69,127.33) .. (521.98,127.33) .. controls (524.28,127.33) and (526.13,129.19) .. (526.13,131.48) .. controls (526.13,133.78) and (524.28,135.63) .. (521.98,135.63) .. controls (519.69,135.63) and (517.83,133.78) .. (517.83,131.48) -- cycle ;
\draw  [fill={rgb, 255:red, 0; green, 0; blue, 0 }  ,fill opacity=1 ] (518.17,87.82) .. controls (518.17,85.52) and (520.02,83.67) .. (522.32,83.67) .. controls (524.61,83.67) and (526.47,85.52) .. (526.47,87.82) .. controls (526.47,90.11) and (524.61,91.97) .. (522.32,91.97) .. controls (520.02,91.97) and (518.17,90.11) .. (518.17,87.82) -- cycle ;
\draw  [fill={rgb, 255:red, 0; green, 0; blue, 0 }  ,fill opacity=1 ] (518,260.82) .. controls (518,258.52) and (519.86,256.67) .. (522.15,256.67) .. controls (524.44,256.67) and (526.3,258.52) .. (526.3,260.82) .. controls (526.3,263.11) and (524.44,264.97) .. (522.15,264.97) .. controls (519.86,264.97) and (518,263.11) .. (518,260.82) -- cycle ;
\draw [color={rgb, 255:red, 0; green, 0; blue, 0 }  ,draw opacity=1 ]   (390.32,40.82) -- (522.32,87.82) ;

\draw    (390.48,131.82) -- (522.32,87.82) ;

\draw  [color={rgb, 255:red, 0; green, 0; blue, 0 }  ,draw opacity=1 ]   (390.32,87.82) -- (522.15,175.48) ;

\draw    (389.98,131.48) -- (522.15,175.48) ;

\draw  [color={rgb, 255:red, 0; green, 0; blue, 0 }  ,draw opacity=1 ]   (390.15,175.48) -- (521.98,131.48) ;

\draw  [color={rgb, 255:red, 0; green, 0; blue, 0 }  ,draw opacity=1 ]   (390.48,131.82) -- (522.32,40.82) ;

\draw  [color={rgb, 255:red, 0; green, 0; blue, 0 }  ,draw opacity=1 ]   (390.15,260.82) -- (522.15,218.32) ;

\draw  [color={rgb, 255:red, 0; green, 0; blue, 0 }  ,draw opacity=1 ]   (390.15,218.32) -- (522.15,260.82) ;

\draw (104,31) node  [align=left] {$\displaystyle n_{1}$};
\draw (360,43) node  [align=left] {$\displaystyle n_{1}$};
\draw (550,41) node  [align=left] {$\displaystyle n_{1}$};
\draw (256,34) node  [align=left] {$\displaystyle n_{2}$};
\draw (363,90) node  [align=left] {$\displaystyle n_{2}$};
\draw (552,87) node  [align=left] {$\displaystyle n_{2}$};
\draw (103,171) node  [align=left] {$\displaystyle n_{3}$};
\draw (362,132) node  [align=left] {$\displaystyle n_{3}$};
\draw (552,133) node  [align=left] {$\displaystyle n_{3}$};
\draw (259,173) node  [align=left] {$\displaystyle n_{4}$};
\draw (361,177) node  [align=left] {$\displaystyle n_{4}$};
\draw (553,177) node  [align=left] {$\displaystyle n_{4}$};
\draw (107.42,269) node  [align=left] {$\displaystyle n_{5}$};
\draw (361,220) node  [align=left] {$\displaystyle n_{5}$};
\draw (553,220) node  [align=left] {$\displaystyle n_{5}$};
\draw (258.45,266) node  [align=left] {$\displaystyle n_{6}$};
\draw (361,263) node  [align=left] {$\displaystyle n_{6}$};
\draw (552,262) node  [align=left] {$\displaystyle n_{6}$};
\draw (183,310) node  [align=left] {$\displaystyle G\ =\ ( N,A)$};
\draw (458,310) node  [align=left] {$\displaystyle B(G)\ =\ ( V_{1} \cup V_{2} ,\ E)$};

\end{tikzpicture}

\caption{Example of graph $G$ and its bipartite representation $B(G)$. \label{FigureBipartiteRepresentation}}
\end{figure}

The matrix $[U^\top, V^\top ]^\top$ equals the incidence matrix of $B(G)$. Suppose we orient all edges of $B(G)$ from $V_1$ to $V_2$. The incidence matrix with respect to this orientation equals $[U^\top, - V^\top ]^\top$. Clearly, we have $\Nul( [U^\top, V^\top ]^\top) = \Nul([U^\top, -V^\top]^\top )$. The null space of the incidence matrix of a directed graph is in the literature known as the flow space of a graph. Hence, it follows that the columns of $\overline{W}$ form a basis for the flow space of the bipartite representation of $G$ (with respect to the orientation from $V_1$ to $V_2$).

Let $C$ be a cycle in $B(G)$. Since $B(G)$ is a bipartite graph, $C$ consists of an even number of edges. Let $z \in \mathbb{R}^m$ denote its signed characteristic vector, i.e., we alternately assign values $+1$ and $-1$ to the edges on $C$ and assign value 0 otherwise. It is well-known that the flow space of a graph is spanned by the signed characteristic vectors of all its cycles. For more information about the flow space of a graph, we refer to e.g., \cite{GodsilRoyle}. \\ \\
Hence, $\mathcal{R}$ is related to the cycles of the bipartite representation of $G$. A natural question is how do the cycles of $B(G)$ relate to the original graph $G$? To answer this question, we exploit the notion of a closed antidirected trail, which is introduced in \cite{Bang-JensenEtAl}. Recall that a trail is a walk in a graph that does not contain repeated arcs, but is allowed to contain repeated nodes. A closed trail is a trail that has the same start and ending node.

\begin{definition} A closed antidirected trail (CAT) in a directed graph $G$ is a closed trail of even length with arcs oriented alternately.
\end{definition}

\noindent A cycle in $B(G)$ corresponds to a CAT in $G$. To verify this, let $\phi : E \rightarrow A$ be the bijection between the edges of $B(G)$ and the arcs of $G$ in the natural way. Then, $C$ equals a cycle in $B(G)$ if and only if $\phi(C)$ equals a CAT in $G$. Obviously, since $C$ starts and ends at the same vertex in $B(G)$, $\phi(C)$ also starts and ends at the same node in $G$. Moreover, since $B(G)$ is bipartite, $C$ and thus $\phi(C)$ must be of even length. Finally, each two consecutive edges of $C$ have one common vertex in $V_1$ (resp.\ $V_2$) and the other vertices in $V_2$ (resp.\ $V_1$). By construction of $B(G)$, it follows that two consecutive edges of $C$ correspond to alternately oriented arcs in $G$. Thus $\phi(C)$ is a CAT. The reverse statement can be shown in the same fashion. This leads to the following proposition.
\begin{proposition} The flow space of $B(G)$ equals the subspace spanned by the closed antidirected trails in $G$.
\end{proposition}
\noindent We now have two interpretations of the column space of $\overline{W}$, one with respect to $G$ and the other with respect to $B(G)$. The latter one is more suitable for finding a sparse expression for $\overline{W}$.

Since $\overline{W}$ has $m - \alpha$ columns and the flow space of $B(G)$ has dimension $|E| - |V_1 \cup V_2| + c_{B(G)}$, where $c_{B(G)}$ equals the number of connected components in $B(G)$, it follows that:
\begin{align}
\alpha = |V_1 \cup V_2| - c_{B(G)} = 2n - c_{B(G)}.
\end{align}
Observe that the extreme cases are established by the directed cycle and the complete digraph on $n$ nodes, which yield $\alpha = n$ and $\alpha = 2n -1$, respectively.

There exist several natural bases for the flow space of a graph, see e.g., \cite{GodsilRoyle}. We use the following construction: Let $T$ be a spanning forest of $B(G)$ and let $E(T) \subseteq E$ denote its corresponding edge set. Then, for all $e \in E \setminus E(T)$, we know that $T \cup \{e\}$ contains a cycle. By alternately assigning values $+1$ and $-1$ to the edges of the cycle and assigning value 0 to all remaining edges, we obtain a signed characteristic vector of the cycle. By repeating this construction for all edges in $E \setminus E(T)$, we obtain $m - \alpha$ linearly independent vectors in $\Nul([U^\top, V^\top ]^\top )$, which form a basis for this subspace.
Finding a spanning forest and detecting a cycle in $T \cup \{e\}$ can both be done by a breadth first search.

The pseudo-code for the computation of a sparse $W$ is given in Algorithm \ref{AlgW}. This algorithm applies to all \textsc{QCCP} instances.

\algnewcommand{\AND}{\algorithmicand}
\renewcommand{\algorithmicrequire}{\textbf{Input:}}
\renewcommand{\algorithmicensure}{\textbf{Output:}}

\begin{algorithm}[h!]
\footnotesize
\caption{Computation of Transformation Matrix $W$}\label{AlgW}
\begin{algorithmic}[1]
\Require $G = (N, A)$
\State Construct the bipartite representation $B(G) = (V_1 \cup V_2, E)$ of $G$.
\State Find a spanning forest $T$ of $B(G)$. \label{aspanningtree}
\For {$e \in E \setminus E(T)$}
\State Find the unique cycle $C$ in $T \cup \{e\}$. \label{afindcycle}
\State Alternately assign values $+1$ and $-1$ to edges on $C$.
\State Construct vector $w^e \in \mathbb{R}^m$ by $w^e_f = \begin{cases} \pm 1 & \text{if $f \in C$ (according to step 5),} \\
0 & \text{otherwise.}
\end{cases}$
\EndFor
\State Find a cycle cover $\bar{x} \in P$. \label{avectorx}
\State Let $W \in \mathbb{R}^{(m+1) \times (m + 1 - \alpha)}$ be the matrix whose columns are $\begin{pmatrix}
1 \\ \bar{x}
\end{pmatrix} \cup \left\{ \begin{pmatrix}
0 \\
w^e
\end{pmatrix} :  e \in E \setminus E(T) \right\}$.
\Ensure $W$
\end{algorithmic}
\end{algorithm}

\begin{remark} Although Algorithm \ref{AlgW} uses $B(G)$ to compute $W$, it is possible to perform the same construction using the original graph $G$. This follows from the fact that the CATs of $G$ form the circuits of a matroid $(A,\mathcal{F})$ where
\begin{align*}
\mathcal{F} := \{ F \subseteq A \, : \, \, \text{subgraph $(N(F),F)$ does not contain a CAT} \},
\end{align*}
see \cite{Bang-JensenEtAl}. Step \ref{aspanningtree} of Algorithm \ref{AlgW} then reduces to finding a maximal basis of $(A, \mathcal{F})$ using a greedy algorithm, while step \ref{afindcycle} boils down to finding the unique CAT in $T \cup \{e\}$ using a breadth first search.
\end{remark}

\section{A  cutting plane augmented Lagrangian approach} \label{SectionSolving}

It is known that SDP solvers based on interior point methods  exhibit problems in terms of both time and memory for solving  even medium-size SDPs.
Moreover,  interior point methods have difficulties with handling additional cutting planes such as nonnegativity constraints and triangle inequalities.
Therefore, solving strong SDP models remains  a challenging task.

Recently, a promising alternative for solving large-scale SDP relaxations based on alternating direction augmented Lagrangian methods has been investigated, see \cite{BurerVandenbussche, PovhEtAl, WenEtAl, ZhaoEtAl,SunTohetAl}.
There exist several variants of alternating direction augmented Lagrangian methods for solving SDPs, see e.g., \cite{PovhEtAl, ZhaoEtAl,HeEtAl2014,HeEtAl2016, OliveiraEtAl, HuEtAl, HuSotirov}.
A recent method for solving large-scale SDPs that is related to the augmented Lagrangian paradigm is the conditional gradient augmented Lagrangian method \cite{MaiEtAl1,MaiEtAl2,YurtseverEtAl}.
Here, we first  consider two variants known as the (original) Alternating Direction Method of Multipliers (\textsc{ADMM}) and the Peaceman--Rachford splitting method (\textsc{PRSM}), also called the symmetric \textsc{ADMM}.
Then, we present a novel approach that puts these alternating direction augmented Lagrangian methods into a cutting plane framework.
In particular, we show how to efficiently combine the \textsc{PRSM} with Dykstra's projection algorithm \cite{Dykstra} within  a cutting plane approach.

\subsection{The Alternating Direction Method of Multipliers and the  Peaceman-Rachford Splitting Method} \label{SubsectionADMM}
The \textsc{ADMM} is a first-order method that is introduced in the 1970s to solve large-scale convex optimization problems. Starting from the augmented Lagrangian function, it decomposes the problem into various subproblems that are relatively easy to solve.
In \cite{OliveiraEtAl}, the  authors use the \textsc{ADMM} to solve an SDP relaxation for the quadratic assignment problem and in \cite{HuSotirov} the similar approach is used to compute strong SDP bounds  for
the quadratic shortest path problem.
Their approaches allow for inexpensive iterations and cheap ways for obtaining lower and upper bounds.
In this section we first show how to exploit the approach from  \cite{OliveiraEtAl,HuSotirov}  to solve $(SDP_{S2})$ by the \textsc{ADMM}.  Then, we present the \textsc{PRSM} for our problem.\\

\noindent Let us first rewrite $(SDP_{S2})$ by introducing the constraint $Y = WZW^\top$. The purpose of adding this equality is to split the remaining set of constraints into the SDP constraint on $Z$ and the linear constraints on $Y$. To deal with the latter type, we introduce the following set:
\begin{align} \label{DefCaligraphicY}
\mathcal{Y} := \left\{Y \in \mathcal{S}^{m+1} \, : \, \, \begin{aligned} Y_{00} = 1, \quad \diag(Y) = Y\bold{e}_0,  \quad Y_{ef} \leq 1 \, \, \, \forall e \neq f  \\
  Y \geq \bold{0}, \,\,\, \tr(Y) = n+1, \,\,\, Y_{ef} = 0  \quad  \forall (e,f) \in \mathcal{Z}
\end{aligned} \right\},
\end{align}
where $\mathcal{Z}$ is given in (\ref{defCaligraphicZ}).
Observe that $\mathcal{Y}$ also contains the constraints that are redundant for $(SDP_{S2})$, see Section \ref{SectionSDPrelaxations}.
However, these constraints are not redundant in the subproblems after splitting, see \eqref{Ysubproblem} below. By including them in $\mathcal{Y}$ we therefore fasten the convergence of the \textsc{ADMM} as observed in \cite{HuSotirov,HuEtAl,OliveiraEtAl}. Indeed, these constraints make the alternating projections more accurate.

\begin{remark} Observe that $\mathcal{Y}$ does not contain the redundant constraint $\langle J, Y \rangle = (n+1)^2$.
Namely, our preliminary experiments show that the gain in convergence after adding that constraint
 is not worth the additional computational effort caused by adding it to $\mathcal{Y}$.
\end{remark}

\noindent Now, the starting point of the algorithm is the following relaxation:
\begin{equation}   \label{startADMM}
\min  \left \{   \langle  \hat{Q}, Y \rangle :~  Y = WZW^\top, ~ Y \in \mathcal{Y}, ~ Z \succeq 0 \right \},
\end{equation}
that is equivalent to $(SDP_{S2})$.
We assume that the transformation matrix $W$ is normalized such that $W^\top W = I$. Observe that the sparse $W$ resulting from Algorithm \ref{AlgW} does not have orthogonal columns.
Therefore, we apply a QR-decomposition on the matrix obtained from Algorithm \ref{AlgW}.\\

\noindent Let $S \in \mathcal{S}^{m+1}$ denote the Lagrange multiplier  for the linear constraint $Y = WZW^\top$.
We consider the augmented Lagrangian function of (\ref{startADMM}) w.r.t.~this constraint for a fixed penalty parameter $\beta > 0$:
\begin{align*}
L_\beta (Z,Y,S) := \langle \hat{Q}, Y \rangle + \langle S, Y - WZW^\top \rangle + \frac{\beta}{2}|| Y - WZW^\top ||_F^2.
\end{align*}

\noindent The \textsc{ADMM} aims to minimize $L_\beta(Z,Y,S)$ subject to $Y \in \mathcal{Y}$ and $Z \succeq 0$ while iteratively updating $S$. This problem can be decomposed into subproblems, where we only minimize with respect to one of the matrix variables while keeping the other fixed.

Suppose that $(Z^k, Y^k, S^k)$ denotes the $k$-th iterate of the \textsc{ADMM}. Then the new iterate $(Z^{k+1},$ $Y^{k+1}, S^{k+1})$ can be obtained by the following updates:
\begin{numcases}{(ADMM)\quad}
\,\,\, Z^{k+1} & $:=  \arg \underset{Z \succeq 0}{\min} \,\, L_\beta(Z, Y^k, S^k)$,\label{Zsubproblem} \\
\,\,\, Y^{k+1} & $:=  \arg \underset{Y \in \mathcal{Y}}{\min} \,\, L_\beta (Z^{k+1}, Y, S^k )$, \label{Ysubproblem} \\
\,\,\,  S^{k+1} & $:= S^k + \gamma \cdot \beta \cdot (Y^{k+1} - WZ^{k+1}W^\top)$. \label{Ssubproblem}
\end{numcases}
Here $\gamma \in (0, \frac{1 + \sqrt{5}}{2})$ is the stepsize parameter for updating the Lagrange multiplier $S$, see e.g., \cite{WenEtAl}. The efficiency of the \textsc{ADMM} depends on the difficulty of solving the subproblems (\ref{Zsubproblem}) and (\ref{Ysubproblem}).
\\ \\
The $Z$-subproblem can be solved as follows, see also \cite{OliveiraEtAl, HuSotirov}:
\begin{align*}
Z^{k+1} & = \arg \min_{Z \succeq 0} \left[ \langle \hat{Q}, Y^k \rangle - \frac{1}{2 \beta} \left\Vert S^k\right\Vert_F^2 + \frac{\beta}{2} \left\Vert WZW^\top - \left(Y^k + \frac{1}{\beta}S^k \right)\right\Vert_F^2 \right] \\
&= \arg \min_{Z\succeq 0} \left\Vert WZW^\top - \left(Y^k + \frac{1}{\beta}S^k \right)\right\Vert_F^2 = \mathcal{P}_{\mathcal{S}^{m+1-\alpha}_+}\left( W^\top \left(Y^k + \frac{1}{\beta}S^k \right)W\right),
\end{align*}
where $\mathcal{P}_{\mathcal{S}^{m}_+}( \cdot )$ denotes the orthogonal projection onto the cone of positive semidefinite matrices of order $m$, which can be performed explicitly, see e.g., \cite{Higham}. \\ \\
The $Y$-subproblem can be rewritten as follows:
\begin{align*}
Y^{k+1} & = \arg \min_{Y \in \mathcal{Y}} \left[ \langle \hat{Q}, WZ^{k+1}W^\top \rangle - \frac{\beta}{2} \left \Vert \frac{\hat{Q} + S^k}{\beta}\right \Vert_F^2 + \frac{\beta}{2} \left\Vert Y - WZ^{k+1}W^\top + \frac{\hat{Q} + S^k}{\beta}\right\Vert_F^2 \right] \\
& = \arg \min_{Y \in \mathcal{Y}} \left \Vert  Y -  \left( WZ^{k+1}W^\top - \frac{\hat{Q} + S^k}{\beta} \right) \right\Vert_F^2 = \mathcal{P}_\mathcal{Y} \left(  WZ^{k+1}W^\top - \frac{\hat{Q} + S^k}{\beta} \right),
\end{align*}
where $\mathcal{P}_\mathcal{Y}(\cdot)$ denotes the orthogonal projection onto the polyhedral set $\mathcal{Y}$.

We now show how to project a matrix $M \in \mathcal{S}^{m+1}$ onto $\mathcal{Y}$. For that purpose, we define several operators, see Table~\ref{TableOperators}.

\begin{table}[H]
\centering
\scriptsize
\begin{tabular}{@{}ccll@{}}
\toprule
     \multicolumn{3}{c}{Operator}                                               & Description                                                                                                                                                 \\ \midrule
$T_{\arrow}$     & : & $\mathcal{S}^{m+1} \rightarrow \mathbb{R}^m$                      & \begin{tabular}[c]{@{}l@{}} $T_\arrow \left( \begin{pmatrix}
x_0 & x^\top \\ x & X
\end{pmatrix} \right) = \frac{1}{3}\left( \diag(X) + 2 x\right)$. \end{tabular}                 \\[0.3cm]

$T_\arrow^*$ & : & $\mathbb{R}^m \rightarrow \mathcal{S}^{m+1}$ & $T_\arrow^*(x) = \begin{pmatrix}
0 & \frac{1}{3}x^\top \\ \frac{1}{3}x & \Diag(\frac{1}{3}x)
\end{pmatrix}$. \\[0.3cm]

$T_{\inner}$     & : & $\mathcal{S}^{m+1} \rightarrow \mathcal{S}^{m+1}$                 & \begin{tabular}[c]{@{}l@{}}
$T_{\inner}\left( \begin{pmatrix}
x_0 & x^\top \\ x & X
\end{pmatrix} \right) =  \begin{pmatrix}
0 & \bold{0}_m^\top \\ \bold{0}_m & \tilde{X} - \Diag(\tilde{X})
\end{pmatrix} $ where $\tilde{X} \in \mathcal{S}^m$ is s.t. \\
$\tilde{X}_{ef} = 0$ if $(e,f) \in \mathcal{Z}$ and $\tilde{X}_{ef} = X_{ef}$ otherwise.
\end{tabular}\\ [0.5cm]
$T_{\text{box}}$ & : & $\mathcal{S}^{m+1} \rightarrow \mathcal{S}^{m+1}$ & $T_{\text{box}}(X)_{ef} = \min( \max( X_{ef}, 0), 1)$ for all $(e,f)$.

\\ \bottomrule
\end{tabular}
\caption{Overview of operators and their definitions. \label{TableOperators}}
\end{table}
\noindent Let $\hat{M}$ denote the projection of a matrix $M$ onto $\mathcal{Y}$. The projection can be split into two parts: the projection of the so-called arrow of $M$, i.e., the zeroth row, zeroth column and diagonal of $M$, and the projection of the remaining entries. We specify details below.

We clearly have $\hat{M}_{00} = 1$. The remaining entries of the arrow of $\hat{M}$ are obtained as the solution to the following minimization problem:
\begin{align*}
& \min_{y \in \mathbb{R}^m} \left\{ \left \Vert y - T_\arrow(M) \right\Vert_2^2 : \, \, \bold{1}^\top y = n,~ y \geq \bold{0} \right\} .
\end{align*}
Observe that the problem above boils down to a projection of a vector onto the simplex $\Delta(n)$, where
$\Delta(a) := \{x \in \mathbb{R}^m \, : \, \, \bold{1}^\top x = a, ~x \geq \bold{0}\}$ for all nonnegative $a \in \mathbb{R}$.
The projection onto $\Delta(a)$,  denoted  by $\mathcal{P}_{\Delta(a)}( \cdot )$,  can be performed explicitly in $O(m \log m)$, see \cite{HeldEtAl}. The projection of the remaining entries of $M$ is trivial.

We conclude that the explicit projection of $M$ onto $\mathcal{Y}$ equals:
\begin{align*}
\mathcal{P}_\mathcal{Y}(M) = E_{00} + {T_{\rm box}} \Big( T_\inner(M) \Big) + T_\arrow^*\Big( 3 \cdot \mathcal{P}_{\Delta(n)}\big( T_\arrow(M) \big) \Big),
\end{align*}
where $E_{00} = \bold{e}_0\bold{e}_0^\top \in \mathcal{S}^{m+1}$. The fact that our SDP relaxations satisfy the constant trace property, i.e., $\tr(Y) = n+1$, is exploited in the $Y$-subproblem. The presence of the constant trace property in SDPs has been exploited recently in conditional gradient-based augmented Lagrangian methods. These methods iteratively solve a minimization problem with respect to the set of positive semidefinite matrices having fixed trace, see e.g., \cite{MaiEtAl1,MaiEtAl2,YurtseverEtAl}. In contrast, our method exploits the constant trace property in the polyhedral projections.\\ \\
In the \textsc{ADMM} the Lagrange multiplier is only updated after both primal variables have been updated. We present below the Peaceman--Rachford splitting method (\textsc{PRSM}) or the symmetric \textsc{ADMM} with larger stepsize \cite{ HeEtAl2016}.
This method consists of two dual updates per iteration. Let $(Z^k, Y^k, S^k)$ denote the $k$-th iterate of the \textsc{PRSM}.
Then the following iterative scheme is applied:
\begin{numcases}{(PRSM)\quad}
\,\,\, Z^{k+1} & $:=  \arg \underset{Z \succeq 0}{\min} \,\, L_\beta(Z, Y^k, S^k)$,\label{ZsubproblemPR} \\
\,\,\, S^{k + \frac{1}{2}} & $:= S^k + \gamma_1 \cdot \beta \cdot (Y^k - WZ^{k+1}W^\top )$, \\
\,\,\, Y^{k+1} & $:=  \arg \underset{Y \in \mathcal{Y}}{\min} \,\, L_\beta (Z^{k+1}, Y, S^{k + \frac{1}{2}} )$, \label{YsubproblemPR} \\
\,\,\,  S^{k+1} & $:= S^{k + \frac{1}{2}} + \gamma_2 \cdot \beta \cdot (Y^{k+1} - WZ^{k+1}W^\top)$. \label{SsubproblemPR}
\end{numcases}
Here $\gamma_1$ and $\gamma_2$ are  parameters that must be carefully chosen in order to guarantee convergence.
The \textsc{PRSM} is known for accelerated speed of convergence in comparison with other \textsc{ADMM}-like algorithms, see \cite{HeEtAl2016}.

\subsection{ADMM versus PRSM: Preliminary Results} \label{sect:preliminarCompare}
In Section \ref{SubsectionADMM} we present two methods for solving $(SDP_{S2})$: the \textsc{ADMM} and the \textsc{PRSM}. Both approaches can be incorporated within  the cutting plane augmented Lagrangian method that we present later. We here provide some preliminary experiments to present the behaviour of both methods in terms of convergence.

We consider a test set of 10 Erd\H os-R\'enyi instances with $m$ ranging from 250 to 750, see Section \ref{SectionNumerics} for a specification of these instances. For each instance, we use the \textsc{ADMM} and the \textsc{PRSM} to compute $(SDP_{S2})$ under the same parameter settings as will be explained in Section \ref{SectionNumerics}. We compute lower bounds obtained from the methods, see Section \ref{SubsubsectionLB}, and scale them such that the final bound is indexed to 100. Figure~\ref{FigureADMMvsPRSM:a} shows these scaled bounds for all instances, while Figure~\ref{FigureADMMvsPRSM:b} shows their average over all instances with respect to the number of iterations performed.

\begin{subfigures}
\centering
\begin{figure}[H]
\centering
\includegraphics[scale=0.46]{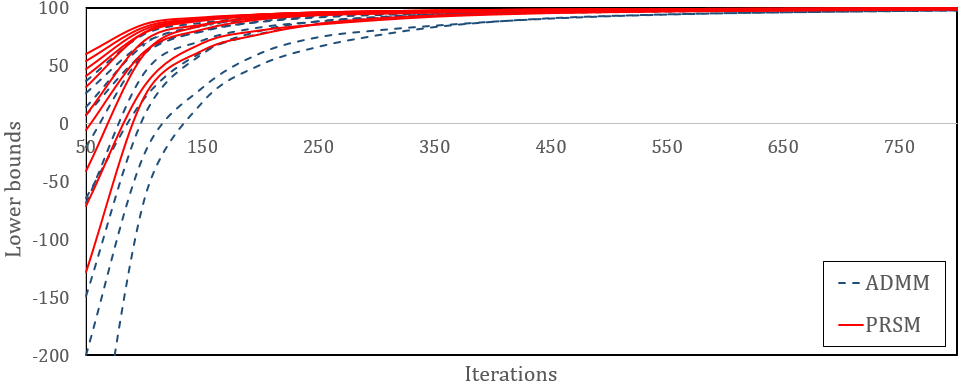}
\caption{Lower bounds for the \textsc{ADMM} (dashed) and the \textsc{PRSM} (solid) for full test set. \label{FigureADMMvsPRSM:a}}
\end{figure}
\begin{figure}[H]
\centering
\includegraphics[scale=0.46]{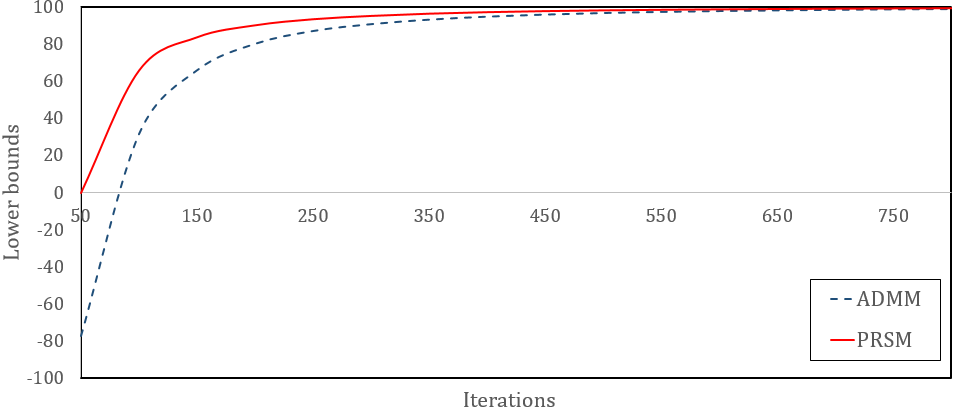}
\caption{Lower bounds for the \textsc{ADMM} (dashed) and the \textsc{PRSM} (solid) on average.  \label{FigureADMMvsPRSM:b}}
\end{figure}\label{FigureADMMvsPRSM}
\end{subfigures}
\noindent Figure~\ref{FigureADMMvsPRSM} shows that although both methods converge, the \textsc{PRSM} in general produces strong lower bounds faster than the \textsc{ADMM}. This is in line with the  accelerated numerical performance of the \textsc{PRSM} in contrast to the \textsc{ADMM} presented in \cite{HeEtAl2016}. Because we desire a fast convergence  when iteratively adding cuts, we incorporate the \textsc{PRSM} in the cutting plane augmented Lagrangian approach introduced in Section  \ref{sect:cutPlane}.

\subsection{Projection onto a single BQP Cut} \label{SubsectionBQP}
The implementation of the \textsc{ADMM} and the \textsc{PRSM} discussed in the previous section can be used to solve $(SDP_{S1})$ and $(SDP_{S2})$. In order to solve $(SDP_{S3})$, the constraints $Y \in \mathcal{C}(\mathcal{T})$ are added to the set of polyhedral constraints, which significantly increases the complexity of the $Y$-subproblem (\ref{Ysubproblem}).
To project onto $\mathcal{Y} \cap \mathcal{C}(\mathcal{T})$, we use an iterative projection framework, see Section \ref{SubsectionDykstra}. In this section, we first show how to project onto the polyhedron induced by a single triangle inequality.\\ \\
We assume that $\mathcal{T}$ contains the arc triples $(e,f,g)$ with 
$e \neq f, f \neq g, e \neq g$ that correspond to  (possibly violated) triangle inequalities, e.g., resulting from a cutting plane framework. For each $(e,f,g) \in \mathcal{T}$ we let $\mathcal{H}_{efg}$ be the following polyhedron:
\begin{align*}
\mathcal{H}_{efg} := \left\{ Y \in \mathcal{S}^{m+1} \, : \, \, Y_{ef} + Y_{eg} \leq Y_{ee} + Y_{fg}, \, \, \diag(Y) = Y\bold{e}_0 \right\}.
\end{align*}
Let $\mathcal{P}_{\mathcal{H}_{efg}}(M)$ denote the projection of a matrix $M \in \mathcal{S}^{m+1}$ onto $\mathcal{H}_{efg}$. This projection can be written explicitly as stated in the following lemma.
\begin{lemma} \label{LemmaProjectionBQP}
Let $\hat{M} := \mathcal{P}_{\mathcal{H}_{efg}}(M)$ be the projection of a matrix $M \in \mathcal{S}^{m+1}$ onto $\mathcal{H}_{efg}$.
If  $M_{ef} + M_{eg} \leq \frac{M_{ee} + 2M_{0e}}{3} + M_{fg}$, then
\begin{align*}
\hat{M}_{st} = \begin{cases} \frac{1}{3}M_{ee} + \frac{2}{3}M_{0e} & \text{if $(s,t) \in \{(0,e), (e,0), (e,e)\}$,} \\
M_{st} & \text{otherwise.}
\end{cases}
\end{align*}
If  $M_{ef} + M_{eg} > \frac{M_{ee} + 2M_{0e}}{3} + M_{fg}$, then the projection $\hat{M}$ can be written explicitly as:
\begin{align*}
\hat{M}_{st} = \begin{cases} \frac{1}{11}M_{ee} + \frac{2}{11}M_{0e} + \frac{3}{11}M_{fg} + \frac{8}{11}M_{ef} - \frac{3}{11}M_{eg} & \text{if $(s,t) \in \left\{ \begin{aligned} (e,f),(f,e) \end{aligned} \right\}$,}\\
\frac{1}{11}M_{ee} + \frac{2}{11} M_{0e} + \frac{3}{11}M_{fg} - \frac{3}{11}M_{ef} + \frac{8}{11}M_{eg} & \text{if $(s,t) \in \left\{ \begin{aligned} (e,g),(g,e) \end{aligned} \right\}$,} \\
-\frac{1}{11}M_{ee} - \frac{2}{11}M_{0e} + \frac{8}{11}M_{fg} + \frac{3}{11}M_{ef} + \frac{3}{11}M_{eg} & \text{if $(s,t) \in \{(f,g), (g,f)\}$,} \\
\frac{3}{11}M_{ee} + \frac{6}{11}M_{0e} - \frac{2}{11}M_{fg} + \frac{2}{11}M_{ef} + \frac{2}{11}M_{eg} & \text{if $(s,t) \in \left\{ \begin{aligned} (0,e), (e,0), (e,e)  \end{aligned} \right\}$,} \\
M_{st} & \text{otherwise.}
\end{cases}
\end{align*}
\end{lemma}
\begin{proof}
See Appendix \ref{AppendixProofLemma}.
\end{proof}

\subsection{Semi-Parallel Dykstra's projection algorithm} \label{SubsectionDykstra}

A reasonable argument for the fact that a cutting plane technique in an alternating direction augmented Lagrangian approach has never been considered before, is the increasing complexity of the involved projections. In our case, it requires a projection onto the intersection of $\mathcal{Y}$, see (\ref{DefCaligraphicY}), and a finite collection of polyhedra $\mathcal{H}_{efg}$. This can be performed in an iterative approach based on Dykstra's projection algorithm \cite{Dykstra, BoyleDykstra}. Although there exist some similarities between the \textsc{ADMM} and Dykstra's algorithm, see \cite{Tibshirani}, we are the first that combine both methods to compute SDP bounds.

Finding the projection onto the intersection of polyhedra or general convex sets is a well-known problem  for which multiple algorithms have been proposed. For a detailed background on projection methods, we refer the reader to \cite{BauschkeKoch, Cegielski}. Bauschke and Koch \cite{BauschkeKoch} compare several projection algorithms for problems motivated by road design and conclude that Dykstra's cyclic algorithm performs best for projections onto the intersection of convex sets. The idea behind Dykstra's algorithm is to iteratively project a deflected version of the previous iterate onto the individual sets. This method was first proposed by Dykstra \cite{Dykstra} for closed convex cones in finite-dimensional Euclidean spaces and later generalized to closed convex sets in Hilbert spaces by Boyle and Dykstra \cite{BoyleDykstra}.\\ \\
We are interested in the following best approximation problem:
\begin{align} \label{ProjectionProblem}
\min \quad & \Vert \hat{M} - M \Vert_F^2 \quad \text{ s.t. } \quad \hat{M} \in  \mathcal{Y}_\mathcal{T} := \left( \bigcap_{(e,f,g) \in \mathcal{T}} \mathcal{H}_{efg} \right) \cap \mathcal{Y},
\end{align}
where $M$ is the matrix that we  project onto $\mathcal{Y}_{\mathcal{T}}$. Observe that $\mathcal{Y}_\mathcal{T} = \mathcal{C}(\mathcal{T})\cap \mathcal{Y}$.

Dykstra's algorithm starts by initializing the so-called normal matrices $R^0_\mathcal{Y} = \bold{0}$ and $R^0_{efg} = \bold{0}$ for all $(e,f,g) \in \mathcal{T}$. Now, we set $X^0 = M$ and iterate for $k \geq 1$:
\begin{align} \label{AlgCyclicDykstra} \tag{CycDyk}
\begin{aligned}
\begin{aligned}
X^k & := \mathcal{P}_\mathcal{Y} \left( X^{k-1} + R_\mathcal{Y}^{k-1} \right) \\
R_{\mathcal{Y}}^k & := X^{k-1} + R_\mathcal{Y}^{k-1} - X^k \end{aligned}  \,\, \quad \, & \\
\left.
\begin{aligned}
L_{efg} & := X^k + R_{efg}^{k-1} \\
X^k & := \mathcal{P}_{\mathcal{H}_{efg}} \left( L_{efg} \right) \\
R^k_{efg} & := L_{efg} - X^k
\end{aligned} \quad \right\} & \quad \text{for all } (e,f,g) \in \mathcal{T}
\end{aligned}
\end{align}
Several authors have shown that the sequence $\{X^k\}_{k\geq 1}$ strongly converges to the solution of the best approximation problem (\ref{ProjectionProblem}), see \cite{BoyleDykstra, Han, GaffkeMathar}. Since the polyhedra are considered in a cyclic order, the iterates (\ref{AlgCyclicDykstra}) are refered to as Dykstra's cyclic algorithm. Observe that if $\mathcal{T} = \emptyset$, then (\ref{AlgCyclicDykstra}) boils down to a single projection onto $\mathcal{Y}$.\\ \\
Instead of projecting on each polyhedron one after another, it is also possible to project on all polyhedra simultaneously.
This method is referred to as Dykstra's parallel algorithm. We refer the interested reader to Appendix \ref{AppendixParallelDykstra} for an implementation and some details of this parallel  version.
Although the parallel version takes longer to converge in our case, the projections can be done simultaneously, which might be beneficial if used on parallel machines. Preliminary experiments show that in our cutting plane setting the parallel version,  not implemented on parallel machines, is not able to improve on the cyclic version.  \\

\noindent To make  (\ref{AlgCyclicDykstra})  efficient, we (partly) parallelize the algorithm.
Note that  a projection onto $\mathcal{H}_{efg}$ only concerns the entries $(e,f), (e,g), (f,g), (e,e)$ and $(0,e)$. Hence, if two projections onto $\mathcal{H}_{e_1f_1g_1}$ and $\mathcal{H}_{e_2f_2g_2}$ take place one after another and $\{e_1, f_1, g_1\} \cap \{e_2, f_2, g_2\} = \emptyset$, they can in fact be performed simultaneously.
We partition the triples in $\mathcal{T}$ into $r$ clusters $C_i$, $i = 1, ..., r$, such that $C_1 \cup ... \cup C_r = \mathcal{T}$ and $C_i \cap C_j = \emptyset$ for all $i,j$. By doing so, an iterate of (\ref{AlgCyclicDykstra}) is performed in $r + 1$ consecutive steps, instead of $|\mathcal{T}| + 1$ consecutive steps. More details about this clustering step are given in Section \ref{SubsubsectionClustering}. This provides a semi-parallel implementation of (\ref{AlgCyclicDykstra}). \\

\noindent We  take the following actions to further accelerate the algorithm:
\begin{itemize}
\item All matrices in (\ref{AlgCyclicDykstra}) are symmetric, hence we save memory by only working with the upper triangular part of the matrices;
\item The normal matrices $R_{efg}^k$ for all $(e,f,g) \in \mathcal{T}$ are very sparse, i.e., only the entries $(e,f), (e,g),$ $(f,g), (e,e)$ and $(0,e)$ may be nonzero. Therefore, we work with normal vectors corresponding to the nonzero elements in $R_{efg}^k$ instead of using full $(m+1) \times (m+1)$ matrices.
     This  has the additional advantage that the memory needed does not increase with the size of the instance;

\item The projection onto $\mathcal{Y}$ is considerably more costly than the projection onto the triangle inequalities in terms of computation time. Instead of performing all separate projections exactly once and iterate, numerical tests show that the convergence is accelerated if we perform the projection onto $\mathcal{Y}$ only occasionally. That is, after the projection onto $\mathcal{Y}$ we perform the $|\mathcal{T}|$ triangle inequality projections $K$ times in a cyclic order before we again project onto $\mathcal{Y}$. \label{ParameterK}
\end{itemize}

\subsection{A cutting plane augmented Lagrangian method} \label{sect:cutPlane}
In this section we combine the \textsc{PRSM} discussed in Sections \ref{SubsectionADMM} and \ref{sect:preliminarCompare} with the projection method discussed in Sections \ref{SubsectionBQP} and \ref{SubsectionDykstra}. This leads to a cutting plane augmented Lagrangian method (\textsc{CP-ALM}).
To the best of our knowledge, no such algorithm exists for solving SDP problems. \\ \\
In the \textsc{CP-ALM}, we iteratively solve $(SDP_{S3})$ for a set of cuts $\mathcal{T}$ using the \textsc{PRSM}. Each time the \textsc{PRSM} has converged up to some precision, we evaluate the solution for violated cuts and add the $numCuts$ most violated ones to $\mathcal{T}$, where $numCuts$ is a predefined parameter, and repeat.
An advantage of using the \textsc{PRSM} in a cutting plane approach, as opposed to an interior point method, is that after the addition of new cuts we can start the new \textsc{PRSM} loop from the last obtained triple $(Z^k, Y^k, S^k)$. In other words, we exploit the use of warm starts, which speeds up the convergence.

The \textsc{CP-ALM} is provided in Algorithm \ref{AlgCuttingPlaneALM}. In the sequel, we explain several ingredients of the algorithm in more detail.

\begin{algorithm}[H]
\footnotesize
\caption{\textsc{CP-ALM}}\label{AlgCuttingPlaneALM}
\begin{algorithmic}[1]
\Require $\varepsilon_{PRSM}, \varepsilon_{stag}, \varepsilon_{proj}, maxIter, maxTotalIter, maxStagIter$
\State Compute $\widetilde{W}$ by Algorithm \ref{AlgW} and perform a QR-decomposition on $\widetilde{W}$ to obtain $W$.
\State Set $Y^0 = \bold{0}$, $Z^0 = \bold{0}$, $S^0 = \bold{0}$, $k = 0$ and $\mathcal{T} = \emptyset$.
\While {stopping criteria not met} \Comment{See Section \ref{SubsubsectionStopping}}
\While {stopping criteria not met} \Comment{See Section \ref{SubsubsectionStopping}}
\State Update $Z^{k+1} := \mathcal{P}_{\mathcal{S}^{m+1 - \alpha}_+} \left(W^\top \left( Y^k + \frac{1}{\beta} S^k\right) W \right)$.
\State $S^{k + \frac{1}{2}} := S^k + \gamma_1 \cdot \beta \cdot (Y^k - WZ^{k+1}W^\top )$.
\State Update $Y^{k+1} = \mathcal{P}_{\mathcal{Y}_\mathcal{T}}\left( WZ^{k+1}W^\top - \frac{\hat{Q} + S^{k + \frac{1}{2}}}{\beta}\right)$ by solving (\ref{ProjectionProblem}) using semi-parallel (\ref{AlgCyclicDykstra}). \label{LineDykstra}
\State Update $S^{k+1} := S^{k + \frac{1}{2}} + \gamma_2 \cdot \beta \cdot (Y^{k+1} - WZ^{k+1}W^\top )$.
\State $k \gets k+1$.
\EndWhile
\State Identify the violated triangle inequalities and add the $numCuts$ most violated cuts to $\mathcal{T}$.
\State Cluster the cuts in $\mathcal{T}$ into sets $C_1, ..., C_r$. \Comment{See Section \ref{SubsubsectionClustering}}
\EndWhile
\State Compute $LB(S^k)$ using the final dual variable $S^k$.  \Comment{See Section \ref{SubsubsectionLB}}
\Ensure $LB(S^k)$
\end{algorithmic}
\end{algorithm}

\subsubsection{Stopping criteria} \label{SubsubsectionStopping}
The inner while-loop of Algorithm \ref{AlgCuttingPlaneALM} constructs a \textsc{PRSM} sequence for a fixed $\mathcal{T}$. Experiments show that the algorithm is stabilized if, as opposed to adding many cuts at once, we add cuts smoothly in order to keep the residuals small. Hence, we want the inner \textsc{PRSM} sequence to converge before adding new cuts to $\mathcal{T}$. We consider three types of stopping criteria for the inner while-loop:
\begin{enumerate}
\item Let $\varepsilon_{PRSM} > 0$ be a predefined tolerance parameter. The inner while-loop is terminated after iteration $k$ if
\begin{align*}
\min \Big( \Vert Y^{k+1} - WZ^{k+1}W^\top \Vert_F \, , \,\, \beta \Vert W^\top \left( Y^{k+1} - Y^k \right) W \Vert_F \Big) < \varepsilon_{PRSM}
\end{align*}
The first term on the left hand side measures primal feasibility, while the second term measures dual feasibility.
\item We stop when a fixed number of iterations $maxIter$ is reached.
\item We add a stagnation criterion. Let $\varepsilon_{stag} > 0$ be a tolerance parameter. We introduce a variable $stagIter$ that is increased by one each time we have $| \langle Y^{k+1},\hat{Q} \rangle - \langle Y^k, \hat{Q} \rangle | < \varepsilon_{stag}$. We stop the inner while-loop whenever $stagIter > maxStagIter$ for some predefined integer $maxStagIter$.
\end{enumerate}
The cyclic Dykstra algorithm in line \ref{LineDykstra} of Algorithm \ref{AlgCuttingPlaneALM} is stopped whenever $\Vert X^{k+1} - X^k \Vert_F < \varepsilon_{proj}$ for some predefined $\varepsilon_{proj} > 0$.

Finally, the outer while-loop, i.e., the cutting plane part, is stopped whenever no more violated cuts can be found or after a predefined number of iterations $maxTotalIter > maxIter$ has been reached.

\subsubsection{Clustering} \label{SubsubsectionClustering}
As explained in Section \ref{SubsectionDykstra}, the cyclic Dykstra algorithm can be partially parallelized by partitioning the set $\mathcal{T}$ into $r$ clusters of non-overlapping cuts. We explain here how this clustering is done.

Let $H = (V,E)$ denote a graph where each node $i \in V$ represents a cut in $\mathcal{T}$ and two nodes are connected by an edge whenever the corresponding cuts are overlapping. Clustering $\mathcal{T}$ into the smallest number of non-overlapping sets is then equivalent to finding a minimum coloring in $H$.
This problem is known to be $\mathcal{NP}$-hard. Galinier and Hertz \cite{GalinierHertz} provide an overview of graph coloring heuristics, where it is concluded that the Tabucol algorithm of Hertz and De Werra \cite{HertzDeWerra} is overall very successful. We implement here the improved Tabucol algorithm provided in \cite{GalinierHao}.

\subsubsection{Lower bound} \label{SubsubsectionLB}
After each \textsc{CP-ALM} iterate $k$, we obtain a triple $(Z^{k}, Y^{k},S^{k})$ which allows us to compute $\langle \hat{Q}, Y^k \rangle$. Although this value converges to the optimal solution of the SDP relaxation  $(SDP_{S3})$, the convergence is typically not monotonic, which implies that this value does not necessarily provide a lower bound for the \textsc{QCCP} instance. We can still use the output of the \textsc{CP-ALM} to obtain a lower bound. Various methods for obtaining lower bounds from approximate solutions have been proposed in the literature \cite{Eckstein, JanssonEtAl, OliveiraEtAl}. We adopt here the method introduced by Oliveira et al.\ \cite{OliveiraEtAl}.

Let $\mathcal{W^\top SW} := \{S \, : \, \, W^\top S W \preceq 0 \}$. Then, a lower bound is obtained by solving:
\begin{align} \label{trueLB}
LB(S^k) := \min_{Y \in \mathcal{Y}_{\mathcal{T}}} \langle \hat{Q} + \mathcal{P}_{\mathcal{W^\top S W}}(S^k), Y \rangle,
\end{align}
where $\mathcal{P}_{\mathcal{W^\top S W}}(S^k)$ is the projection of $S^k$ onto the set $\mathcal{W^\top SW}$.
This projection can be performed efficiently, see \cite{OliveiraEtAl}. Moreover, note that (\ref{trueLB}) is a linear programming problem.

\section{Upper bounds} \label{SectionUpper}

The matrices resulting from the \textsc{CP-ALM} can be used to construct upper bounds for the \textsc{QCCP}.
 In this section we derive several upper bounding approaches, among which a deterministic method, two randomized algorithms and a Q-learning algorithm that is based on reinforcement learning.
 We are not aware of other SDP-based rounding algorithms that make use of reinforcement learning.
 We end the section by providing a hybrid approach that combines all aforementioned heuristics.

\subsection{Best Euclidean approximation} \label{SubsectionTB}
Let $(Z^{out}, Y^{out}, S^{out})$ be the outcome of the \textsc{CP-ALM}. Throughout the entire section we assume that the \textsc{CP-ALM} is solved up to high precision in order for the utilized results to be valid. Let $x^{out}$ be the vector consisting of the diagonal elements of $Y^{out}$ excluding the first entry.
As $x^{out}$ is  an approximation of the optimal cycle cover, one can search for the vector $x \in P$ that is closest to $x^{out}$ in Euclidean norm.
  This vector can be obtained as follows:
\begin{align}\label{transport}
x^* := \arg\max \left \{ x^\top x^{out} \, : \, \, x \in \Conv(P) \right\}.
\end{align}
The corresponding upper bound is $UB_{EB} := (x^*)^\top Q x^*$.

\subsection{Randomized Undersampling}

Randomized SDP-based heuristics have proven to be successful for various optimization problems, mainly sparked by the seminal work of Goemans and Williamson \cite{GoemansWilliamson}.
A widely used procedure in the design of approximation algorithms is randomized rounding \cite{Raghavan},
which rounds a relaxed solution to a solution for the original problem that is close to optimal in expectation.
We present an SDP-based randomized rounding algorithm that we refer to as randomized undersampling.\\\\
Let $x^{out} \in \mathbb{R}^m$ be as discussed in Section \ref{SubsectionTB}. Observe that since all entries of $x^{out}$ are non-negative and $\sum_{e \in \delta^+(i)}x_e^{out} = 1$, see (\ref{setP}), we can view $\{x_e\}_{e \in \delta^+(i)}$ as a probability distribution on all arcs leaving node $i$. Similarly, $\{x_e\}_{e \in \delta^-(i)}$ represents a probability distribution on the set of arcs entering node $i$. Hence, for each node $i$ we can draw exactly one arc from $\delta^+(i)$ according to the distribution $\{x_e\}_{e \in \delta^+(i)}$. Let $y_1 \in \{0,1\}^m$ denote the characteristic vector of the outcome of these $n$ trials. We do the same for the incoming arcs, yielding a vector $y_2 \in \{0,1\}^m$. By construction we have $Uy_1 = Vy_2 = \bold{1}_n$, but not necessarily $Vy_1 = Uy_2 = \bold{1}_n$.

The vector $y = y_1 \circ y_2$ denotes a partial cycle cover that satisfies $Uy \leq \bold{1}_n$ and $Vy \leq \bold{1}_n$. Observe that the probability of including arc $e$ in $y$ equals $x_e^2$. To extend $y$ to a feasible cycle cover, we define:
\begin{align}
\begin{aligned}
N^+  := \{i \in N \, : \, \, y_e = 0 \quad  \forall e \in \delta^+(i) \} \quad \text{and} \quad
N^-  := \{i \in N \, : \, \, y_e = 0 \quad  \forall e \in \delta^-(i) \}.
\end{aligned} \label{NplusNminus}
\end{align}
We still have to select exactly one arc from $\delta^+(i)$ for all $i \in N^+$ and one arc from $\delta^-(i)$ for all $i \in N^-$ to extend $y$ to a feasible cycle cover. We can do this by solving a modified version of (\ref{transport}). Let $U_{N^+} \in \mathbb{R}^{|N^+| \times m}$ (resp.\ $V_{N^-} \in \mathbb{R}^{|N^-| \times m}$) denote the submatrix of $U$ (resp.\ $V$) induced by the rows corresponding to $N^+$ (resp.\ $N^-$).
Let us define the following vector:
\begin{align}
\bar{x}^{out}_e := \begin{cases} -\infty & \text{if $e^- \in N \setminus N^-$ or $e^+ \in N \setminus N^+$,} \\
x^{out}_e & \text{otherwise,}
\end{cases} \label{barx}
\end{align}
where some values are set to $-\infty$ in order to avoid in- or outflows larger than one. We now solve
\begin{align} \label{PartialToFull}
z^* := \argmax_{z \in \mathbb{R}^m} \left\{ z^\top \bar{x}^{out} \, : \, \, U_{N^+}z = \bold{1}_{|N^+|}, V_{N^-}z = \bold{1}_{|N^-|}, \bold{0}_m \leq z \leq \bold{1}_m \right\}.
\end{align}
A partial solution $y$ can be extended to a feasible cycle cover if and only if the optimal value to (\ref{PartialToFull}) is finite. Indeed, in that case we have $y + z^* \in P$, which yields the bound $UB_{US} = (y+z^*)^\top Q(y + z^*)$. We now repeat this procedure and store the smallest obtained bound.

As we select at most $n$ arcs at random and extend the solution to a full cycle cover, we call this method randomized undersampling. The steps of this method are summarized in Algorithm \ref{AlgUR}.

\begin{algorithm}[H]
\footnotesize
\caption{Randomized Undersampling for the \textsc{QCCP}}\label{AlgUR}
\begin{algorithmic}[1]
\Require $G, Q, x^{out}$
\State Initialize $y_1 = \bold{0}_m$ and $y_2 = \bold{0}_m$.
\For {$i \in N$}
\State Draw $f_1$ from $\delta^+(i)$ with respect to $\{x^{out}_e\}_{e \in \delta^+(i)}$ and $f_2$ from $\delta^-(i)$ with respect to $\{x^{out}_e\}_{e \in \delta^-(i)}$.
\State Set $y_1(f_1) = 1$ and $y_2(f_2) = 1$.
\EndFor
\State $y \gets y_1 \circ y_2$.
\State Obtain the sets $N^+$ and $N^-$ and the vector $\bar{x}^{out} \in \mathbb{R}^m$ as in (\ref{NplusNminus}) and (\ref{barx}), respectively.
\If {problem (\ref{PartialToFull}) has a finite objective value}
\State $UB_{US} \gets (y + z^*)^\top Q(y + z^*)$ where $z^*$ is computed by (\ref{PartialToFull}).
\Else \State Go back to \textsc{Step 2} \EndIf
\Ensure $UB_{US}$
\end{algorithmic}
\end{algorithm}

\subsection{Randomized Oversampling}

Instead of sampling a partial solution and deterministically extend it to a full cycle cover, we can also randomly add arcs to a subgraph $H$ of $G$ until it contains a cycle cover visiting all nodes. We call this method randomized oversampling. \\ \\
We initialize $H = (N, \emptyset)$ and iteratively add pairs of successive arcs to $H$. This is done randomly using a probability distribution on the set $\delta^-(i) \times \delta^+(i)$ for all $i \in N$. We use a rank-one approximation of $Y^{out}$ for the sake of finite convergence, see Lemma \ref{LemmaOversampling} below.

The best rank-one approximation of $Y^{out}$ is given by $\lambda_{\max} ww^\top$, where $\lambda_{\max}$ and $w \in \mathbb{R}^{m+1}$ are the corresponding Perron-Frobenius eigenvalue and eigenvector, respectively. Let $w_0$ denote the zeroth entry of $w$ and let $\bar{w} \in \mathbb{R}^m$ be the vector obtained by excluding $w_0$ from $w$. It follows from the Perron-Frobenius theorem that $w$ can be chosen such that it has nonnegative entries.  Since the vectors $(-1, u_i^\top)^\top$ and $(-1, v_i^\top)^\top$ are eigenvectors of $Y^{out}$ associated with the eigenvalue zero, see Lemma \ref{LemmaNotSlater}, it follows that
\begin{align*}
u_i^\top \bar{w} = w_0 \quad \text{and} \quad v_i^\top \bar{w} = w_0 \quad \text{for all $i \in N$.}
\end{align*}
Suppose that $w_0 = 0$. Then $u_i^\top \bar{w} = v_i^\top \bar{w} = 0$ for all $i \in N$, which implies that $w$ only contains zeros. Since this contradicts with the fact that $||w|| > 0 $, we have $w_0 > 0$.

Now, let $r \in \mathbb{R}^m$ be defined as $r := \frac{1}{w_0}\bar{w}$.
Since $u_i^\top r = v_i^\top r = 1$ for all $i \in N$ and $r \geq \bold{0}$, we conclude that $r$ is contained in the directed 2-factor polytope. Hence, we can view $\{ r_e \cdot r_f \}_{(e,f) \in \delta^-(i) \times \delta^+(i)}$ as a probability distribution on the pairs of successive arcs for all $i \in N$.

The oversampling algorithm, see Algorithm \ref{AlgOR}, iteratively draws a pair of successive arcs $(e,f)$ around  $i \in N$ according to the distribution implied by $r$ and adds this pair to $H$. We repeat this until $H$ contains a cycle cover. The best among possibly multiple cycle covers in $H$ is obtained by solving  problem  \eqref{transport} with respect to $x^{out}$ restricted to the arcs in $H$.

\begin{algorithm}[H]
\footnotesize
\caption{Randomized Oversampling for the \textsc{QCCP}}\label{AlgOR}
\begin{algorithmic}[1]
\Require $G, Q, Y^{out}, x^{out}$
\State Obtain Perron-Frobenius eigenpair $(w, \lambda_{\max})$ of $Y^{out}$ and let  $r = \frac{1}{w_0}\bar{w}$.
\State Let $H = (N, \emptyset)$ be the empty subgraph of $G$.
\While {$H$ contains no directed 2-factor}
\For {$i \in N$}
\State Select a pair $(e,f)$ according to probability distribution $\{r_e \cdot r_f \}_{(e,f) \in \delta^-(i) \times \delta^+(i)}$. Set $H \gets H \cup \{e, f\}$
\EndFor
\EndWhile
\State Solve \eqref{transport} with respect to $x^{out}$ restricted to $H$, and compute the corresponding upper  bound $UB_{OS}$.
\Ensure $UB_{OS}$
\end{algorithmic}
\end{algorithm}
\noindent We can prove the following result with respect to the termination of Algorithm \ref{AlgOR}.

\begin{lemma} Algorithm \ref{AlgOR} terminates in a finite number of steps with high probability. \label{LemmaOversampling}
\end{lemma}

\begin{proof} See Appendix \ref{AppendixProofOversampling}.
\end{proof}

\subsection{Sequential Q-learning}

The final rounding approach we propose is based on a distributed reinforcement learning (RL) technique, namely Q-learning \cite{Watkins}. Q-learning is a branch of machine learning in which artificial agents learn how to take actions in order to maximize an expected total reward. Recently, RL techniques have shown successful in deriving good feasible solutions for combinatorial optimization problems, see e.g., \cite{BarrettEtAl}.
We propose here an algorithm in which a set of agents learn how to find (near-)optimal cycles in $G$ by exploiting  our SDP relaxation.
Our sequential Q-learning algorithm (SQ-algorithm) is inspired by the work of Gambardella and Dorigo \cite{GambardellaDorigo} and exploits the solution of the \textsc{CP-ALM} within the learning process. \\

\noindent In the sequential Q-learning algorithm we introduce $n$ agents each having the independent task to construct a set of node-disjoint cycles. This is done iteratively by adding nodes to the agent's current path until the path contains a directed cycle or no more nodes can be added. For each agent $k = 1, ..., n$, let $P_k$ denote its current path and let $c_k$ and $p_k$ denote the current node and its predecessor on the agent's search, respectively. Besides, let $J_k$ be the set of nodes that is not placed on a cycle by agent $k$. In each iteration, the successor $s_k$ of $c_k$ is selected among one of the nodes in $N^+(c_k) \cap J_k$, where $N^+(c_k)$ is the set of nodes reachable from $c_k$ via a single arc, based on a matrix $SQ \in \mathbb{R}^{m \times m}$. This matrix indicates on position $(e,f)$ how useful it is to traverse an arc $f$ after an arc $e$. We select the successor $s_k$ that leads to a high $SQ((p_k,c_k),(c_k,s_k))$-value and add it to $P_k$. If the addition of $s_k$ to $P_k$ does not result in a cycle, we set the current node $c_k$ to be $s_k$. If the addition of $s_k$ does lead to a cycle $C_k \subseteq P_k$, we memorize this cycle into the agent's partial solution vector $y_k \in \{0,1\}^m$ and set $c_k$ to one of the nodes not yet on a cycle. An agent's search terminates whenever no new successor can be found, i.e., $N^+(c_k) \cap J_k = \emptyset$, or when $y_k$ is a full cycle cover. If one of these events occurs, we deacivitate the agent. We repeat the steps above for all active agents, until all agents have been deactivated. This results in $n$ vectors $y_k$ that represent sets of node-disjoint cycles, not necessarily full cycle covers. At the end of the cycle-building phase, the (partial) solution $y_k$ that has relative minimum cost is used to update the $SQ$-matrix via delayed reinforcement learning. Now all agents are again activited and a new cycle-building trial starts using the new $SQ$-matrix until certain stopping criteria are satisfied, e.g., after a fixed number of trials. \\ \\
To decide which successor $s_k$ to select for a given $p_k,c_k$ and $J_k$, we define a fit function $f$ that depends on the $SQ$-values and the quadratic costs $Q = (q_{ef})$. The fit of visiting $u \in N^+(c_k) \cap J_k$ after $c_k$ is:
\begin{align*}
f(u \, | \, p_k, c_k, J_k) := \begin{cases} \left[  \underset{e \in \delta^+(J_k, c_k)}{\sum}SQ\left(e,(c_k,u)\right) \right]^\delta \cdot \left[  \underset{e \in \delta^+(J_k,c_k)}{\sum} \frac{1}{q_{e,(c_k,u)} + \epsilon} \right]^\beta & \text{if } p_k = \emptyset, \vspace{0.2cm} \\
\, \left[ SQ\left((p_k,c_k), (c_k,u)\right) \right]^\delta \cdot \left[ \frac{1}{q_{(p_k,c_k),(c_k,u)} + \epsilon}  \right]^\beta & \text{otherwise,}
\end{cases}
\end{align*}
where $\delta, \beta > 0$ \label{ParameterDeltaBeta} are parameters which represent the relative importance between the learned $SQ$-values and the quadratic costs and $\epsilon > 0$ is a small value to deal with quadratic costs that are zero. After computing the fit for all potential successors, we deterministically select the one with the highest fit value or select randomly proportional to their fit values. That is,
\begin{align} \label{successorRL}
s_k = \begin{cases}  \underset{u \in N^+(c_k) \cap J_k}{\arg \max} f(u \, | \, p_k, c_k, J_k) & \text{if } q \leq q_0 \vspace{0.2cm}  \\
\quad S & \text{otherwise, }\end{cases}
\end{align}
where $S$ is a random variable over the set $N^+(c_k) \cap J_k$, where each node is chosen with probability proportional to its fit value. The parameter $q_0 \in [0,1]$  from \eqref{successorRL} is the probability of selecting the successor node deterministically. \\

\noindent The $SQ$-values measure the usefulness of traversing two successive arcs.
Recall that $Y^{out}$ is the output of the \textsc{CP-ALM}.
 As $Y^{out}_{ef}$ is likely to be larger when two arcs $e$ and $f$ are in an optimal solution, we initialize the $SQ$-matrix by setting $SQ(e,f) = Y^{out}_{ef}$ for all pairs of successive arcs $(e,f)$. The $SQ$-update is based on a mixture between local memory and a reinforcement learning, similar to \cite{GambardellaDorigo}:
\begin{align} \label{SQ-update}
SQ(e,f) \gets (1 - \alpha) SQ(e,f) + \alpha \left( \Delta SQ(e,f) + \gamma \max_{g \in \delta^+(f^-, J_k)} SQ(f, g) \right),
\end{align}
where $\alpha, \gamma \in (0,1)$ represent the learning rate and discount factor, respectively. The learning update consists of a discounted reward of the next state and a reinforcement term $\Delta SQ(e,f)$. Similar to the algorithm of \cite{GambardellaDorigo}, we assume that this reinforcement term is zero throughout the cycle-building phase and update it only at the end of a trial. Hence, we only incur a delayed reinforcement term $\Delta SQ(e,f)$. The discounted reward, however, is incorporated during the cycle-building phase.

The delayed reinforcement of a pair of successive arcs $(e,f)$ can be seen as a reward for cost minimal cycles that is obtained at the end of each trial. After all agents are deactivated, each vector $y_k$ is the characteristic vector of a set of node-disjoint cycles. For each agent $k$ that constructed at least one cycle, we compute $L_k := (y_k^\top Q y_k) / (\bold{1}^\top y_k)$, i.e., the relative cost per arc in $y_k$. Let $k_{best}$ denote the agent that constructed the solution with the smallest value $L_k$, and let $L_{best}$ denote its relative cost per arc. Then $\Delta SQ(e,f)$ is computed as:
\begin{align} \label{delayedRL}
\Delta SQ(e,f) = \begin{cases} \frac{\Omega}{L_{best}} & \text{if $(e,f)$ is a pair of successive arcs in $y_{k_{best}}$,} \\
0 & \text{otherwise,}
\end{cases}
\end{align}
where $\Omega$ is a constant.\\ \\
We let the SQ-algorithm run until some fixed number of trials has passed. All cycles that have been constructed throughout the entire algorithm are stored in memory. Let $\Gamma$ denote the number of distinct cycles that are constructed and define the matrix $B \in \mathbb{R}^{n \times \Gamma}$ as follows:
\begin{align*}
B_{i,k} = \begin{cases} 1 & \text{if node $i$ is on cycle $k$,} \\
0 & \text{otherwise.}
\end{cases}
\end{align*}
Let $b \in \mathbb{R}^\Gamma$ be the vector containing the quadratic cost of each cycle. Then the best upper bound based on our SQ-algorithm is obtained by solving the following set partitioning problem (\textsc{SPP}):
\begin{align} \label{SPPproblem}
\min \left\{b^\top x \, : \, \, Bx = \bold{1}_n \, , \, \, x \in \{0,1\}^\Gamma \right\}.
\end{align}
As the \textsc{SPP} is $\mathcal{NP}$-hard, computing an optimal solution to (\ref{SPPproblem}) might be too much to ask for. Instead, an approximate solution to (\ref{SPPproblem}) can be obtained efficiently, e.g., by using the Lagrangian heuristic of Atamt\"urk et al.\ \cite{AtamturkEtAl} which is able to compute near-optimal or even optimal solutions to (\ref{SPPproblem}) most of the time. For moderate values of $\Gamma$ and $n$, however, current ILP solvers are able to solve (\ref{SPPproblem}) to optimality in a very short time.

A pseudocode of the SQ-algorithm is provided in Algorithm \ref{AlgSQR}.

\algnewcommand{\IIf}[1]{\State\algorithmicif\ #1\ \algorithmicthen}
\algnewcommand{\EndIIf}{\unskip\ \algorithmicend\ \algorithmicif}

\begin{algorithm}[h]
\footnotesize
\caption{Sequential Q-learning for \textsc{QCCP}}\label{AlgSQR}
\begin{algorithmic}[1]
\Require $G, Q, Y^{out}$
\State For all pairs of successive arcs $(e,f)$, initialize $SQ(e,f) = Y_{ef}^{out}$.
\State For all agents $k = 1, ..., n$, initialize the starting node $c_k = k \in N$, starting edge $e_k = \emptyset$ and $J_k = N$. Set $P_k = \emptyset$ and $y_k  = \bold{0}_m$ and activate all agents.
\While {there is at least one active agent}
\vspace{0.2cm}
\For {all active agents $k$}
\State Update the fit function $f(u \, | \, p_k, c_k, J_k)$ for all $u \in N^+(c_k) \cap J_k$ and obtain $s_k$ according to (\ref{successorRL}).
\State Add $s_k$ to $P_k$.
\EndFor
\vspace{0.2cm}
\For {all active agents $k$ with $p_k \neq \emptyset$}
\State $SQ((p_k, c_k),(c_k,s_k)) \gets (1 - \alpha) SQ((p_k, c_k), (c_k,s_k)) + \alpha \gamma \max_{e \in \delta^+(s_k, J_k)} SQ((c_k,s_k),e)$.
\EndFor
\vspace{0.2cm}
\For {all active agents $k$}
\If {$P_k$ contains a cycle $C_k$}
\State Set $J_k \gets J_k \setminus C_k$ and $P_k \gets \emptyset$
\State Set $(y_k)_e = 1$ for all arcs $e$ in $C_k$.
\If {$J_k = \emptyset$}
\State Deactivate agent $k$.
\Else
\State Set $p_k \gets \emptyset$ and choose $c_k$ uniformly at random out of $J_k$.
\EndIf
\Else
\State Set $p_k \gets c_k$ and $c_k \gets s_k$.
\IIf {$N^+(c_k) \cap J_k = \emptyset$}  Deactivate agent $k$  \EndIIf
\EndIf
\EndFor
\vspace{0.2cm}
\EndWhile
\For {all pairs of successive arcs $(e,f)$}

\State Compute the delayed reinforcement $\Delta SQ(e,f)$ according to (\ref{delayedRL}).
\State $SQ(e,f) \gets (1 - \alpha) SQ(e,f) + \alpha \Delta SQ(e,f)$
\EndFor
\State If stopping criteria are met, obtain $UB_{SQ}$ using (\ref{SPPproblem}). Otherwise,  go to \textsc{Step 2}.
\Ensure $UB_{SQ}$
\end{algorithmic}
\end{algorithm}

\subsection{Hybrid upper bounding algorithm}
The design of the SQ-algorithm discussed in the previous section gives rise to a straightforward hybrid implementation of all above-mentioned upper bounding approaches. Indeed, by adding all cycles that have been created by the best Euclidean approximation, the undersampling and the oversampling algorithm to the matrix $B$ and solve or approximate the corresponding \textsc{SPP}, a hybrid upper bound $UB_{HY}$ is obtained which provably outperforms any independent implementation of the mentioned upper bounds.

\section{Computational Results} \label{SectionNumerics}
We now test the introduced SDP-based lower and upper bounds on several sets of instances and compare them to various bounds from the literature.

This section is organized as follows: we start by introducing the test sets and the parameter settings that we consider. After that, the performance of the lower and upper bounds are discussed in Section \ref{SubsectionResultsLower} and \ref{SubsectionResultsUpper}, respectively.

\subsection{Design of numerical experiments}
The SDP bounds that we take into account are $(SDP_{S2})$ and $(SDP_{S3})$, which we obtain via the \textsc{PRSM} and the \textsc{CP-ALM}, respectively. The \textsc{CP-ALM} is implemented as presented in Algorithm \ref{AlgCuttingPlaneALM}, i.e., using the \textsc{PRSM} and Dykstra's semi-parallel projection algorithm in the subproblem. We present results for different number of added cuts. To compare our SDP bounds, we use the first level RLT bound ($RLT1$), see Adams and Sherali \cite{AdamsSherali1986, AdamsSherali1990}, the MILP-based bound ($MILP$) and the linearization based bound ($LBB1$) from \cite{DeMeijerSotirov}. This latter bound is currently the best bound from the literature when taking both quality and efficiency into account. Since upper bounds for the \textsc{QCCP} are never considered before, we present and compare the upper bounds introduced in Section \ref{SectionUpper}.

All lower and upper bounds are implemented in Matlab on a PC with an Intel(R) Core(TM) i7-8700 CPU, 3.20 GHz and 8 GB RAM. The linear programming problems appearing in our approaches and in the computation of $MILP$, $LBB1$ and $RLT1$ are solved using CPLEX 12.7.1. All computation times reported in this section concern wall-clock times. \\ \\
We test our bounds on three sets of instances:
\begin{itemize}
\item \textbf{Reload instances:} The reload instances are the same as the ones used in Rostami et al.\ \cite{RostamiEtAl} for the \textsc{QTSP} and are based on a similar setting from Fischer et al.\ \cite{AandFFischer}. The underlying graph is the complete directed graph on $n$ nodes. The quadratic costs are based on the reload model \cite{WirthSteffan}, where each arc is randomly assigned a color from a color set $L$. The quadratic costs between two successive arcs with the same color is zero. If successive arcs $e$ and $f$ are assigned distinct colors $s$ and $t$, respectively, the costs equal $r(s,t)$, where $r: L \times L \rightarrow \{1, ..., D\}$ is a reload cost function. The function $r$ is constructed uniformly at random. We consider 60 instances with $n \in \{10, 15,20\}$, $D \in \{1,10\}$ and $|L| = 20$. As preliminary experiments show that the addition of cuts do not significantly improve the bounds, we only compute $(SDP_{S2})$  for these instances.
\item \textbf{Erd\H os-R\' enyi instances:} These instances are based on the $G(n,p)$ model by Erd\H os and R\' enyi \cite{ErdosRenyi}. A graph is constructed by fixing $n$ nodes and including each arc independently with probability $p$. We present two types of cost structures on these instances:
\begin{itemize}
\item \textit{Uniform Erd\H os-R\' enyi instances:} the quadratic cost between any pair of successive arcs is chosen discrete uniformly at random from $\{0, ..., 100\}$;
\item \textit{Reload Erd\H os-R\' enyi instances:} the quadratic cost between any pair of successive arcs is based on a reload cost model using 20 colors and reload costs drawn uniformly from $\{1, ..., 100\}$.
\end{itemize}
We consider 15 instances of each type for $n$ between 20 and 80 and $p$ between 0.3 and 0.5.
\item \textbf{Manhattan instances:} Comellas et al.\ \cite{ComellasEtAl} introduced multidimensional directed grid instances that resemble the street pattern of  cities like New York and Barcelona. Given a set of positive integers $(n_1, ..., n_k)$, the Manhattan instances are constructed as explained in \cite{DeMeijerSotirov}. The quadratic costs between any pair of successive arcs is chosen discrete uniformly at random out of $\{0, ..., 10\}$. We consider a set of 32 Manhattan instances ranging from type $(5,5)$ to type $(9,10,10)$.
\end{itemize}
{The interested reader can download all instances online\footnote{Instances can be found at \url{https://github.com/frankdemeijer/SDPforQCCP}.}.

Numerical results show that Erd\H os-R\' enyi instances and reload instances up to approximately 400 arcs can be solved to optimality within one hour, respectively. The computation limit for Manhattan instances is around 2000 arcs, due to the small density of these graphs. As all costs are integer, we round up all bounds.\\ \\
For the computation of the bounds we need to specify various parameters. The \textsc{PRSM} is implemented using $\beta = \lceil m / n \rceil$, $\gamma_1 = 0.9$ and $\gamma_2 = 1.09$, see (\ref{ZsubproblemPR})--(\ref{SsubproblemPR}), as preliminary experiments show that this setting gives the most stable performance. The \textsc{CP-ALM} uses the same \textsc{PRSM} parameters in the subproblem, where $K = 5$ is used in the semi-parallel implementation of (\ref{AlgCyclicDykstra}), see the third bullet on page \pageref{ParameterK}.
 The stopping criteria of the \textsc{PRSM} and the \textsc{CP-ALM} are as explained in Section~\ref{SubsubsectionStopping}, where we use $\varepsilon_{stag} = 10^{-5}$ and $\varepsilon_{proj} = 10^{-8}$. The parameter $\varepsilon_{PRSM}$ is initially set to $10^{-6}$, but after the addition of cuts increased to $ 10^{-4}$, since solving the $Y$-subproblem using Dykstra's algorithm is significantly slower than the initial $Y$-subproblem without cuts. Hence, we allow for a lower precision. For the same reason, the maximum number of iterations of the inner while-loop (i.e., $maxIter$, see Section~\ref{SubsubsectionStopping}) of the \textsc{CP-ALM} is initialized to some value and decreased when the first cuts are added, after which we do not change it anymore. For the Erd\H os-R\' enyi and Manhattan instances, we initialize $maxIter$ to 1000 and 1500, respectively, and decrease it to 500 after the addition of cuts. The initial iteration limit for the Manhattan instances is larger, as the \textsc{CP-ALM} needs more iterations to converge for these type of instances.

It turns out that the number of cuts added per main loop, i.e., the value of $numCuts$, see Section~\ref{sect:cutPlane}, is of major importance for the quality of the final bound. To demonstrate this behaviour, the lower bounds against the iteration number for a moderate-size Erd\H os-R\'enyi instance (ER\_4 with $n = 35$ and $m =361$) is plotted in Figure~\ref{PlotER4} for various values of $numCuts$ using an iteration limit of 2500. The base line shows the behaviour of the \textsc{PRSM}, i.e., the \textsc{CP-ALM} without the addition of cuts. It is clear that the addition of cuts after 1000 iterations immediately starts improving the bounds. Moreover, as one might expect, the addition of more cuts leads to a higher lower bound, although the largest improvement is due to the addition of the first few cuts. As the addition of more cuts also leads to higher computation times, a trade-off between quality and time has to be made. Based on preliminary experiments, we report results for $numCuts = 50, 150, 300$ and $500$ for the Erd\H os-R\'enyi instances. For the Manhattan instances, we only show results for $numCuts = 300$ and $500$, as the addition of a small number of cuts does not significantly improve the bounds.
\begin{figure}[h] \centering
\includegraphics[scale=0.51]{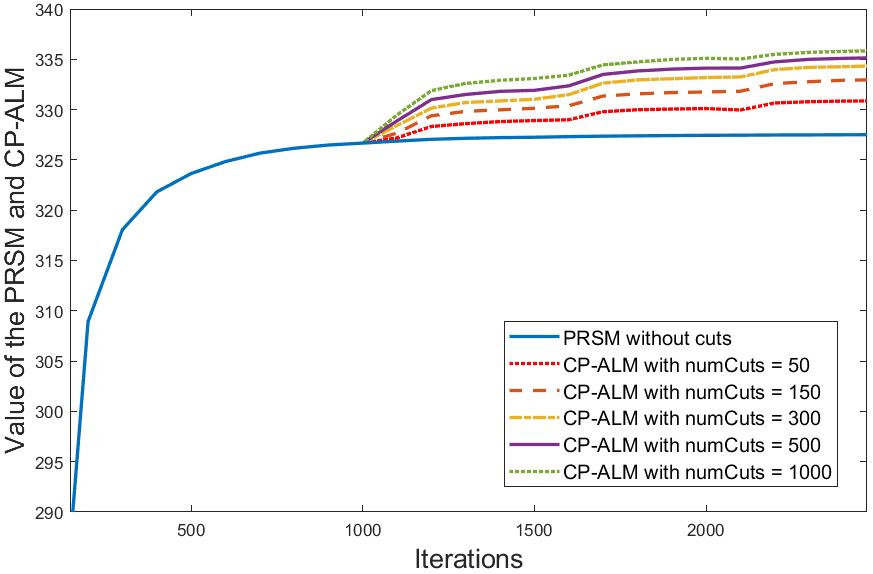}
\caption{Behaviour of the \textsc{PRSM} and the \textsc{CP-ALM} for different values of $numCuts$ for instance ER\_4. \label{PlotER4}}
\end{figure}

Finally, we need to specify the maximum total number of iterations $maxTotalIter$, see Section~\ref{SubsubsectionStopping}. For the reload instances we set this value to 2500 iterations, although the algorithm in most cases terminates earlier for these instances due to the other stopping criteria. The value of $maxTotalIter$ for the other two instance types is based on preliminary tests. Similar to the \textsc{PRSM} and the \textsc{ADMM}, the \textsc{CP-ALM} can suffer from tailing off. Since the addition of more cuts makes later iterations more expensive, one has to decide carefully when to stop. This threshold mainly depends on the value of $m$. Figure~\ref{PlotIterations} shows the behaviour of the lower bounds averaged over $numCuts = 150, 300$ and $500$ on three instances: a small, a moderate-size and a large instance. We normalize the bounds in order to make them comparable, i.e., the plots show the fraction of the final lower bound that is obtained after each iteration. Although at first sight there seems not much difference, one can see from the zoomed image on the right-hand side that the \textsc{CP-ALM} converges relatively faster for smaller instances.

\begin{figure}[h]
\includegraphics[scale=0.49]{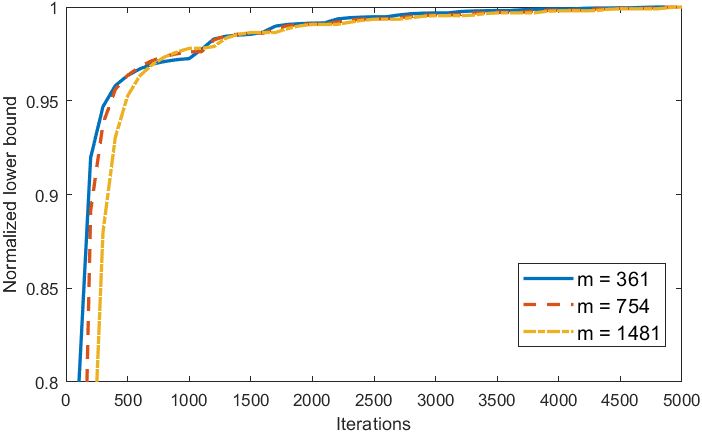}
\includegraphics[scale=0.52]{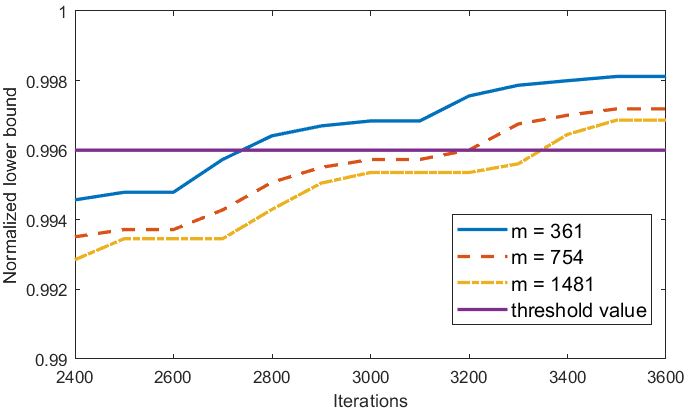}
\caption{Normalized lower bounds (averaged over different number of cuts) for three instances with different numbers of arcs.
Right figure shows zoomed plot including a threshold at 0.996. \label{PlotIterations}}
\end{figure}

\noindent Based on these preliminary results, the parameter $maxTotalIter$ is set to 2500, 3000 or 3500 if $m<500, 500 \leq m < 1000$ and $m \geq 1000$, respectively, for the Erd\H os-R\'enyi instances. For the Manhattan instances these values are 3000, 3500 and 4000, respectively, using the same distinction on $m$. \\ \\
For the computation of upper bounds, we compute the randomized undersampling and oversampling bounds 500 times for each instance and return the best value. For the SQ-algorithm, we use different parameter settings for each instance type based on preliminary tests. It turns out that the algorithm performs best if the value of $\delta$ is significantly larger than the value of $\beta$, see page~\pageref{ParameterDeltaBeta}, i.e., we put more emphasis on the SDP-based SQ-values than on the original quadratic costs. This difference seems more beneficial for larger $n$, since more agents provide more reliable information on useful cycles. Hence, we use $(\delta, \beta) = (20,1)$ for the Erd\H os-R\'enyi and Manhattan instances, while we use $(\delta, \beta) = (5,1)$ for the reload instances. Furthermore, we use $q_0 = 0.4$, $\gamma = 0.6$ and $\Omega = 3(m/n)$, see (\ref{successorRL}), (\ref{SQ-update}) and (\ref{delayedRL}), respectively, for all instance types. Finally, as it is not clear from our tests which value of the learning rate parameter $\alpha$, see \eqref{SQ-update},  provides the best results, we run the SQ-algorithm three times using $\alpha = 0.3, 0.5$ and $0.7$ and solve the final \textsc{SPP}, see (\ref{SPPproblem}), using all generated cycles. The number of iterations of the SQ-algorithm is set to 500 for the Erd\H os-R\'enyi and the reload instances, while it is set to 100 for the Manhattan instances, due to the large number of nodes. As the final \textsc{SPP}s can be solved efficiently by CPLEX for all our instances, we report the optimal \textsc{SPP} bounds.

\subsection{Results on Lower Bounds} \label{SubsectionResultsLower}
We now discuss our findings with respect to the lower bounds on all test instances. For the reload instances we compare the performance of $(SDP_{S2})$ to the performance of $(MILP), (LBB1)$ and $(RLT1)$. We omit the brackets from now on to indicate the bound values. Table~\ref{Table:reloadbounds} shows for each of the 60 reload instances the bound value resulting from each of the approaches. Table~\ref{Table:reloadtimes} shows all computation times for the reload instances, including the number of iterations and the average of the primal and dual residual, see Section~\ref{SubsubsectionStopping}, for the \textsc{PRSM}. To visualize the quality of the bounds over the entire reload test set, Figure~\ref{figure:Boxplot} shows a boxplot of the test data in Table~\ref{Table:reloadbounds}. On the $y$-axis the deviation from the average bound is presented, i.e., for each instance we compute the ratio of each single bound over the average value of the four bounds and these ratios are visualized per bound type.

\begin{figure}[h]
\centering
\includegraphics[scale=0.6]{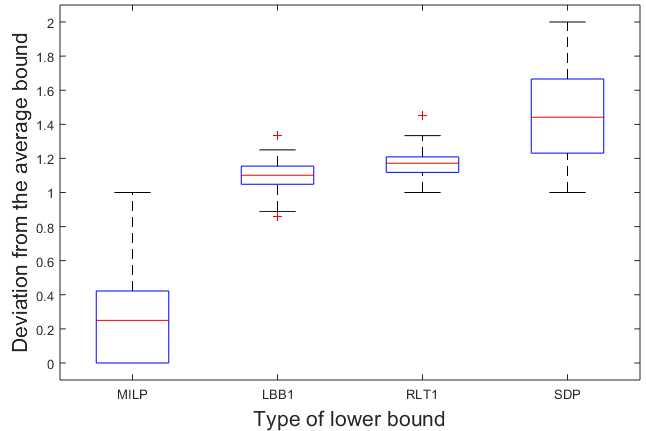}
\caption{Boxplot showing the quality of lower bounds $MILP$, $LBB1$, $RLT1$ and $SDP_{S2}$ on reload instances. \label{figure:Boxplot} }
\end{figure}

\noindent It follows from Table~\ref{Table:reloadbounds} and Figure~\ref{figure:Boxplot} that $SDP_{S2}$ clearly provides the strongest bounds, followed by $RLT1, LBB1$ and finally by $MILP$, which behaves poorly for most of the instances. In fact, the hierarchy $LBB1 \leq RLT1 \leq SDP_{S2}$ can be proven easily and holds with strict inequality for the majority of the instances. It can be seen that $SDP_{S2}$ performs generally about 1.5 times better than the average of the four bounds. The bound $SDP_{S2}$ even turns out to be optimal for 88\% of the instances. When considering Table~\ref{Table:reloadtimes}, it follows that although the computation times are larger than those of $MILP$ and $LBB1$, the SDP-bound can be computed efficiently for most of the instances. The computation times are always within 30 seconds and for 75\% of the instances within 10 seconds, while the computation time of $RLT1$ is above 90 seconds for 67\% of the instances. Moreover, although the optimum can be computed for all tested reload instances, the computation time is in some cases as large as 2000 seconds. Hence, for the reload instances we conclude that $SDP_{S2}$ can be favoured above other bounds in both quality and time.

Next, we consider the Erd\H os-R\'enyi instances. Table~\ref{Table:erbounds} shows the bound values for the Erd\H os-R\'enyi test set, among which the bounds $SDP_{S2}$ and $SDP_{S3}$ for various number of cuts. We do not consider the first level RLT bound, as it cannot be efficiently computed for the majority of the instances. The column $OPT$ reports the optimal solution if this solution could be computed in 3 hours and `-' otherwise. The computation times are reported in Table~\ref{Table:ertimes} and the average of primal and dual residual and the number of final cuts in the \textsc{CP-ALM} for the SDP bounds are reported in Table~\ref{Table:errescut}.

\begin{table}[H]
\centering
\scriptsize
\setlength{\tabcolsep}{2.5pt}
\begin{tabular}{@{}ccccccc@{}}
\toprule
\textbf{Instance} & \textbf{$n$} & \textbf{$D$} & \textbf{$MILP$} & \textbf{$LBB1$} & \textbf{$RLT1$} & \textbf{$SDP_{S2}$} \\ \midrule
REL1              & 10           & 1            & 3               & 4               & 4               & 4                   \\
REL2              &              & 10           & 3               & 9               & 9               & 9                   \\
REL3              &              & 1            & 3               & 4               & 5               & 5                   \\
REL4              &              & 10           & 3               & 8               & 9               & 12                  \\
REL5              &              & 1            & 3               & 4               & 4               & 4                   \\
REL6              &              & 10           & 5               & 12              & 13              & 14                  \\
REL7              &              & 1            & 3               & 4               & 5               & 5                   \\
REL8              &              & 10           & 4               & 9               & 11              & 11                  \\
REL9              &              & 1            & 2               & 2               & 2               & 2                   \\
REL10             &              & 10           & 4               & 9               & 11              & 12                  \\
REL11             &              & 1            & 2               & 3               & 3               & 3                   \\
REL12             &              & 10           & 5               & 9               & 9               & 9                   \\
REL13             &              & 1            & 2               & 4               & 4               & 4                   \\
REL14             &              & 10           & 3               & 9               & 11              & 11                  \\
REL15             &              & 1            & 3               & 4               & 4               & 4                   \\
REL16             &              & 10           & 3               & 8               & 9               & 11                  \\
REL17             &              & 1            & 4               & 4               & 4               & 4                   \\
REL18             &              & 10           & 3               & 8               & 9               & 10                  \\
REL19             &              & 1            & 3               & 5               & 5               & 5                   \\
REL20             &              & 10           & 3               & 10              & 11              & 11                  \\
REL21             & 15           & 1            & 2               & 4               & 4               & 5                   \\
REL22             &              & 10           & 2               & 9               & 9               & 12                  \\
REL23             &              & 1            & 1               & 3               & 3               & 4                   \\
REL24             &              & 10           & 1               & 7               & 8               & 11                  \\
REL25             &              & 1            & 1               & 4               & 5               & 5                   \\
REL26             &              & 10           & 1               & 6               & 6               & 9                   \\
REL27             &              & 1            & 1               & 4               & 4               & 4                   \\
REL28             &              & 10           & 1               & 7               & 7               & 9                   \\
REL29             &              & 1            & 1               & 5               & 5               & 6                   \\
REL30             &              & 10           & 0               & 6               & 7               & 10                  \\
\bottomrule
\end{tabular} \hspace{1cm}
\begin{tabular}{@{}ccccccc@{}}
\toprule
\textbf{Instance} & \textbf{$n$} & \textbf{$D$} & \textbf{$MILP$} & \textbf{$LBB1$} & \textbf{$RLT1$} & \textbf{$SDP_{S2}$} \\ \midrule
REL31             &    15          & 1            & 1               & 4               & 4               & 5                   \\
REL32             &              & 10           & 1               & 7               & 8               & 11                  \\
REL33             &              & 1            & 1               & 4               & 4               & 4                   \\
REL34             &              & 10           & 1               & 5               & 5               & 8                   \\
REL35             &              & 1            & 1               & 4               & 4               & 4                   \\
REL36             &              & 10           & 1               & 4               & 5               & 8                   \\
REL37             &              & 1            & 2               & 6               & 6               & 6                   \\
REL38             &              & 10           & 1               & 9               & 9               & 11                  \\
REL39             &              & 1            & 1               & 3               & 4               & 3                   \\
REL40             &              & 10           & 1               & 6               & 7               & 7                   \\
REL41             & 20           & 1            & 0               & 3               & 3               & 4                   \\
REL42             &              & 10           & 1               & 4               & 4               & 7                   \\
REL43             &              & 1            & 0               & 2               & 2               & 3                   \\
REL44             &              & 10           & 0               & 5               & 5               & 7                   \\
REL45             &              & 1            & 0               & 2               & 2               & 3                   \\
REL46             &              & 10           & 0               & 5               & 5               & 6                   \\
REL47             &              & 1            & 0               & 2               & 2               & 3                   \\
REL48             &              & 10           & 0               & 3               & 3               & 5                   \\
REL49             &              & 1            & 0               & 3               & 3               & 4                   \\
REL50             &              & 10           & 0               & 5               & 6               & 8                   \\
REL51             &              & 1            & 0               & 3               & 3               & 3                   \\
REL52             &              & 10           & 0               & 3               & 4               & 6                   \\
REL53             &              & 1            & 0               & 3               & 3               & 4                   \\
REL54             &              & 10           & 0               & 6               & 6               & 9                   \\
REL55             &              & 1            & 0               & 2               & 2               & 3                   \\
REL56             &              & 10           & 1               & 6               & 6               & 8                   \\
REL57             &              & 1            & 0               & 3               & 3               & 4                   \\
REL58             &              & 10           & 0               & 3               & 4               & 7                   \\
REL59             &              & 1            & 0               & 2               & 2               & 3                   \\
REL60             &              & 10           & 0               & 5               & 5               & 8                   \\ \bottomrule
\end{tabular}
\caption{Comparison of different bounds for reload instances. \label{Table:reloadbounds} }
\end{table}

\begin{table}[H]
\scriptsize
\centering
\setlength{\tabcolsep}{3pt}
\begin{tabular}{@{}ccccccc@{}}
\toprule
                  & \textbf{$MILP$} & \textbf{$LBB1$} & \textbf{$RLT1$} & \multicolumn{3}{c}{\textbf{$SDP_{S2}$}}      \\ \cmidrule(l){2-2} \cmidrule(l){3-3} \cmidrule(l){4-4} \cmidrule(l){5-7}
\textbf{Instance} & \textbf{time}   & \textbf{time}   & \textbf{time}   & \textbf{time} & \textbf{iter} & \textbf{res} \\ \midrule
REL1              & 0.112           & 0.007           & 1.004           & 0.278         & 437           & 0.003        \\
REL2              & 0.113           & 0.007           & 0.398           & 0.215         & 300           & 0.032        \\
REL3              & 0.111           & 0.006           & 0.350           & 0.751         & 1192          & 0.005        \\
REL4              & 0.106           & 0.006           & 0.380            & 0.624         & 978           & \textless{}0.001       \\
REL5              & 0.107           & 0.006           & 0.321           & 0.332         & 529           & 0.001        \\
REL6              & 0.106           & 0.006           & 0.358           & 0.230          & 377           & 0.045        \\
REL7              & 0.108           & 0.006           & 0.364           & 0.665         & 1044          & 0.005        \\
REL8              & 0.107           & 0.006           & 0.366           & 0.182         & 294           & 0.086        \\
REL9              & 0.107           & 0.006           & 0.217           & 0.370          & 594           & 0.042        \\
REL10             & 0.106           & 0.006           & 0.377           & 0.430          & 681           & 0.053        \\
REL11             & 0.108           & 0.006           & 0.320            & 0.358         & 586           & 0.027        \\
REL12             & 0.108           & 0.006           & 0.334           & 0.147         & 232           & 0.078        \\
REL13             & 0.108           & 0.005           & 0.384           & 0.273         & 441           & 0.004        \\
REL14             & 0.106           & 0.006           & 0.380            & 0.287         & 327           & 0.038        \\
REL15             & 0.108           & 0.005           & 0.284           & 0.396         & 589           & 0.039        \\
REL16             & 0.107           & 0.006           & 0.374           & 0.494         & 790           & 0.048        \\
REL17             & 0.109           & 0.005           & 0.218           & 0.425         & 695           & 0.040         \\
REL18             & 0.107           & 0.006           & 0.393           & 0.442         & 695           & 0.041        \\
REL19             & 0.108           & 0.005           & 0.425           & 0.659         & 1034          & 0.005        \\
REL20             & 0.106           & 0.006           & 0.408           & 0.170       & 278           & 0.170         \\
REL21             & 0.421           & 0.050           & 9.542           & 3.146         & 1371          & 0.007        \\
REL22             & 0.412           & 0.048           & 8.475           & 3.950          & 1684          & \textless{}0.001       \\
REL23             & 0.415           & 0.044           & 8.617           & 3.066         & 1277          & 0.006        \\
REL24             & 0.411           & 0.048           & 8.693           & 2.695         & 1124          & 0.001        \\
REL25             & 0.415           & 0.044           & 8.565           & 5.069         & 2149          & 0.006        \\
REL26             & 0.413           & 0.046           & 8.697           & 2.701         & 1117          & \textless{}0.001       \\
REL27             & 0.414           & 0.044           & 9.208           & 6.097         & 2500          & 0.012        \\
REL28             & 0.411           & 0.051           & 8.229           & 1.667         & 689           & 0.017        \\
REL29             & 0.414           & 0.045           & 8.249           & 1.972         & 818           & 0.006        \\
REL30             & 0.421           & 0.050            & 8.676           & 3.979         & 1626          & \textless{}0.001       \\ \bottomrule
\end{tabular} \hspace{0.5cm}
\begin{tabular}{@{}ccccccc@{}}
\toprule
                  & \textbf{$MILP$} & \textbf{$LBB1$} & \textbf{$RLT1$} & \multicolumn{3}{c}{\textbf{$SDP_{S2}$}}      \\ \cmidrule(l){2-2} \cmidrule(l){3-3} \cmidrule(l){4-4} \cmidrule(l){5-7}
\textbf{Instance} & \textbf{time}   & \textbf{time}   & \textbf{time}   & \textbf{time} & \textbf{iter} & \textbf{res} \\ \midrule
REL31             & 0.415           & 0.043           & 9.720            & 3.044         & 1279          & 0.006        \\
REL32             & 0.412           & 0.047           & 8.450            & 3.002         & 1183          & 0.005        \\
REL33             & 0.415           & 0.044           & 8.626           & 5.998         & 2500          & 0.022        \\
REL34             & 0.410            & 0.046           & 9.133           & 0.913         & 380           & 0.040         \\
REL35             & 0.419           & 0.048           & 8.277           & 5.998         & 2500          & 0.017        \\
REL36             & 0.413           & 0.048           & 8.711           & 1.925         & 775           & 0.010         \\
REL37             & 0.419           & 0.045           & 8.143           & 3.439         & 1417          & 0.006        \\
REL38             & 0.412           & 0.048           & 8.004           & 6.119         & 2500          & 0.032        \\
REL39             & 0.414           & 0.045           & 7.099           & 5.864         & 2433          & 0.028        \\
REL40             & 0.415           & 0.044           & 8.357           & 3.131         & 1291          & 0.066        \\
REL41             & 1.062           & 0.137           & 120.2           & 28.54         & 2500          & 0.008        \\
REL42             & 1.088           & 0.133           & 142.1           & 17.93         & 1501          & 0.001        \\
REL43             & 1.075           & 0.145           & 127.2           & 29.58         & 2500          & 0.007        \\
REL44             & 1.083           & 0.137           & 120.2           & 24.44         & 1990          & 0.001        \\
REL45             & 1.067           & 0.127           & 105.8           & 21.15         & 1781          & 0.007        \\
REL46             & 1.095           & 0.139           & 118.4           & 6.955         & 570           & 0.035        \\
REL47             & 1.067           & 0.155           & 158.5           & 19.05         & 1626          & 0.007        \\
REL48             & 1.091           & 0.138           & 184.2           & 4.340          & 357           & 0.059        \\
REL49             & 1.073           & 0.132           & 143.0             & 24.12         & 2067          & 0.007        \\
REL50             & 1.091           & 0.133           & 118.1           & 4.428         & 355           & 0.061        \\
REL51             & 1.067           & 0.136           & 97.63           & 19.34         & 1651          & 0.007        \\
REL52             & 1.093           & 0.130            & 162.1           & 23.48         & 1907          & 0.001        \\
REL53             & 1.076           & 0.138           & 128.9           & 29.56         & 2500          & 0.007        \\
REL54             & 1.080            & 0.129           & 107.1           & 21.26         & 1725          & 0.001        \\
REL55             & 1.064           & 0.146           & 128.8           & 30.50          & 2500          & 0.009        \\
REL56             & 1.086           & 0.132           & 127.0             & 19.09         & 1548          & 0.001        \\
REL57             & 1.071           & 0.152           & 122.1           & 19.68         & 1578          & 0.007        \\
REL58             & 1.095           & 0.136           & 145.5           & 5.723         & 430           & 0.041        \\
REL59             & 1.069           & 0.136           & 119.0             & 17.12         & 1439          & 0.007        \\
REL60             & 1.091           & 0.126           & 129.4           & 6.227         & 502           & 0.035        \\ \bottomrule
\end{tabular}
\caption{Computation times in seconds, average residuals and number of iterations for reload instances. \label{Table:reloadtimes}}
\end{table}
\noindent For the Erd\H os-R\'enyi instances we also see that $SDP_{S2}$ significantly outperforms $MILP$ and $LBB1$ in terms of quality of the bound. Moreover, it is clear that we can successfully improve the bounds by adding cuts using the new \textsc{CP-ALM}. Except for the instances where $SDP_{S2}$ is already optimal, we see that $SDP_{S3}$ provides a strictly higher bound already after adding 50 cuts at a time. For most instances, this improvement of $SDP_{S3}$ compared to $SDP_{S2}$ is about 3\%-6\%. Interestingly, this improvement seems to be independent of the problem size.
As we already observed in Figure~\ref{PlotER4}, we see that a higher value of $numCuts$ leads to a higher lower bound. This higher value comes, however, at the cost of computation time as can be seen from Table~\ref{Table:ertimes}. When taking both quality and efficiency into account, it seems beneficial to add only a small number of cuts, as this often leads to a significant increase of the bound at a relatively low computational cost. For instances up to 1000 arcs the \textsc{CP-ALM} terminates often within 30 minutes, while SDP bounds for instances up to 1850 arcs (!) can be computed within 2 hours. Hence, the \textsc{CP-ALM} is able to provide strong lower bounds for very large-scale SDPs in a reasonable time span, whereas the interior point method of Mosek \cite{Mosek} can solve $(SDP_{S2})$ for instances up to only 300 arcs without running out of memory. \vspace{-0.3cm}
\begin{table}[H]
\centering
\scriptsize
\setlength{\tabcolsep}{3pt}
\begin{tabular}{@{}cccccccccccc@{}}
\toprule
\multicolumn{1}{l}{}                  & \textbf{}    & \textbf{}    & \textbf{}    & \textbf{}      & \textbf{}       & \textbf{}       & \textbf{}           & \multicolumn{4}{c}{\textbf{$SDP_{S3}$}}                                                                                    \\ \cmidrule(l){9-12}
\multicolumn{1}{l}{\textbf{Instance}} & \textbf{$p$} & \textbf{$n$} & \textbf{$m$} & \textbf{$OPT$} & \textbf{$MILP$} & \textbf{$LBB1$} & \textbf{$SDP_{S2}$} & \textbf{\begin{tabular}[c]{@{}c@{}}$numCuts$\\ $50$\end{tabular}} & \textbf{\begin{tabular}[c]{@{}c@{}}$numCuts$\\ $150$\end{tabular}} & \textbf{\begin{tabular}[c]{@{}c@{}}$numCuts$\\ $300$\end{tabular}} & \textbf{\begin{tabular}[c]{@{}c@{}}$numCuts$\\ $500$\end{tabular}} \\ \midrule
ER1                                   & 0.3          & 20           & 119          & 319            & 165             & 260             & 319                 & 319                     & 319                                          & 319                      & 319                      \\
RER1                                  &              &              & 113          & 293            & 154             & 274             & 293                 & 293                     & 293                                          & 293                      & 293                      \\
ER2                                   &              & 25           & 177          & 386            & 167             & 305             & 386                 & 386                     & 386                                          & 386                      & 386                      \\
RER2                                  &              &              & 169          & 391            & 151             & 303             & 391                 & 391                     & 391                                          & 391                      & 391                      \\
ER3                                   &              & 30           & 284          & -              & 122             & 230             & 287                 & 292                     & 294                                          & 295                      & 296                      \\
RER3                                  &              &              & 256          & 281            & 69              & 208             & 258                 & 262                     & 264                                          & 265                      & 266                      \\
ER4                                   &              & 35           & 361          & -              & 138             & 273             & 328                 & 331                     & 333                                          & 335                      & 336                      \\
RER4                                  &              &              & 347          & -              & 61              & 189             & 233                 & 236                     & 238                                          & 239                      & 240         \\
ER5                                   &              & 40           & 468          & -              & 131             & 265             & 318                 & 321                     & 322                                          & 323                      & 324                      \\
RER5                                  &              &              & 495          & -              & 17              & 177             & 215                 & 217                     & 219                                          & 219                      & 220                      \\
ER6                                   &              & 45           & 592          & -              & 138             & 287             &       330                 & 333                     & 336                                          & 337                      & 338                       \\
RER6                                  &              &              & 623          & -              & 9               & 110             & 146                 & 148                     & 149                                          & 150                      & 151                      \\
ER7                                   &              & 50           & 754          & -              & 130             & 267             & 313                 & 316                     & 318                                          & 319                      & 319                      \\
RER7                                  &              &              & 746          & -              & 3               & 91              & 116                 & 117                     & 118                                          & 119                      & 119                      \\
ER8                                   &              & 60           & 1062         & -              & 118             & 272             & 301                 & 303                     & 304                                          & 305                      & 305                      \\
RER8                                  &              &              & 995          & -              & 1               & 74              & 93                  & 94                      & 95                                           & 95                       & 95                       \\
ER9                                   &              & 70           & 1481         & -              & 123             & 255             & 286                 & 287                     & 288                                          & 289                      & 289                      \\
RER9                                  &              &              & 1512         & -              & 0               & 99              & 131                 & 132                     & 132                                          & 133                      & 133                      \\
ER10                                  &              & 80           & 1842         & -              & 122             & 263             & 291                 & 292                     & 293                                          & 293                      & 293                      \\
RER10                                 &              &              & 1859         & -              & 0               & 33              & 52                  & 53                      & 53                                           & 53                       & 54                       \\
ER11                                  & 0.5          & 20           & 195          & 236            & 95              & 175             & 227                 & 232                     & 233                                          & 234                      & 234                      \\
RER11                                 &              &              & 194          & 172            & 34              & 136             & 172                 & 172                     & 172                                          & 172                      & 172                      \\
ER12                                  &              & 25           & 327          & -              & 67              & 136             & 169                 & 171                     & 172                                          & 173                      & 173                      \\
RER12                                 &              &              & 308          & 99             & 7               & 57              & 84                  & 85                      & 86                                           & 87                       & 87                       \\
ER13                                  &              & 30           & 434          & -              & 79              & 161             & 197                 & 200                     & 201                                          & 202                      & 202                      \\
RER13                                 &              &              & 435          & -              & 9               & 106             & 139                 & 141                     & 142                                          & 143                      & 143                      \\
ER14                                  &              & 40           & 793          & -              & 74              & 166             & 196                 & 198                     & 199                                          & 199                      & 200                      \\
RER14                                 &              &              & 770          & -              & 0               & 50              & 72                  & 73                      & 74                                           & 73                       & 74                       \\
ER15                                  &              & 50           & 1197         & -              & 77              & 165             & 188                 & 189                     & 190                                          & 191                      & 191                      \\
RER15                                 &              &              & 1235         & -              & 0               & 18              & 35                  & 36                      & 37                                           & 37                       & 37                       \\ \bottomrule
\end{tabular}
\caption{Comparison of different bounds for Erd\H os-R\'enyi instances. \label{Table:erbounds}} \vspace{-0.3cm}
\end{table}
\noindent Finally, we consider the performance of the lower bounds on the Manhattan instances, which can be found in Table~\ref{Table:mhbounds} and \ref{Table:mhtimes}. With respect to the quality of the bounds we can draw the same conclusions as before. Namely, the SDP bound $SDP_{S3}$ performs best on all instances, followed by $SDP_{S2}$. Since the optimal values for many of these instances can be computed, we moreover see that our SDP bounds are very close to optimal. Although we again see that the cuts can successfully improve the lower bounds, the relative improvement is smaller than for the Erd\H os-R\'enyi test set. An explanation can be found by looking at the residuals in Table~\ref{Table:mhtimesb}, which are significantly larger than the residuals for the first two types of instances. Apparently, the Manhattan instances need more iterations to converge, probably due to the inner structure of these instances. Stopping the \textsc{CP-ALM} when it has only partly converged, leads to weaker and less stable lower bounds. Namely, the reported lower bound is obtained by a projection of the current dual matrix, and further experiments show that in particular the dual residual converges slowly. The residuals increase with the size of the instance. Hence, we expect that even  better bounds for the Manhattan instances
 can be obtained by letting the \textsc{CP-ALM} run for more iterations.
However, we conclude from the current tables that the SDP bounds for the  Manhattan instances significantly outperform the bounds from the literature in a reasonable time span.

\begin{table}[H]
\centering
\scriptsize
\begin{tabular}{@{}cccccccc@{}}
\toprule
\textbf{}         & \textbf{}       & \textbf{}       & \textbf{}           & \multicolumn{4}{c}{\textbf{$SDP_{S3}$}}                                                                                                                                                                                                                                          \\ \cmidrule(l){5-8} 
\textbf{Instance} & \textbf{$MILP$} & \textbf{$LBB1$} & \textbf{$SDP_{S2}$} & \textbf{\begin{tabular}[c]{@{}c@{}}$numCuts$\\ $50$\end{tabular}} & \textbf{\begin{tabular}[c]{@{}c@{}}$numCuts$\\ $150$\end{tabular}} & \textbf{\begin{tabular}[c]{@{}c@{}}$numCuts$\\ $300$\end{tabular}} & \textbf{\begin{tabular}[c]{@{}c@{}}$numCuts$\\ $500$\end{tabular}} \\ \midrule
ER1               & 0.201           & 0.016           & 0.390               & 0.330                                                             & 0.470                                                              & 0.400                                                              & 0.330                                                              \\
RER1              & 0.194           & 0.008           & 0.150               & 0.140                                                             & 0.130                                                              & 0.130                                                              & 0.150                                                              \\
ER2               & 0.333           & 0.019           & 2.193               & 1.700                                                             & 1.790                                                              & 1.660                                                              & 1.710                                                              \\
RER2              & 0.319           & 0.016           & 3.520               & 18.37                                                             & 155.6                                                              & 55.27                                                              & 163.8                                                              \\
ER3               & 0.827           & 0.068           & 11.89               & 35.15                                                             & 128.6                                                              & 333.2                                                              & 924.8                                                              \\
RER3              & 0.673           & 0.042           & 8.120               & 29.55                                                             & 67.15                                                              & 110.0                                                              & 205.0                                                              \\
ER4               & 1.151           & 0.106           & 28.17               & 54.61                                                             & 86.61                                                              & 137.7                                                              & 237.3                                                              \\
RER4              & 1.166           & 0.107           & 24.83               & 60.40                                                             & 93.71                                                              & 161.4                                                              & 257.9                                                              \\
ER5               & 1.914           & 0.139           & 48.57               & 86.64                                                             & 121.2                                                              & 230.8                                                              & 526.3                                                              \\
RER5              & 2.088           & 0.153           & 53.83               & 89.76                                                             & 112.7                                                              & 148.2                                                              & 234.4                                                              \\
ER6               & 2.856           & 0.201           & 99.59               & 203.4                                                             & 253.4                                                              & 350.3                                                              & 499.2                                                              \\
RER6              & 3.048           & 0.220           & 113.4               & 200.4                                                             & 247.1                                                              & 345.7                                                              & 494.7                                                              \\
ER7               & 4.489           & 0.327           & 168.4               & 291.3                                                             & 356.6                                                              & 607.0                                                              & 1173                                                               \\
RER7              & 4.207           & 0.297           & 164.6               & 285.2                                                             & 358.1                                                              & 447.9                                                              & 601.8                                                              \\
ER8               & 10.824          & 0.625           & 463.0               & 870.8                                                             & 1000                                                               & 1306                                                               & 2198                                                               \\
RER8              & 7.969           & 0.529           & 340.5               & 624.0                                                             & 652.6                                                              & 748.8                                                              & 950.6                                                              \\
ER9               & 24.184          & 1.346           & 1305                & 2160                                                              & 2293                                                               & 2838                                                               & 4517                                                               \\
RER9              & 25.232          & 1.420           & 1381                & 2308                                                              & 2371                                                               & 2555                                                               & 2961                                                               \\
ER10              & 42.034          & 2.273           & 2446                & 4110                                                              & 4088                                                               & 4548                                                               & 7178                                                               \\
RER10             & 41.74           & 2.305           & 2516                & 4035                                                              & 4130                                                               & 4451                                                               & 4868                                                               \\
ER11              & 0.397           & 0.031           & 6.071               & 22.95                                                             & 67.97                                                              & 178.4                                                              & 512.1                                                              \\
RER11             & 0.415           & 0.027           & 4.640               & 13.53                                                             & 32.20                                                              & 51.59                                                              & 82.89                                                              \\
ER12              & 0.967           & 0.107           & 17.99               & 37.11                                                             & 70.64                                                              & 200.9                                                              & 524.0                                                              \\
RER12             & 0.862           & 0.088           & 14.07               & 31.89                                                             & 51.67                                                              & 97.69                                                              & 170.2                                                              \\
ER13              & 1.554           & 0.155           & 42.23               & 77.36                                                             & 108.5                                                              & 196.2                                                              & 472.9                                                              \\
RER13             & 1.579           & 0.150           & 42.59               & 77.82                                                             & 98.80                                                              & 142.1                                                              & 208.3                                                              \\
ER14              & 4.539           & 0.431           & 201.7               & 348.8                                                             & 493.6                                                              & 868.5                                                              & 2840                                                               \\
RER14             & 4.283           & 0.389           & 187.8               & 319.2                                                             & 373.9                                                              & 457.3                                                              & 627.8                                                              \\
ER15              & 12.663          & 1.068           & 721.6               & 1243                                                              & 1426                                                               & 1854                                                               & 3773                                                               \\
RER15             & 12.99           & 1.141           & 795.2               & 1438                                                              & 1427                                                               & 1559                                                               & 1775                                                               \\ \bottomrule
\end{tabular}
\caption{Comparison of computation times (in seconds) for Erd\H os-R\'enyi instances. \vspace{-0.3cm} \label{Table:ertimes}}
\end{table}

\begin{table}[H]
\centering
\scriptsize
\begin{tabular}{@{}cccccccccc@{}}
\toprule
\textbf{}         & \textbf{}           & \multicolumn{8}{c}{\textbf{$SDP_{S3}$}}                                                                                                                                                                                                                                                                                                                  \\ \cmidrule(l){3-10} 
\textbf{}         & \textbf{$SDP_{S2}$} & \multicolumn{2}{c}{\begin{tabular}[c]{@{}c@{}}$numCuts$\\ 50\end{tabular}} & \multicolumn{2}{c}{\begin{tabular}[c]{@{}c@{}}$numCuts$\\ 150\end{tabular}} & \multicolumn{2}{c}{\begin{tabular}[c]{@{}c@{}}$numCuts$\\ 300\end{tabular}} & \multicolumn{2}{c}{\begin{tabular}[c]{@{}c@{}}$numCuts$\\ 500\end{tabular}} \\ \cmidrule(lr){3-4} \cmidrule(lr){5-6} \cmidrule(lr){7-8} \cmidrule(lr){9-10}
\textbf{Instance} & \textbf{res}        & \textbf{res}                             & \textbf{cuts}                            & \textbf{res}                             & \textbf{cuts}                             & \textbf{res}                             & \textbf{cuts}                             & \textbf{res}                             & \textbf{cuts}                             \\ \midrule
ER1               & \textless{}0.001              & \textless{}0.001                                   & 0                                        & \textless{}0.001                                   & 0                                         & \textless{}0.001                                   & 0                                         & \textless{}0.001                                   & 0                                         \\
RER1              & \textless{}0.001              & \textless{}0.001                                   & 0                                        & \textless{}0.001                                   & 0                                         & \textless{}0.001                                   & 0                                         & \textless{}0.001                                   & 0                                         \\
ER2               & 0.003               & 0.002                                    & 0                                        & 0.002                                    & 0                                         & 0.003                                    & 0                                         & 0.003                                    & 0                                         \\
RER2              & \textless{}0.001              & \textless{}0.001                                   & 100                                      & \textless{}0.001                                   & 294                                       & \textless{}0.001                                   & 303                                       & \textless{}0.001                                   & 532                                       \\
ER3               & 0.002               & 0.002                                    & 148                                      & 0.003                                    & 446                                       & 0.003                                    & 878                                       & 0.003                                    & 1455                                      \\
RER3              & 0.002               & 0.003                                    & 150                                      & 0.003                                    & 444                                       & 0.004                                    & 882                                       & 0.005                                    & 1461                                      \\
ER4               & 0.002               & 0.002                                    & 148                                      & 0.003                                    & 450                                       & 0.003                                    & 899                                       & 0.003                                    & 1494                                      \\
RER4              & 0.003               & 0.002                                    & 150                                      & 0.002                                    & 446                                       & 0.002                                    & 887                                       & 0.003                                    & 1479                                      \\
ER5               & 0.002               & 0.002                                    & 150                                      & 0.002                                    & 449                                       & 0.007                                    & 898                                       & 0.003                                    & 1494                                      \\
RER5              & 0.003               & 0.003                                    & 150                                      & 0.003                                    & 450                                       & 0.003                                    & 899                                       & 0.004                                    & 1497                                      \\
ER6               & 0.001               & 0.003                                    & 200                                      & 0.004                                    & 598                                       & 0.005                                    & 1195                                      & 0.004                                    & 1992                                      \\
RER6              & 0.003               & 0.003                                    & 200                                      & 0.003                                    & 598                                       & 0.003                                    & 1199                                      & 0.003                                    & 1995                                      \\
ER7               & 0.001               & 0.002                                    & 200                                      & 0.002                                    & 597                                       & 0.002                                    & 1195                                      & 0.008                                    & 1996                                      \\
RER7              & 0.002               & 0.002                                    & 199                                      & 0.002                                    & 598                                       & 0.002                                    & 1191                                      & 0.002                                    & 1987                                      \\
ER8               & 0.001               & 0.002                                    & 250                                      & 0.002                                    & 750                                       & 0.002                                    & 1500                                      & 0.002                                    & 2500                                      \\
RER8              & 0.002               & 0.002                                    & 200                                      & 0.002                                    & 599                                       & 0.003                                    & 1198                                      & 0.003                                    & 1998                                      \\
ER9               & 0.001               & 0.001                                    & 250                                      & 0.002                                    & 749                                       & 0.002                                    & 1497                                      & 0.002                                    & 2497                                      \\
RER9              & 0.003               & 0.003                                    & 248                                      & 0.007                                    & 745                                       & 0.003                                    & 1497                                      & 0.003                                    & 2500                                      \\
ER10              & 0.002               & 0.002                                    & 250                                      & 0.002                                    & 750                                       & 0.002                                    & 1500                                      & 0.003                                    & 2500                                      \\
RER10             & 0.003               & 0.003                                    & 250                                      & 0.003                                    & 750                                       & 0.003                                    & 1499                                      & 0.003                                    & 2500                                      \\
ER11              & 0.001               & 0.003                                    & 148                                      & 0.003                                    & 440                                       & 0.003                                    & 877                                       & 0.006                                    & 1458                                      \\
RER11             & 0.055               & \textless{}0.001                                   & 50                                       & \textless{}0.001                                   & 150                                       & \textless{}0.001                                   & 300                                       & \textless{}0.001                                   & 500                                       \\
ER12              & 0.001               & 0.002                                    & 150                                      & 0.002                                    & 450                                       & 0.003                                    & 898                                       & 0.003                                    & 1498                                      \\
RER12             & 0.002               & 0.002                                    & 133                                      & 0.003                                    & 415                                       & 0.003                                    & 851                                       & 0.003                                    & 1431                                      \\
ER13              & 0.001               & 0.002                                    & 149                                      & 0.002                                    & 445                                       & 0.002                                    & 892                                       & 0.002                                    & 1481                                      \\
RER13             & 0.003               & 0.003                                    & 150                                      & 0.003                                    & 448                                       & 0.003                                    & 891                                       & 0.004                                    & 1481                                      \\
ER14              & 0.001               & 0.001                                    & 200                                      & 0.002                                    & 597                                       & 0.002                                    & 1196                                      & 0.004                                    & 1992                                      \\
RER14             & 0.002               & 0.003                                    & 200                                      & 0.002                                    & 600                                       & 0.019                                    & 1200                                      & 0.003                                    & 2000                                      \\
ER15              & 0.001               & 0.001                                    & 250                                      & 0.001                                    & 750                                       & 0.002                                    & 1500                                      & 0.002                                    & 2500                                      \\
RER15             & 0.003               & 0.003                                    & 250                                      & 0.003                                    & 750                                       & 0.003                                    & 1500                                      & 0.003                                    & 2500                                      \\ \bottomrule
\end{tabular}
\caption{Comparison of average residuals and total number of added cuts for Erd\H os-R\'enyi instances. \label{Table:errescut}}
\end{table}

\begin{table}[H]
\centering
\scriptsize
\setlength{\tabcolsep}{3pt}
\begin{tabular}{@{}cccccccccc@{}}
\toprule
\textbf{}         & \textbf{}     & \textbf{}    & \textbf{}    & \textbf{}      & \textbf{}       & \textbf{}       & \textbf{}           & \multicolumn{2}{c}{\textbf{$SDP_{S3}$}}             \\ \cmidrule(l){9-10}
\textbf{Instance} & \textbf{Type} & \textbf{$n$} & \textbf{$m$} & \textbf{$OPT$} & \textbf{$MILP$} & \textbf{$LBB1$} & \textbf{$SDP_{S2}$} & \textbf{\begin{tabular}[c]{@{}c@{}}$numCuts$\\ $300$\end{tabular}} & \textbf{\begin{tabular}[c]{@{}c@{}}$numCuts$\\ $500$\end{tabular}} \\ \midrule
MH1               & $(5,5)$       & 25           & 50           & 103            & 102             & 103             & 103                 & 103                      & 103                      \\
MH2               & $(10,10)$     & 100          & 200          & 418            & 394             & 418             & 418                 & 418                      & 418                      \\
MH3               & $(15,15)$     & 225          & 450          & 892            & 847             & 892             & 892                 & 892                      & 892                      \\
MH4               & $(16,16)$     & 256          & 512          & 1030           & 985             & 1030            & 1030                & 1030                     & 1030                     \\
MH5               & $(17,17)$     & 289          & 578          & 1226           & 1162            & 1214            & 1226                & 1226                     & 1226                     \\
MH6               & $(18,18)$     & 324          & 648          & 1283           & 1230            & 1282            & 1282                & 1283                     & 1283                     \\
MH7               & $(19,19)$     & 361          & 722          & 1448           & 1378            & 1446            & 1446                & 1446                     & 1446                     \\
MH8               & $(20,20)$     & 400          & 800          & 1539           & 1472            & 1537            & 1536                & 1537                     & 1537                     \\
MH9               & $(25,25)$     & 625          & 1250         & 2572           & 2439            & 2559            & 2568                & 2568                     & 2568                     \\
MH10              & $(4,4,4)$     & 64           & 192          & 199            & 156             & 193             & 199                 & 199                      & 199                      \\
MH11              & $(4,4,5)$     & 80           & 240          & 258            & 203             & 249             & 258                 & 258                      & 258                      \\
MH12              & $(4,5,5)$     & 100          & 300          & 343            & 260             & 324             & 342                 & 342                      & 342                      \\
MH13              & $(4,5,6)$     & 120          & 360          & 400            & 312             & 384             & 398                 & 400                      & 400                      \\
MH14              & $(5,5,5)$     & 125          & 375          & 391            & 304             & 376             & 391                 & 391                      & 391                      \\
MH15              & $(5,5,6)$     & 150          & 450          & 528            & 422             & 513             & 528                 & 528                      & 528                      \\
MH16              & $(5,6,6)$     & 180          & 540          & 607            & 479             & 586             & 607                 & 607                      & 607                      \\
MH17              & $(5,6,7)$     & 210          & 630          & 698            & 539             & 668             & 696                 & 697                      & 697                      \\
MH18              & $(6,6,6)$     & 216          & 648          & 700            & 561             & 683             & 697                 & 698                      & 699                      \\
MH19              & $(6,6,7)$     & 252          & 756          & 834            & 663             & 808             & 830                 & 832                      & 832                      \\
MH20              & $(6,7,7)$     & 294          & 882          & 994            & 779             & 958             & 990                 & 992                      & 992                      \\
MH21              & $(6,7,8)$     & 336          & 1008         & 1087           & 847             & 1047            & 1079                & 1083                     & 1083                     \\
MH22              & $(7,7,7)$     & 343          & 1029         & 1162           & 907             & 1107            & 1155                & 1158                     & 1159                     \\
MH23              & $(7,7,8)$     & 392          & 1176         & 1246           & 975             & 1201            & 1238                & 1241                     & 1242                     \\
MH24              & $(7,8,8)$     & 448          & 1344         & 1449           & 1135            & 1393            & 1439                & 1442                     & 1442                     \\
MH25              & $(7,8,9)$     & 504          & 1512         & 1645           & 1281            & 1576            & 1626                & 1631                     & 1631                     \\
MH26              & $(8,8,8)$     & 512          & 1536         & 1566           & 1247            & 1530            & 1555                & 1557                     & 1557                     \\
MH27              & $(8,8,9)$     & 576          & 1728         & 1883           & 1485            & 1817            & 1861                & 1866                     & 1867                     \\
MH28              & $(8,9,9)$     & 648          & 1944         & 2075           & 1643            & 2003            & 2057                & 2060                     & 2060                     \\
MH29              & $(8,9,10)$    & 720          & 2160         & 2339           & 1850            & 2259            & 2309                & 2313                     & 2314                     \\
MH30              & $(9,9,9)$     & 729          & 2187         & -              & 1894            & 2329            & 2416                & 2421                     & 2422                     \\
MH31              & $(9,9,10)$    & 810          & 2430         & -              & 2081            & 2535            & 2603                & 2608                     & 2608                     \\
MH32              & $(9,10,10)$   & 900          & 2700         & -              & 2304            & 2817            & 2886                & 2888                     & 2889                     \\ \bottomrule                   
\end{tabular}
\caption{Comparison of different bounds for Manhattan instances. \label{Table:mhbounds}}
\end{table}

\begin{table}[h]
\begin{subtable}[h]{0.45\textwidth}
\centering \scriptsize \setlength{\tabcolsep}{3pt}
\begin{tabular}{@{}cccccc@{}}
\toprule
\textbf{}         & \textbf{}       & \textbf{}       & \textbf{}           & \multicolumn{2}{c}{\textbf{$SDP_{S3}$}}                                                                                                 \\ \cmidrule(l){5-6} 
\textbf{} & \textbf{$MILP$} & \textbf{$LBB1$} & \textbf{$SDP_{S2}$} & \textbf{\begin{tabular}[c]{@{}c@{}}$numCuts$\\ $300$\end{tabular}} & \textbf{\begin{tabular}[c]{@{}c@{}}$numCuts$\\ $500$\end{tabular}} \\
\cmidrule(lr){5-5} \cmidrule(lr){6-6}
\textbf{Instance} & \textbf{time} & \textbf{time} & \textbf{time} & \textbf{time} & \textbf{time} \\ \midrule
MH1               & 0.103           & 0.016           & 0.047               & 0.030                                                              & 0.040                                                              \\
MH2               & 0.309           & 0.011           & 0.916               & 0.839                                                              & 0.882                                                              \\
MH3               & 2.665           & 0.036           & 20.78               & 57.07                                                              & 103.5                                                              \\
MH4               & 3.805           & 0.044           & 39.27               & 221.2                                                              & 414.6                                                              \\
MH5               & 5.406           & 0.057           & 52.62               & 380.2                                                              & 834.5                                                              \\
MH6               & 7.684           & 0.063           & 69.39               & 558.2                                                              & 1478.1                                                             \\
MH7               & 10.29           & 0.079           & 86.52               & 336.9                                                              & 588.5                                                              \\
MH8               & 13.44           & 0.118           & 112.3               & 489.5                                                              & 898.9                                                              \\
MH9               & 49.72           & 0.233           & 343.6               & 2226                                                               & 3210                                                               \\
MH10              & 0.329           & 0.013           & 3.138               & 40.71                                                              & 101.6                                                              \\
MH11              & 0.491           & 0.022           & 4.557               & 119.9                                                              & 251.0                                                              \\
MH12              & 0.722           & 0.033           & 3.847               & 3.717                                                              & 3.628                                                              \\
MH13              & 0.963           & 0.047           & 10.76               & 117.2                                                              & 202.6                                                              \\
MH14              & 1.103           & 0.052           & 18.83               & 133.8                                                              & 267.6                                                              \\
MH15              & 2.305           & 0.073           & 28.47               & 177.8                                                              & 377.5                                                              \\
MH16              & 3.520           & 0.092           & 54.29               & 252.1                                                              & 422.5                                                              \\
MH17              & 5.265           & 0.135           & 78.98               & 291.3                                                              & 429.1                                                              \\
MH18              & 5.548           & 0.138           & 84.35               & 364.7                                                              & 470.6                                                              \\
MH19              & 8.333           & 0.159           & 115.8               & 397.0                                                              & 527.1                                                              \\
MH20              & 12.98           & 0.191           & 162.0               & 512.9                                                              & 829.1                                                              \\
MH21              & 18.76           & 0.234           & 261.3               & 888.5                                                              & 1201                                                               \\
MH22              & 20.29           & 0.234           & 272.8               & 881.1                                                              & 1177                                                               \\
MH23              & 29.29           & 0.290           & 382.2               & 967.7                                                              & 1457                                                               \\
MH24              & 43.01           & 0.338           & 525.9               & 1482                                                               & 1697                                                               \\
MH25              & 61.54           & 0.425           & 670.5               & 1695                                                               & 1880                                                               \\
MH26              & 63.89           & 0.444           & 732.8               & 1726                                                               & 2265                                                               \\
MH27              & 90.82           & 0.512           & 939.7               & 2354                                                               & 2752                                                               \\
MH28              & 132.2           & 0.634           & 1227                & 2869                                                               & 3222                                                               \\
MH29              & 177.0           & 0.772           & 1597                & 3414                                                               & 4264                                                               \\
MH30              & 181.4           & 0.815           & 1600                & 3358                                                               & 3643                                                               \\
MH31              & 249.3           & 0.948           & 2096                & 4242                                                               & 4710                                                               \\
MH32              & 344.2           & 1.128           & 2773                & 5530                                                               & 5851                                                               \\ \bottomrule
\end{tabular}
\caption{Computation times (in seconds) \label{Table:mhtimesa}}
\end{subtable} \hspace{1cm} \begin{subtable}[h]{0.45\textwidth}
\centering \scriptsize \setlength{\tabcolsep}{3pt}\begin{tabular}{@{}cccccc@{}}
\toprule
                  &              & \multicolumn{4}{c}{$SDP_{S3}$}                                                                                                                                \\ \cmidrule(l){3-6} 
                  & $SDP_{S2}$   & \multicolumn{2}{c}{\begin{tabular}[c]{@{}c@{}}$numCuts$\\ $300$\end{tabular}} & \multicolumn{2}{c}{\begin{tabular}[c]{@{}c@{}}$numCuts$\\ $500$\end{tabular}} \\ \cmidrule(lr){3-4} \cmidrule(lr){5-6} 
\textbf{Instance} & \textbf{res} & \textbf{res}                          & \textbf{cuts}                         & \textbf{res}                          & \textbf{cuts}                         \\ \midrule
MH1               & 0.001        & 0.002                                 & 0                                     & \textless{}0.001                                & 0                                     \\
MH2               & \textless{}0.001       & \textless{}0.001                                & 0                                     & 0.006                                 & 0                                     \\
MH3               & \textless{}0.001       & \textless{}0.001                                & 300                                   & \textless{}0.001                                & 500                                   \\
MH4               & 0.013        & 0.013                                 & 600                                   & 0.013                                 & 1000                                  \\
MH5               & 0.017        & 0.011                                 & 1200                                  & 0.011                                 & 2000                                  \\
MH6               & 0.011        & 0.010                                  & 916                                   & 0.01                                  & 1500                                  \\
MH7               & 0.024        & 0.024                                 & 1200                                  & 0.024                                 & 2000                                  \\
MH8               & 0.024        & 0.023                                 & 1200                                  & 0.023                                 & 2000                                  \\
MH9               & 0.029        & 0.035                                 & 1500                                  & 0.034                                 & 2500                                  \\
MH10              & \textless{}0.001       & \textless{}0.001                                & 311                                   & \textless{}0.001                                & 569                                   \\
MH11              & 0.012        & \textless{}0.001                                & 574                                   & \textless{}0.001                                & 969                                   \\
MH12              & 0.002        & 0.004                                 & 0                                     & 0.004                                 & 0                                     \\
MH13              & 0.004        & 0.018                                 & 827                                   & 0.022                                 & 1322                                  \\
MH14              & 0.022        & 0.041                                 & 759                                   & 0.040                                  & 1284                                  \\
MH15              & 0.021        & 0.036                                 & 858                                   & 0.044                                 & 1219                                  \\
MH16              & 0.005        & 0.031                                 & 1072                                  & 0.037                                 & 1842                                  \\
MH17              & 0.009        & 0.031                                 & 1075                                  & 0.035                                 & 1781                                  \\
MH18              & 0.005        & 0.013                                 & 1117                                  & 0.016                                 & 1773                                  \\
MH19              & 0.012        & 0.027                                 & 1032                                  & 0.030                                  & 1707                                  \\
MH20              & 0.013        & 0.026                                 & 1121                                  & 0.030                                  & 1821                                  \\
MH21              & 0.008        & 0.019                                 & 1500                                  & 0.022                                 & 2500                                  \\
MH22              & 0.009        & 0.020                                  & 1500                                  & 0.022                                 & 2500                                  \\
MH23              & 0.010         & 0.020                                  & 1500                                  & 0.019                                 & 2500                                  \\
MH24              & 0.011        & 0.023                                 & 1500                                  & 0.026                                 & 2500                                  \\
MH25              & 0.013        & 0.024                                 & 1500                                  & 0.026                                 & 2500                                  \\
MH26              & 0.030         & 0.043                                 & 1500                                  & 0.047                                 & 2500                                  \\
MH27              & 0.025        & 0.037                                 & 1500                                  & 0.040                                  & 2500                                  \\
MH28              & 0.036        & 0.050                                  & 1500                                  & 0.054                                 & 2500                                  \\
MH29              & 0.034        & 0.044                                 & 1500                                  & 0.048                                 & 2500                                  \\
MH30              & 0.028        & 0.038                                 & 1500                                  & 0.040                                  & 2500                                  \\
MH31              & 0.039        & 0.048                                 & 1500                                  & 0.051                                 & 2500                                  \\
MH32              & 0.031        & 0.045                                 & 1500                                  & 0.047                                 & 2500                                  \\ \bottomrule
\end{tabular}
\caption{Average residuals and number of cuts \label{Table:mhtimesb}}
\end{subtable}
\caption{Comparison of computation times, average residuals and total number of added cuts for Manhattan instances. \label{Table:mhtimes}}
\end{table}

\subsection{Upper bounds and overall results} \label{SubsectionResultsUpper}
We discuss here the results on the upper bounds and provide an overview of the relative gap between best lower and upper bounds for all instances. Table~\ref{Table:upperboundsstats} shows several statistics related to the performance of the hybrid and non-hybrid upper bounds on the full test set. Besides, it provides the average percentage gap between best lower and upper bound per instance type. Table~\ref{Table:upperbounds} provides an overview of the best lower bound, best upper bound and their relative gap for the full set of instances.
For each instance and upper bound type, we compute the upper bound based on the SDP solution resulting from the \textsc{CP-ALM},
and select the best among all to report in  Table~\ref{Table:upperbounds}.
 Since, by construction, the hybrid algorithm always provides the best among all upper bounds, the last column of Table~\ref{Table:upperbounds} indicates which of the non-hybrid heuristics performs best when applied independently. Since all upper bounds can be computed relatively fast, we omit computation times here.

It follows from the tables that our bounds are very strong for the Manhattan and the reload instances,
as the average gap between the best lower and best upper bound using the hybrid heuristic is 1.25\% and 3.90\%, respectively.
For the Erd\H os-R\'enyi instances this gap is much larger.
Namely, it is known that the quality of a lower bound, and thus also of a related upper bound, deteriorate when the size of the problem increases. 
Also, the results indicate that the reload and Manhattan instances are easier to solve than the Erd\H os-R\'enyi instances for all here tested QCCP approaches.
Nevertheless, the average gap on the Erd\H os-R\'enyi instances with up to 1000 arcs is only 10\%.

\begin{table}[b]
\centering \scriptsize
\begin{tabular}{@{}lclc@{}}
\toprule
\multicolumn{4}{c}{\textbf{Statistics on upper bounds and average gaps}}                                                           \\ \midrule
Average gap on all instances                & \multicolumn{1}{c|}{20.02\%} & Percentage of instances $UB_{EB}$ performs best & 36.89\% \\
Average gap on Erd\H os-R\'enyi instances    & \multicolumn{1}{c|}{72.30\%} &
Percentage of instances $UB_{US}$ performs best & 53.28\% \\
Average gap on Manhattan instances          & \multicolumn{1}{c|}{1.25\%}  &
Percentage of instances $UB_{OS}$ performs best  & 68.85\% \\
Average gap on Reload instances             & \multicolumn{1}{c|}{3.90\%}  &
Percentage of instances $UB_{SQ}$ performs best & 77.87\% \\
Average gap on instances with $m \leq 1000$ & \multicolumn{1}{c|}{10.58\%} &
Percentage of instances $UB_{HY}$ strictly lower than others & 25.41\%\\ \bottomrule
\end{tabular}
\caption{Statistics on performance of upper bounds and average gaps on total test set. \label{Table:upperboundsstats}}
\end{table}

When comparing the different upper bounds, we conclude that the SQ-algorithm overall outperforms the other methods, followed by oversampling and undersampling rounding. We however observe a clear relationship with the instance type. For the Erd\H os-R\'enyi instances the SQ-algorithm is convincingly the best heuristic, while for the reload instances the other methods perform reliable as well, probably due to the smaller instance size. For the Manhattan instances, however, the sequential Q-learning heuristic performs well on the smaller instances, but is outperformed by oversampling rounding for larger $m$. This can be explained by the smaller number of iterations of the SQ-algorithm for these type of instances. Since the number of agents in the SQ-algorithm for the Manhattan instances is significantly larger than for the other instance types, we needed to decrease the number of iterations in order to be able to solve the resulting \textsc{SPP} efficiently. Hence, the learning effect of the SQ-algorithm is decreased, while it is in particular that part that makes the algorithm powerful. Nevertheless, we observe for almost all Manhattan instances that the hybrid algorithm obtains a strictly stronger upper bound than $UB_{EB}, UB_{US}$ or $UB_{OS}$. This means that the SQ-algorithm, although not always the favoured heuristic when implemented independently, creates cycles that can lead to an improvement of the best upper bound.

\begin{table}[H]
\centering \scriptsize \setlength{\tabcolsep}{2.5pt}
\begin{tabular}{@{}ccccc@{}}
\toprule
\textbf{Instance} & \textbf{\begin{tabular}[c]{@{}c@{}}Best\\ lower \\ bound\end{tabular}} & \textbf{\begin{tabular}[c]{@{}c@{}}Hybrid \\ upper \\ bound\end{tabular}} & \textbf{\begin{tabular}[c]{@{}c@{}}Gap \\ (\%)\end{tabular}} & \textbf{\begin{tabular}[c]{@{}c@{}}Best non-hybrid \\ heuristic\end{tabular}} \\ \midrule
ER1               & 319                                                                    &                                                                          319&                       0                                      &                                                                              EB, US, OS, SQ \\
RER1              & 293                                                                    &                                                                           293 &  0                                                           &                                                                              EB, US, OS, SQ \\
ER2               & 386                                                                    &                                                                           386&     0                                                        &                                                                              EB, US, OS, SQ \\
RER2              & 391                                                                    &                                                                           391&            0                                                 &                                                                              EB, US, OS, SQ \\
ER3               & 296                                                                    &                                                                           311&       5                                                      &                                                                              OS, SQ \\
RER3              & 266                                                                    &                                                                           288&        8                                                     &                                                                              OS \\
ER4               & 336                                                                    &                                                                           447 &           33                                                  &                                                                              SQ \\
RER4              & 240                                                                    &                                                                            294&            23                                                 &                                                                              SQ \\
ER5               & 324                                                                    &                                                                           404&           25                                                  &                                                                              SQ \\
RER5              & 220                                                                    &                                                                           321&           46                                                  &                                                                              SQ \\
ER6               & 338                                                                    &                                                                           451 &         33                                                    &                                                                              SQ \\
RER6              & 151                                                                    &                                                                           253 &       68                                                      &                                                                              SQ \\
ER7               & 319                                                                    &                                                                           493 &        55                                                     &                                                                              SQ \\
RER7              & 119                                                                    &                                                                           236 &       98                                                      &                                                                              SQ \\
ER8               & 305                                                                    &                                                                           525 &       72                                                      &                                                                              SQ \\
RER8              & 95                                                                     &                                                                           283 &       198                                                      &                                                                              SQ \\
ER9               & 289                                                                    &                                                                           520 &      80                                                       &                                                                              SQ \\
RER9              & 133                                                                    &                                                                           399 &      200                                                       &                                                                              SQ \\
ER10              & 293                                                                    &                                                                           455 &    55                                                         &                                                                              SQ \\
RER10             & 54                                                                     &                                                                           312 &    478                                                         &                                                                              SQ \\
ER11              & 234                                                                    &                                                                           236 &    1                                                         &                                                                              SQ \\
RER11             & 172                                                                    &                                                                           172 &    0                                                         &                                                                              US, OS, SQ \\
ER12              & 173                                                                    &                                                                           187 &   8                                                          &                                                                              SQ \\
RER12             & 87                                                                     &                                                                           113 &   30                                                          &                                                                              SQ \\
ER13              & 202                                                                    &                                                                           245 &   21                                                          &                                                                              SQ \\
RER13             & 143                                                                    &                                                                           169 &   18                                                          &                                                                              SQ \\
ER14              & 200                                                                    &                                                                           280 &    40                                                         &                                                                              SQ \\
RER14             & 74                                                                     &                                                                           170 &    130                                                         &                                                                              SQ \\
ER15              & 191                                                                    &                                                                           326 &    71                                                         &                                                                              SQ \\
RER15             & 37                                                                     &                                                                           175 &    373                                                         &                                                                              SQ \\
MH1               & 103                                                                    & 103                                                                       & 0                                                           & EB, US, OS, SQ                                                            \\
MH2               & 418                                                                    & 418                                                                       & 0                                                           & EB, US, OS, SQ                                                            \\
MH3               & 892                                                                    & 892                                                                       & 0                                                           & EB, US, OS, SQ                                                                \\
MH4               & 1030                                                                   & 1030                                                                      & 0                                                           & EB, US, OS, SQ                                                            \\
MH5               & 1226                                                                   & 1226                                                                      & 0                                                           & EB, US, OS                                                                    \\
MH6               & 1283                                                                   & 1283                                                                      & 0                                                           & EB, US, OS, SQ                                                                \\
MH7               & 1446                                                                   & 1448                                                                      & 0                                                           & EB, US, OS, SQ                                                                \\
MH8               & 1537                                                                   & 1539                                                                      & 0                                                           & EB                                                                        \\
MH9               & 2568                                                                   & 2572                                                                      & 0                                                           & EB                                                                        \\
MH10              & 199                                                                    & 199                                                                       & 0                                                           & EB, US, OS, SQ                                                                \\
MH11              & 258                                                                    & 258                                                                       & 0                                                           & EB, US, OS, SQ                                                                \\
MH12              & 342                                                                    & 348                                                                       & 2                                                           & US                                                                            \\
MH13              & 400                                                                    & 400                                                                       & 0                                                           & EB, US, OS                                                                \\
MH14              & 391                                                                    & 391                                                                       & 0                                                           & EB, US, OS                                                                \\
MH15              & 528                                                                    & 528                                                                       & 0                                                           & EB, US, OS                                                                \\
MH16              & 607                                                                    & 607                                                                       & 0                                                           & EB, US, OS                                                                \\
MH17              & 697                                                                    & 698                                                                       & 0                                                           & OS                                                                            \\
MH18              & 699                                                                    & 706                                                                       & 1                                                           & OS                                                                            \\
MH19              & 832                                                                    & 839                                                                       & 1                                                           & US, OS                                                                        \\
MH20              & 992                                                                    & 999                                                                       & 1                                                           & OS                                                                            \\
MH21              & 1083                                                                   & 1093                                                                      & 1                                                           & OS                                                                            \\
MH22              & 1159                                                                   & 1171                                                                      & 1                                                           & OS                                                                            \\
MH23              & 1242                                                                   & 1272                                                                      & 2                                                           & OS                                                                            \\
MH24              & 1442                                                                   & 1498                                                                      & 4                                                           & OS                                                                            \\
MH25              & 1631                                                                   & 1702                                                                      & 4                                                           & OS                                                                            \\
MH26              & 1557                                                                   & 1576                                                                      & 1                                                           & OS                                                                            \\
MH27              & 1867                                                                   & 1940                                                                      & 4                                                           & OS                                                                            \\
MH28              & 2060                                                                   & 2141                                                                      & 4                                                           & OS                                                                            \\
MH29              & 2314                                                                   & 2426                                                                      & 6                                                           & OS                                                                            \\
MH30              & 2422                                                                   & 2552                                                                      & 5                                                           & OS                                                                            \\
MH31              & 2608                                                                   & 2775                                                                      & 6                                                           & OS                                                                            \\
MH32              & 2889                                                                   & 3077                                                                      & 7                                                           & OS                                                                            \\ \bottomrule
\end{tabular}\hspace{1cm} \begin{tabular}{@{}ccccc@{}}
\toprule
\textbf{Instance} & \textbf{\begin{tabular}[c]{@{}c@{}}Best\\ lower \\ bound\end{tabular}} & \textbf{\begin{tabular}[c]{@{}c@{}}Hybrid \\ upper \\ bound\end{tabular}} & \textbf{\begin{tabular}[c]{@{}c@{}}Gap \\ (\%)\end{tabular}} & \textbf{\begin{tabular}[c]{@{}c@{}}Best non-hybrid \\ heuristic\end{tabular}} \\ \midrule
REL1              & 4                                                                      & 4                                                                         & 0                                                           & EB, US, OS, SQ                                                            \\
REL2              & 9                                                                      & 9                                                                         & 0                                                           & EB, US, OS, SQ                                                            \\
REL3              & 5                                                                      & 5                                                                         & 0                                                           & EB, US, OS, SQ                                                                \\
REL4              & 12                                                                     & 12                                                                        & 0                                                           & US                                                                            \\
REL5              & 4                                                                      & 4                                                                         & 0                                                           & EB, US, OS, SQ                                                                \\
REL6              & 14                                                                     & 14                                                                        & 0                                                           & US, OS, SQ                                                                    \\
REL7              & 5                                                                      & 5                                                                         & 0                                                           & US, OS, SQ                                                                    \\
REL8              & 11                                                                     & 11                                                                        & 0                                                           & EB, US, OS, SQ                                                                \\
REL9              & 2                                                                      & 2                                                                         & 0                                                           & EB, US, OS, SQ                                                            \\
REL10             & 12                                                                     & 12                                                                        & 0                                                           & EB, US, OS, SQ                                                            \\
REL11             & 3                                                                      & 3                                                                         & 0                                                           & EB, US, OS, SQ                                                                \\
REL12             & 9                                                                      & 9                                                                         & 0                                                           & EB, US, OS, SQ                                                                \\
REL13             & 4                                                                      & 4                                                                         & 0                                                           & US, OS, SQ                                                                \\
REL14             & 11                                                                     & 11                                                                        & 0                                                           & EB, US, OS, SQ                                                                \\
REL15             & 4                                                                      & 4                                                                         & 0                                                           & EB, US, OS, SQ                                                                \\
REL16             & 11                                                                     & 11                                                                        & 0                                                           & EB, US, OS, SQ                                                            \\
REL17             & 4                                                                      & 4                                                                         & 0                                                           & EB, US, OS, SQ                                                            \\
REL18             & 10                                                                     & 10                                                                        & 0                                                           & EB, US, OS, SQ                                                                \\
REL19             & 5                                                                      & 5                                                                         & 0                                                           & EB, US, OS, SQ                                                                \\
REL20             & 11                                                                     & 11                                                                        & 0                                                           & EB, US, OS, SQ                                                                \\
REL21             & 5                                                                      & 5                                                                         & 0                                                           & US, OS, SQ                                                                    \\
REL22             & 12                                                                     & 12                                                                        & 0                                                           & SQ                                                                            \\
REL23             & 4                                                                      & 4                                                                         & 0                                                           & US, OS, SQ                                                                \\
REL24             & 11                                                                     & 11                                                                        & 0                                                           & US, OS, SQ                                                                    \\
REL25             & 5                                                                      & 5                                                                         & 0                                                           & EB, US, OS, SQ                                                                \\
REL26             & 9                                                                      & 10                                                                        & 11                                                          & SQ                                                                            \\
REL27             & 4                                                                      & 4                                                                         & 0                                                           & US, OS, SQ                                                                    \\
REL28             & 9                                                                      & 9                                                                         & 0                                                           & US, OS, SQ                                                                \\
REL29             & 6                                                                      & 6                                                                         & 0                                                           & US, OS, SQ                                                                    \\
REL30             & 10                                                                     & 10                                                                        & 0                                                           & US, OS, SQ                                                                    \\
REL31             & 5                                                                      & 5                                                                         & 0                                                           & US, OS, SQ                                                                    \\
REL32             & 11                                                                     & 11                                                                        & 0                                                           & US, OS, SQ                                                                    \\
REL33             & 4                                                                      & 4                                                                         & 0                                                           & US, OS, SQ                                                                    \\
REL34             & 8                                                                      & 8                                                                         & 0                                                           & OS, SQ                                                                        \\
REL35             & 4                                                                      & 4                                                                         & 0                                                           & EB, US, OS, SQ                                                            \\
REL36             & 8                                                                      & 8                                                                         & 0                                                           & US, OS, SQ                                                                \\
REL37             & 6                                                                      & 6                                                                         & 0                                                           & US, OS, SQ                                                                \\
REL38             & 11                                                                     & 11                                                                        & 0                                                           & EB, US, OS, SQ                                                                \\
REL39             & 3                                                                      & 3                                                                         & 0                                                           & EB, US, OS, SQ                                                            \\
REL40             & 7                                                                      & 7                                                                         & 0                                                           & EB, US, OS, SQ                                                            \\
REL41             & 4                                                                      & 4                                                                         & 0                                                           & EB, US, OS, SQ                                                                \\
REL42             & 7                                                                      & 11                                                                        & 57                                                          & SQ                                                                            \\
REL43             & 3                                                                      & 3                                                                         & 0                                                           & EB, US, OS, SQ                                                                \\
REL44             & 7                                                                      & 10                                                                        & 43                                                          & SQ                                                                            \\
REL45             & 3                                                                      & 3                                                                         & 0                                                           & SQ                                                                            \\
REL46             & 6                                                                      & 6                                                                         & 0                                                           & EB, US, OS, SQ                                                            \\
REL47             & 3                                                                      & 3                                                                         & 0                                                           & SQ                                                                            \\
REL48             & 5                                                                      & 5                                                                         & 0                                                           & OS, SQ                                                                    \\
REL49             & 4                                                                      & 4                                                                         & 0                                                           & US, OS, SQ                                                                    \\
REL50             & 8                                                                      & 8                                                                         & 0                                                           & OS                                                                            \\
REL51             & 3                                                                      & 3                                                                         & 0                                                           & SQ                                                                            \\
REL52             & 6                                                                      & 9                                                                         & 50                                                          & SQ                                                                            \\
REL53             & 4                                                                      & 4                                                                         & 0                                                           & EB, US, OS, SQ                                                            \\
REL54             & 9                                                                      & 11                                                                        & 22                                                          & SQ                                                                            \\
REL55             & 3                                                                      & 3                                                                         & 0                                                           & EB, US, OS, SQ                                                                \\
REL56             & 8                                                                      & 11                                                                        & 38                                                          & SQ                                                                            \\
REL57             & 4                                                                      & 4                                                                         & 0                                                           & US, OS, SQ                                                                    \\
REL58             & 7                                                                      & 7                                                                         & 0                                                           & US, OS, SQ                                                                    \\
REL59             & 3                                                                      & 3                                                                         & 0                                                           & SQ                                                                        \\
REL60             & 8                                                                      & 9                                                                         & 13                                                          & OS, SQ                                                                        \\ \bottomrule \\ \\ 
\end{tabular}
\caption{Overview of best lower bounds, best hybrid and non-hybrid upper bounds and their relative gaps for all instances. \label{Table:upperbounds}}
\end{table}

\section{Conclusions}
This paper provides an in-depth theoretical as well as practical study on the \textsc{QCCP}. 
We provide various lower and upper bounds for the \textsc{QCCP} based on semidefinite programming.
Moreover, we introduce efficient methods to compute these bounds and give an analysis of their theoretical properties.

We first introduce three SDP relaxations with increasing complexity.
Our strongest SDP relaxation, $(SDP_{S3})$, see~\eqref{SDPS3},
  contains a large number of constraints which  makes it a strong but very difficult to solve relaxation.
Since there are no efficient solvers that can solve SDP relaxations including BQP cuts, we derive  a cutting plane augmented Lagrangian method that is designed to solve such relaxations, 
see  Algorithm~\ref{AlgCuttingPlaneALM}. 
Our algorithm starts from the Peaceman--Rachford splitting method where the involved polyhedral set is strengthened throughout the algorithm by adding valid cuts.
To project onto the polyhedral set, we implement a semi-parallelized version of Dykstra's cyclic projection algorithm, see Section~\ref{SubsectionDykstra} for details.
Parallelization here refers to  clustering the set of BQP inequalities into subsets of nonoverlapping cuts.
Besides the parallelization step we implement several other efficiency improving steps that contribute to the effectiveness  of the \textsc{CP-ALM}.
Our algorithm also benefits from warm starts when adding new cuts.
The \textsc{CP-ALM} is able to compute lower bounds for large instances up to 2700 arcs, thus having a semidefinite constraint of order 2700, including 7,290,000 nonnegative constrained variables, and up to 2500 BQP cuts within two hours.

We also introduce several upper bounding approaches that exploit matrices resulting from the \textsc{CP-ALM}, including 
 randomized undersampling (see Algorithm~\ref {AlgUR}) and  randomized oversampling (see Algorithm~\ref{AlgOR}).
Additionally, we propose an SDP-based distributed reinforcement learning algorithm, which we call sequential Q-learning, see Algorithm~\ref{AlgSQR}.
Starting from the SDP solution matrix, we let artificial agents learn how to find near-optimal cycles in the graph.
We are not aware of other approaches in the literature that combine SDP and reinforcement learning.

We perform  extensive numerical  experiments.
Our numerical results show that both semidefinite programming bounds $SDP_{S2}$ and $SDP_{S3}$ outperform the current strongest \textsc{QCCP} bounds.
The results show that $SDP_{S3}$ bounds are significantly better than $SDP_{S2}$ bounds, provided that there exist violated triangle inequalities.
Among the upper bounding approaches, our sequential Q-learning algorithm is the winner.
The average gap between the best lower and upper bounds on test instances with up to 1000 arcs is about 10\%, while for certain instances this average gap can be as low as 1.25\%, see Table \ref{Table:upperboundsstats} and Table \ref{Table:upperbounds} for details. \\ \\
Several of the newly introduced approaches can be extended to other problems. The various components of the \textsc{CP-ALM} are rather general, which make it possible to adopt it for solving other SDP models that involve large number of cutting planes, such as for the quadratic traveling salesman problem.
Our SDP-based reinforcement learning approach  can also be extended for finding feasible solutions of other optimization problems.
We expect that the sequential Q-learning approach should perform well for  problems on complete graphs.
Finally, an interested reader can download our code for computing a basis for the flow space of the bipartite representation of a directed graph.\\\\
\textbf{Acknowledgements.} We would like to thank Christoph Helmberg for an insightful discussion about the graph theoretical interpretation of the facial reduction.
We would also like to thank Dion Gijswijt for carefully reading the manuscript and giving his valuable feedback.
Moreover, we thank Borzou Rostami for sharing the reload instances with us. Finally, we thank two anonymous referees for improving an earlier version of this work.

\newpage
\thispagestyle{empty}
\begin{appendices}

\section{Proof of Lemma \ref{LemmaProjectionBQP}}
\label{AppendixProofLemma}
\begin{proof} [\unskip\nopunct]
$\mathcal{P}_{\mathcal{H}_{efg}}(M)$ equals the solution of the following convex optimization problem:
\begin{align*}
\min_{\hat{M} \in \mathcal{S}^{m+1}} \left\{ \Vert \hat{M} - M \Vert_F^2  \, : \, \, \hat{M} \in \mathcal{H}_{efg} \right\}.
\end{align*}
Since $\hat{M}_{st} = M_{st}$ for all entries $(s,t)$ that are not involved in the constraints, this optimization problem can be rewritten as:
\begin{align*}
\min_{\delta, \theta, \mu, \pi} \quad & 2 (\delta - M_{ef})^2 + 2 ( \theta - M_{eg})^2 + 2(\mu - M_{fg})^2 + (\pi - M_{ee})^2 + 2 (\pi - M_{0e})^2 \\
\text{s.t.} \quad & \delta + \theta \leq \pi + \mu.
\end{align*}
The explicit expression of $\hat{M}$ follows from the KKT-conditions of the problem above. Let $\lambda \geq 0$ be the Lagrange multiplier of the inequality $\delta + \theta \leq \pi + \mu$. Then, the KKT conditions lead to the following system:
\begin{align*}
\begin{cases} 4 (\delta - M_{ef}) + \lambda = 0 \\
4(\theta - M_{eg}) + \lambda = 0 \\
4(\mu - M_{fg}) - \lambda = 0 \\
2(\pi - M_{ee}) + 4(\pi - M_{0e}) - \lambda = 0 \\
\lambda \geq 0 \\
\lambda(\delta + \theta - \mu - \pi) = 0 \\
\delta + \theta \leq \pi + \mu.
\end{cases}
\end{align*}
Complementarity implies that either $\mu = \delta + \theta - \pi$ or $\lambda = 0$. The latter case leads to the KKT-point $(\delta, \theta, \mu, \pi) = (M_{ef}, M_{eg}, M_{fg}, \frac{M_{ee} + 2M_{0e}}{3})$, which is optimal if and only if $M_{ef} + M_{eg} \leq \frac{M_{ee} + 2M_{0e}}{3} + M_{fg}$. If this inequality does not hold, the substitution $\mu = \delta + \theta - \pi$ leads to the system
\begin{align*}
\begin{cases} 4 (\delta - M_{ef}) + \lambda = 0 \\
4(\theta - M_{eg}) + \lambda = 0 \\
4(\delta + \theta - \pi - M_{fg}) - \lambda = 0 \\
6\pi - 2M_{ee} - 4M_{0e} - \lambda = 0
\end{cases} & \quad \Leftrightarrow & \begin{cases} \delta = -\frac{1}{4}\lambda + M_{ef} \\
\theta = -\frac{1}{4}\lambda + M_{eg} \\
4(\delta + \theta - \pi - M_{fg}) - \lambda = 0 \\
\lambda = 6\pi - 2M_{ee} - 4M_{0e}
\end{cases}
\end{align*}
Substitution into the third equation yields
\begin{align*}
4\pi & = 4\left( - \frac{1}{4}\lambda + M_{ef} - \frac{1}{4}\lambda + M_{eg}  - M_{fg}\right) -\lambda \\
\Leftrightarrow \qquad 4\pi & = -3\lambda + 4M_{ef} + 4M_{eg} - 4M_{fg} \\
\Leftrightarrow \qquad 4\pi & = -3\left(6\pi - 2M_{ee} - 4M_{0e}\right)  + 4M_{ef} + 4M_{eg} - 4M_{fg} \\
\Leftrightarrow \qquad \,\,\, \pi & = \frac{3}{11}M_{ee} + \frac{6}{11}M_{0e} - \frac{2}{11}M_{fg} + \frac{2}{11}M_{ef} + \frac{2}{11}M_{eg}.
\end{align*}
By substitution of this expression into the remaining three equations, we obtain:
\begin{align*}
\lambda & = -\frac{4}{11}M_{ee} - \frac{8}{11}M_{0e} - \frac{12}{11}M_{fg} + \frac{12}{11}M_{ef} + \frac{12}{11}M_{eg}, \\
\delta & = \frac{1}{11}M_{ee} + \frac{2}{11}M_{0e} + \frac{3}{11}M_{fg} + \frac{8}{11}M_{ef} - \frac{3}{11}M_{eg}, \\
\theta & = \frac{1}{11}M_{ee} + \frac{2}{11} M_{0e} + \frac{3}{11}M_{fg} - \frac{3}{11}M_{ef} + \frac{8}{11}M_{eg}, \\
\mu & = -\frac{1}{11}M_{ee} - \frac{2}{11}M_{0e} + \frac{8}{11}M_{fg} + \frac{3}{11}M_{ef} + \frac{3}{11}M_{eg}.
\end{align*}
By setting $\hat{M}_{st} = \delta$ for $(s,t) \in \{(e,f),(f,e)\}$, $\hat{M}_{st} = \theta$ for $(s,t) \in \{(e,g),(g,e)\}$, $\hat{M}_{st} = \mu$ for $(s,t) \in \{(f,g),(g,f)\}$ and $\hat{M}_{st} = \pi$ for $(s,t) \in \{(0,e),(e,0),(e,e)\}$, the claim follows.
\end{proof}

\section{Dykstra's parallel projection algorithm} \label{AppendixParallelDykstra}
\thispagestyle{empty}
In Section \ref{SubsectionDykstra} Dykstra's cyclic algorithm is presented to iteratively project onto the polyhedra induced by the BQP cuts. Instead of projecting on each polyhedron one after another, it is also possible to project on all polyhedra simultaneously. This method is refered to as parallel Dykstra.
Gaffke and Mathar \cite{GaffkeMathar} were the first who proposed this fully simultaneous method.
The convergence of this algorithm in Euclidean spaces was shown by Iusem and De Pierro \cite{IusemDePierro} using a construction by Pierra \cite{Pierra}. The approach was later generalized to Hilbert spaces, see e.g., \cite{BauschkeBorwein}.\\

\noindent The idea of the parallel Dykstra algorithm is to project onto each set simultaneously and monitor the sequence of weighted averages of these projections. We present here a tailor-made version of this approach by giving each triangle inequality an equal weight. Let $\theta \in (0,1)$. At the start, we set $X^0_\mathcal{Y} = X^0_{efg} = M$ for all $(e,f,g) \in \mathcal{T}$, $R^0_\mathcal{Y} = \bold{0}$ and $R^0_{efg} = \bold{0}$. Moreover, we set $\bar{X}^0 = M$. Now, for each $k \geq 1$ we iterate:
\begin{align} \label{AlgParallelDykstra} \tag{ParDyk}
\begin{aligned}
\begin{aligned}
X^k_\mathcal{Y} & := \mathcal{P}_\mathcal{Y} \left(\bar{X}^{k-1} + R_\mathcal{Y}^{k-1}\right) \\
R_{\mathcal{Y}}^k & := \bar{X}^{k-1} + R_\mathcal{Y}^{k-1} - X^k_{\mathcal{Y}} \end{aligned} \qquad \qquad  \quad \, & \\
\left.
\begin{aligned}
X^k_{efg} & := \mathcal{P}_{\mathcal{H}_{efg}} \left( \bar{X}^{k-1} + R_{efg}^{k-1} \right) \\
R^k_{efg} & := \bar{X}^{k-1} + R_{efg}^{k-1} - X^k_{efg}
\end{aligned} \qquad \quad \right\} & \quad \text{for all } (e,f,g) \in \mathcal{T} \\
\bar{X}^k := \theta X^k_\mathcal{Y} + (1 - \theta)\frac{1}{|\mathcal{T}|} \sum_{(e,f,g) \in \mathcal{T}}X^k_{efg}
\end{aligned}
\end{align}
Note that the projections in (\ref{AlgParallelDykstra}) can be performed simultaneously, as each projection solely uses information resulting from the previous iterate. Under some regularity conditions, the sequence $\{\bar{X}^k \}_{k\geq 1}$ in (\ref{AlgParallelDykstra}) converges strongly to the solution of the best approximation problem, see \cite{IusemDePierro, BauschkeBorwein}. One of the sufficient conditions for convergence is that $\mathcal{Y}_{\mathcal{T}} \neq \emptyset$,
which always holds in our setting.\\\\
Based on a construction by Pierra \cite{Pierra}, it follows that the algorithm (\ref{AlgParallelDykstra}) is equivalent to the cyclic Dykstra algorithm performed to the following two convex sets in the higher dimensional space $(\mathcal{S}^{m+1})^{|\mathcal{T}| + 1} := \mathcal{S}^{m+1} \times ... \times \mathcal{S}^{m+1}$:
\begin{align*}
\bold{S}_1 := \mathcal{Y} \,\, \times \prod_{(e,f,g) \in \mathcal{T}} \mathcal{H}_{efg}  \quad \text{and} \quad \bold{S}_2 := \left\{(X,X, ..., X) \in (\mathcal{S}^{m+1})^{|\mathcal{T}| + 1} \, : \, \, X \in \mathcal{S}^{m+1} \right\},
\end{align*}
using the inner product $\langle \cdot , \cdot \rangle_\theta$ defined as
\begin{align*}
\left\langle (X_0, X_1, ..., X_{|\mathcal{T}|}), (Y_0, Y_1, ..., Y_{|\mathcal{T}|}) \right\rangle_\theta := \theta \, \langle X_0, Y_0 \rangle + (1 - \theta) \, \frac{1}{|\mathcal{T}|} \, \sum_{i = 1}^{|\mathcal{T}|}\langle X_i, Y_i \rangle .
\end{align*}
Preliminary experiments show that the convergence of (\ref{AlgParallelDykstra}) in general takes more iterations than the convergence of (\ref{AlgCyclicDykstra}), where we use the semi-parallel implementation of the latter. This is what one might expect, since in the cyclic version each iterate directly builds on the output of the previous iterates. However, since the projections can be performed simultaneously, the total computation time can still be smaller when implemented on parallel machines. Table~\ref{Tab:cyclicvsparallel} shows a comparison of both methods within the \textsc{CP-ALM} on a test set of Erd\H os-R\'enyi instances implemented on non-parallel machines. Results are presented for different values of $\theta$. In all cases the lower bound obtained by the \textsc{CP-ALM} using  (\ref{AlgParallelDykstra}) in the subproblem at the moment the iteration limit is reached is weaker than the lower bound obtained from using (\ref{AlgCyclicDykstra}) in the subproblem. Moreover, since the parallel version takes more iterations to converge, the computation times are significantly larger. We conclude that the use of (\ref{AlgCyclicDykstra}) is favoured above the use of (\ref{AlgParallelDykstra}) within the \textsc{CP-ALM} in both quality and computation time. For that reason, we only use (\ref{AlgCyclicDykstra}) in the numerical experiments of Section \ref{SectionNumerics}.

\begin{table}[H]
\scriptsize
\centering
\begin{tabular}{@{}ccccccccccccc@{}}
\toprule
    &     &     & \textbf{}        & \textbf{}        & \multicolumn{2}{c}{\multirow{2}{*}{\textbf{\begin{tabular}[c]{@{}c@{}}CP-ALM \\ using cyclic\\ Dykstra\end{tabular}}}} & \multicolumn{6}{c}{\textbf{CP-ALM using parallel Dykstra}}                                                                                \\ \cmidrule(l){8-13}
    &     &     & \multicolumn{2}{c}{\textbf{$PRSM$}} & \multicolumn{2}{c}{}                                                                                                          & \multicolumn{2}{c}{\textbf{$\theta = 0.5$}} & \multicolumn{2}{c}{\textbf{$\theta = 0.85$}} & \multicolumn{2}{c}{\textbf{$\theta = 0.95$}} \\ \cmidrule(l){4-5} \cmidrule(l){6-7} \cmidrule(l){8-9}  \cmidrule(l){10-11} \cmidrule(l){12-13}
$p$ & $n$ & $m$ & \textbf{value}   & \textbf{times}   & \textbf{value}                                                & \textbf{times}                                                & \textbf{value}       & \textbf{times}       & \textbf{value}        & \textbf{times}       & \textbf{value}        & \textbf{times}       \\ \midrule
0.3 & 20  & 119 & 319              & 0.331            & 319                                                           & 0.384                                                         & 319                  & 0.415                & 319                   & 0.378                & 319                   & 0.389                \\
    & 25  & 177 & 386              & 1.822            & 386                                                           & 5.437                                                         & 386                  & 26.61                & 386                   & 24.35                & 386                   & 24.01                \\
    & 30  & 280 & 333              & 20.75            & 339                                                           & 96.27                                                         & 335                  & 7426                 & 333                   & 7036                 & 333                   & 1433                 \\
0.5 & 20  & 195 & 227              & 10.15            & 234                                                           & 92.89                                                         & 229                  & 4203                 & 227                   & 2923                 & 227                   & 733.7                \\
    & 25  & 327 & 169              & 35.68            & 173                                                           & 92.13                                                         & 170                  & 6640                 & 169                   & 5623                 & 169                   & 1852                 \\
    & 30  & 442 & 198              & 91.71            & 202                                                           & 130.8                                                         & 199                  & 12437                & 198                   & 10815                & 198                   & 3677                 \\ \bottomrule
\end{tabular}
\caption{Performance of \textsc{CP-ALM} using cyclic and parallel Dykstra on a test set of 6 Erd\H os-R\' enyi instances with $maxIter = 1000$, $maxTotalIter = 5000$, $numCuts = 150$ and all other parameters the same as given in Section \ref{SectionNumerics}. \label{Tab:cyclicvsparallel}}
\end{table}
\thispagestyle{empty}
\noindent In order to reduce the number of iterations to converge, we can perform a preprocessing step before the $Y$-subproblem is solved using (\ref{AlgParallelDykstra}). Suppose this subproblem involves the projection of a matrix $M$ onto $\mathcal{Y}_\mathcal{T}$. Since this projection is done iteratively, the length of the sequence before convergence depends on the initial distance between $M$ and $\mathcal{Y}_\mathcal{T}$. This distance can be shortened using a simple preprocessing step. This step involves the projection onto all affine constraints of $\mathcal{Y}_\mathcal{T}$. We define:
\begin{align*}
\mathcal{Y}^{\aff} := \left\{ Y \in \mathcal{S}^{m+1} \, : \, \, \begin{aligned}  Y_{00} = 1, \,\, \diag(Y) = Y\bold{e}_0, \,\, \tr(Y) = n+1, \,\, Y_{ef} = 0 \,\,\,\, ( \forall (e,f) \in \mathcal{Z} ) \end{aligned} \right\}.
\end{align*}
Since $\mathcal{Y}^{\aff}$ is an affine subspace, the projection $\mathcal{P}_{\mathcal{Y}^\aff}(\cdot)$ onto $\mathcal{Y}^{\aff}$ can be found explicitly. Now, instead of projecting $M$ onto $\mathcal{Y}_\mathcal{T}$, we can equivalently project the `closer' matrix $\mathcal{P}_{\mathcal{Y}^\aff}(M)$ onto $\mathcal{Y}_\mathcal{T}$, as shown by the following lemma.
\begin{lemma} $\mathcal{P}_{\mathcal{Y}_\mathcal{T}}(M) = \mathcal{P}_{\mathcal{Y}_\mathcal{T}}\left(\mathcal{P}_{\mathcal{Y}^\aff}(M)\right)$.
\end{lemma}
\begin{proof}
Let $\bar{M} := \mathcal{P}_{\mathcal{Y}_\mathcal{T}}(M)$ and $\hat{M} := \mathcal{P}_{\mathcal{Y}^\aff}(M)$. We have to show that $\mathcal{P}_{\mathcal{Y}_\mathcal{T}}(\hat{M}) = \bar{M}$. Using the Kolmogorov conditions, the projection of $\hat{M}$ onto $\mathcal{Y}_\mathcal{T}$ is
the unique solution s.t.:
\begin{align*}
\mathcal{P}_{\mathcal{Y}_\mathcal{T}}(\hat{M}) \in \mathcal{Y}_ \mathcal{T} \qquad \text{and} \qquad \langle Y - \mathcal{P}_{\mathcal{Y}_\mathcal{T}}(\hat{M}), \hat{M} - \mathcal{P}_{\mathcal{Y}_\mathcal{T}}(\hat{M}) \rangle \leq 0 \quad \text{for all} \quad Y \in \mathcal{Y}_\mathcal{T}.
\end{align*}
Clearly, $\bar{M}$ satisfies the first condition. Moreover,
\begin{align*}
\langle Y - \bar{M}, \hat{M} - \bar{M} \rangle & =
\underbrace{\langle Y - \bar{M}, M - \bar{M} \rangle}_{\leq 0} + \underbrace{\langle Y - \bar{M}, \hat{M} - M \rangle}_{= 0} \leq 0,
\end{align*}
for all $Y \in \mathcal{Y}_{\mathcal{T}}$. Here $\langle Y - \bar{M}, M - \bar{M} \rangle \leq 0$ follows from the Kolmogorov conditions for the projection of $M$ onto $\mathcal{Y}_\mathcal{T}$ and the equality $\langle Y - \bar{M}, \hat{M} - M \rangle = 0$ follows from the fact that $Y, \bar{M} \in \mathcal{Y}^\aff$ and $\hat{M} - M$ is orthogonal to the affine space $\mathcal{Y}^\aff$. We conclude that $\bar{M} = \mathcal{P}_{\mathcal{Y}_\mathcal{T}}(\hat{M})$.
\end{proof}
\noindent Observe that the projection onto the unconstrained simplex $\bar{\Delta}(a) := \{ x \in \mathbb{R}^m \, : \, \, \bold{1}^\top x = a\}$ is given by $\mathcal{P}_{\bar{\Delta}(a)}(x) = x - \frac{\bold{1}^\top x - a}{\bold{1}^\top\bold{1}}\bold{1}$.
Thus the projection of $M$ onto $\mathcal{Y}^\aff$ is explicitly given by:
\begin{align*}
\mathcal{P}_{\mathcal{Y}^\aff}(M) = E_{00} + T_{\rm inner}(M) + T_\arrow^*\Big( 3 \cdot \mathcal{P}_{\bar{\Delta}(n)}\big( T_\arrow(M) \big) \Big).
\end{align*}
Solving the $Y$-subproblem is now equivalent to performing the projection onto $\mathcal{Y}^{\aff}$ once and apply (\ref{AlgParallelDykstra}) to project $\mathcal{P}_{\mathcal{Y}^\aff}(M)$ onto $\mathcal{Y}_{\mathcal{T}}$. Further experiments show that this step indeed reduces the number of iterations, but this reduction is not enough to exceed the performance of (\ref{AlgCyclicDykstra}).

\section{Proof of Lemma \ref{LemmaOversampling}} \label{AppendixProofOversampling}
\thispagestyle{empty}
\begin{proof} [\unskip\nopunct]
Observe that a pair of successive arcs $(e,f) \in \delta^-(i) \times \delta^+(i)$ can be added to $H$ either by sampling $e$ and $f$ simultaneously in step 5 of Algorithm \ref{AlgOR} or since both arcs are added to $H$ in combination with some other arc. In the former case, we say that the pair $(e,f)$ is drawn around $i$. For all $i \in N$, $(e,f) \in \delta^-(i) \times \delta^+(i)$ and $k \geq 0$, let $Y^k_{i,(e,f)}$ denote the following random variable:
\begin{align*}
Y^k_{i, (e,f)} := \begin{cases} 0 & \text{if $(e,f)$ is not drawn around $i$ during the first $k$ iterations,} \\
1 & \text{otherwise.}
\end{cases}
\end{align*}
Observe that $Y^k_{i,(e,f)}$ is independent over $i$, as step 5 is performed independently over $N$. Since the probability that a pair $(e,f)$ is added to $H$ in a single iteration equals $r_e \cdot r_f$, we have
\begin{align*}
\Pr\left(Y^k_{i,(e,f)} = 1\right) = 1 - \Pr\left(Y^k_{i,(e,f)} = 0\right) = 1 - \left(1 - r_e \cdot r_f \right)^k.
\end{align*}
Since $r \in \Conv(P)$, there must be at least one cycle cover, say $\bar{x} \in P$, that has full support in $r$. We define the functions $p: N \rightarrow A$ and $s: N \rightarrow A$ that map each node $i$ to its predecessor and successor in $\bar{x}$, respectively. We show that the probability that the support of $\bar{x}$ is in $H$ converges to 1 if $k$ increases. Since the probability that a pair of successive arcs $(e,f) \in \delta^-(i) \times \delta^+(i)$ is present in $H$ after $k$ iterations is at least $\Pr \left(Y^k_{i,(e,f)} = 1\right)$, we have:
\begin{align*}
\Pr \left( \supp(\bar{x}) \subseteq H \text{ after $k$ iterations} \right) & \geq \Pr \left( \prod_{i \in N}Y^k_{i,(p(i),s(i))} = 1 \right)  = \prod_{i \in N} \left( 1 - \left( 1 - r_{p(i)} \cdot r_{s(i)} \right)^k\right).
\end{align*}
Since there exists an $\xi > 0$ such that $r_{p(i)} \cdot r_{s(i)} > \xi$ for all $i \in N$, we have
\begin{align*}
\Pr \left( \supp(\bar{x}) \subseteq H \text{ after $k$ iterations} \right) & \geq \left( 1 - (1 - \xi)^k \right)^n.
\end{align*}
Now for any $q < 1$, take $k^* = \left\lceil \frac{\log (1 - \sqrt[n]{q})}{\log (1 - \xi)} \right\rceil$. Then $\Pr( \supp(\bar{x}) \subseteq H$ after $k^*$ iterations$) \geq q$. Thus, the cycle cover $\bar{x}$ is included in $H$ after a finite number of iterations with arbitrary high probability.
\end{proof}
\thispagestyle{empty}

\end{appendices}

\end{document}